\documentclass[a4paper]{amsart}
\usepackage{a4}
\usepackage{amsmath}
\usepackage{enumerate}
\usepackage{amssymb}
\usepackage{mathrsfs}
\usepackage[applemac]{inputenc}
\usepackage{graphicx}
\usepackage{color}
\usepackage{algorithm2e}
\usepackage{subfig}
\usepackage{multirow}
\usepackage{array}
\usepackage[
pdfborder={0 0 0},
colorlinks=true,
linkcolor=blue,
citecolor=red,
pagebackref=true,
]{hyperref}

\def\a{a}

\def\t{\tau}

\newcommand{\ydos}{\mathbf{y_{2}}}

\newcommand{\M}{{\mathbb M}}

\newcommand{\EE}{{\mathbb E}}

\newcommand{\ds}{\displaystyle}
\newcommand{\supp}{\mathrm{supp}\;}

\newtheorem{remark}{\textbf{Remark}}[section]

\newtheorem{lemma}{\textbf{Lemma}}[section]
\newtheorem{theorem}{\textbf{Theorem}}[section]

\newtheorem{proposition}{\textbf{Proposition}}[section]
\newtheorem{problem}{\textbf{Problem}}[section]

\numberwithin{equation}{section}

\title[Approximation of deterministic mean field games with control-affine dynamics]{Approximation of deterministic mean field games with control-affine dynamics} 

\author{Justina Gianatti \and   Francisco J. Silva}

\thanks{CIFASIS-CONICET-UNR, Ocampo y Esmeralda, S2000EZP, Rosario, Argentina (gianatti@cifasis-conicet.gov.ar). }

\thanks{Institut de recherche XLIM-DMI, UMR-CNRS 7252, Facult\'e des Sciences et Techniques, 
Universit\'e de Limoges, 87060 Limoges, France (francisco.silva@unilim.fr)}

\def\dd{{\rm d}}

\newcommand{\ov}[1]{\overline{#1}}

\def\weight(#1,#2){c_{#1,#2}}

 \if {
  
 } \fi

\if {

}{{\caption}}
} \fi

\def\Lb{\bar{L}}

\def\A{\mathcal{A}}
\def\B{\mathcal{B}}

\def\G{\mathcal{G}}

\def\I{\mathcal{I}}

\def\M{\mathcal{M}}

\def\P{\mathcal{P}}

\def\SS{\mathcal{S}}

\def\eps{\varepsilon}

\def\argmin{\mathop{\rm argmin}}

\def\supp{\mathop{\rm supp}}

\def\1B{{\bf  1}}

\newcommand{\NN}{\mathbb{N}}
\newcommand{\ZZ}{\mathbb{Z}}

\newcommand{\QQ}{\mathbb{Q}}

\newcommand{\RR}{\mathbb{R}}

\def\EE{\mathbb{E}}

\def\cR{\mathbb{R}}

\newcommand\be{\begin{equation}}
\newcommand\ee{\end{equation}}
\newcommand\ba{\begin{array}}
\newcommand\ea{\end{array}}
\newcommand{\bean}{\begin{eqnarray*}}
\newcommand{\eean}{\end{eqnarray*}}

\def\ds{\displaystyle}

\begin{document}
\begin{abstract} We consider deterministic mean field games  where the dynamics of a typical agent is non-linear with respect to the state variable and affine with respect to the control variable.  Particular instances of the problem considered here are mean field games with control on the acceleration (see \cite{MR4102464, MR4132067,MR4177552}).  We focus our attention on the approximation of such mean field games by analogous problems in discrete time and finite state space which fall in the framework of \cite{MR2601334}. For these approximations, we show the existence and, under an additional monotonicity assumption, uniqueness of solutions. In our main result, we establish the convergence of equilibria of the discrete mean field games problems towards equilibria of the continuous one. Finally, we provide some numerical results for two MFG problems. In the first one, the dynamics of a typical player is nonlinear with respect to the state  and, in the second one, a typical player controls its acceleration. 
\end{abstract}

\maketitle

{\small
\noindent {\bf AMS-Subject Classification:} 91A16, 	49N80, 35Q89, 65M99, 91A26.  \\[0.5ex]
\noindent {\bf Keywords:} Deterministic mean field games, control-affine dynamics, Lagrangian equilibrium, approximation of equilibria, convergence results, numerical experiences.
}
\section{Introduction}
Mean Field Games (MFGs) systems have been introduced independently by \cite{HMC06} and \cite{MR2269875,MR2271747,MR2295621} in order to describe the asymptotic behaviour of Nash equilibria of non-cooperative symmetric dynamic differential games with a large number of indistinguishable  players, which, individually, have a negligible influence on the game.  We refer the reader to \cite{MR2762362,Cardialaguet10,MR3134900,MR3195844,MR3559742,MR3752669,MR3753660,MR4214773}, and the references therein, for a general overview on MFGs theory including analytic and probabilistic aspects as well as their numerical approximation and applications in crowd motion, economics, and finance.  

In the case of time-dependent MFGs, the articles  \cite{MR2271747,MR2295621} model equilibria of the game with a continuum of agents  by a system of two Partial Differential Equations (PDEs) having the form
\be
\label{mfg_Edp_formulation}
\ba{rcl}
-\partial_t v - \frac{\sigma^{2}}{2}\Delta v+ H(t,x,  \nabla_{x} v) &=& f(x,m(t))  \quad \; \text{for }  (t,x)\in  [0,T] \times \RR^{d}, \\[4pt]
v(T,x)&=& g(x,m(T))  \quad \text{for }  x\in \RR^d, \\[4pt]
\partial_t m - \frac{\sigma^{2}}{2}\Delta m -\text{div}\left(\partial_{p} H(t,x,\nabla_x v)m\right) &=& 0\quad \text{for }  (t,x) \in [0,T] \times \RR^{d},\\[4pt]
m(0) &=& m_0 \quad \text{in }  \RR^{d},
\ea
\ee
where $T>0$, $\sigma \in \RR$, $H\colon[0,T]\times \RR^d \times \RR^d \to \RR$ is given by
\be
\label{definicion_H_introduccion}
H(t,x,p)= \sup_{a\in \RR^r}\left\{- \langle p, b(t,a,x) \rangle - \ell_0(t,a,x)\right\} \quad \text{for all } (t,x,p)\in [0,T]\times \RR^d \times \RR^d,
\ee
with $b\colon[0,T]\times \RR^r \times \RR^d \to \RR^d$, $\ell_0\colon[0,T]\times \RR^r \times \RR^d \to \RR$ ($r\leq d$), and, denoting by $\P(\RR^d)$ the set of probability measures over $\RR^d$, $m_0 \in \P(\RR^d)$, and $f$, $g\colon \RR^d\times \P(\RR^d) \to \RR$.  The first equation in \eqref{mfg_Edp_formulation}, together with the final condition, is a Hamilton-Jacobi-Bellman (HJB) equation which describes the value  function  $v\colon[0,T]\times \RR^d \to \RR$  associated to an optimal control problem solved by a {\it typical player}. Notice that, because of the presence of  $f$ and $g$, the cost function of this problem depends on the distribution $m(t)\in \P(\RR^d)$ of the entire population of agents at each time $t\in[0,T]$. The third equation, together with the initial distribution,  is a Fokker-Planck (FP) equation which describes the fact that $m\colon[0,T]\to \P(\RR^d)$ evolves following  the optimal dynamics of the typical player. 

In this work we focus our attention on the first order case $\sigma=0$, which means that the underlying optimal control problem is  deterministic. When the players directly control their velocity, i.e.  
\be
\label{b_control_velocity}
b(t,a,x)=a \quad \text{for all }(t,a,x)\in[0,T]\times \RR^r\times \RR^d, 
\ee
the existence of solutions to \eqref{mfg_Edp_formulation} has been shown in \cite{MR2295621} under suitable assumptions on the data (see also \cite{Cardialaguet10}). In the case where the dynamics is affine and the agents control the acceleration, i.e. a particular instance of $b(t,a,x)=Ax+ Ba$ for all $(t,a,x)\in[0,T]\times \RR^r\times \RR^d$, for some matrices $A\in \RR^{d\times d}$ and $B\in \RR^{d\times r}$,  the existence of solutions to \eqref{mfg_Edp_formulation} has been shown in \cite{MR4102464,MR4132067} under suitable assumptions on $H$, $f$, and $g$. In the aforementioned references, uniqueness of the solution to \eqref{mfg_Edp_formulation} essentially holds under the so-called Lasry-Lions monotonicity condition on $f$ and $g$ (see \cite[Section 2.3]{MR2295621}).  A relaxed notion of equilibrium in the deterministic case, called {\it Lagrangian equilibrium}, has been recently studied in \cite{CH17,MR3644590,cardaliaguet_meszaros_santambrogio_2018,cannarsa_capuani_2018}. In this context, an equilibrium $\xi^{*}$ is a measure on the path space $C([0,T];\RR^d)$ of $\RR^d$-valued continuous function on $[0,T]$ whose support is contained in the set of trajectories that solve an optimal control problem whose cost function depends on the time marginals of $\xi^*$ (see Section~\ref{sec:definition_lagrangian_equilibria} below). The existence of a Lagrangian equilibrium $\xi^*$ holds under rather general assumptions on the data and, under stronger conditions, a solution $(v,m)$  to \eqref{mfg_Edp_formulation} can be built in terms of $\xi^*$ (see e.g. \cite{CH17,MR3986796,MR4056836,MR4132067,MR4256274}).

The numerical approximation of the MFG system with non-local couplings \eqref{mfg_Edp_formulation} has been an active research subject over the last decade. In the second order case $\sigma\neq 0$, convergent finite difference and Semi-Lagrangian (SL) schemes have been proposed in \cite{AchdouCapuzzo10,MR3097034} and \cite{CarliniSilva18}, respectively. Let us also mention the recent article~\cite{Bertucci_Cecchin_22}, where the convergence of solutions to MFGs with finite state space towards solutions to MFGs with continuous state space, including common noise, is established from the convergence of the solutions to the corresponding master equations.  We refer the reader to \cite{MR4214777}, and the references therein, for a recent account on numerical methods for second order MFGs with local and non-local coupling terms $f$ and $g$. In the first order case $\sigma=0$, with $b$ given by \eqref{b_control_velocity},  the article \cite{MR2928379} studies a time discretization  of \eqref{mfg_Edp_formulation} whose solutions are shown to converge towards a MFG equilibrium as the time step tends to zero. In the same framework, the paper \cite{MR3148086} introduces a fully discrete SL scheme, i.e. the state variable is also discretized, to approximate solutions to \eqref{mfg_Edp_formulation}. In this context, convergence of solutions of the fully discrete scheme towards a MFG equilibrium is established when the state dimension $d$ is equal to one.  Let us also mention the recent contribution \cite{Chowdhury_et_al} which considers a generalization of  \eqref{mfg_Edp_formulation} to the case where the HJB and Fokker-Planck equations involve non-local and fractional diffusion terms. In that reference, an extension of the SL scheme in \cite{MR3148086} is studied and, for degenerate non-local diffusion operators, convergence is also shown in the one-dimensional case. A different strategy has been proposed in \cite{MR4030259}, where the authors provide a full discretization of  deterministic MFGs problems, with $b$ being given by \eqref{b_control_velocity}, which falls into the category of discrete time, finite state MFGs, introduced in \cite{MR2601334}. In order to solve the discrete MFG problem, the authors justify the application of the so-called {\it fictitious play method} from game theory (see e.g. \cite[Chapter 2]{MR1629477}) in the context of discrete MFGs. Moreover, as the discretization parameters tend to zero,  the convergence of equilibria of the discrete MFGs towards an equilibrium of the original MFG in Lagrangian form is established in arbitrary state dimensions. If the Lagrangian equilibrium can be described by the PDE system \eqref{mfg_Edp_formulation}, which is the case of the MFG problem considered in \cite{MR3148086},  then the discretization in \cite{MR4030259} allows to approximate solutions to \eqref{mfg_Edp_formulation} in arbitrary dimensions (see  \cite[Corollary 4.1]{MR4030259}).  Finally, we also refer the reader to the recent papers \cite{MR3962816,MR4322102}, where the authors approximate solutions to  \eqref{mfg_Edp_formulation} by using Fourier and variational techniques. 

In this work, we consider the approximation of deterministic MFGs where the dynamics of a typical agent is non-linear with respect to the state variable and affine with respect to the control variable, i.e. where $b$ takes the form 
\be
\label{control_affine_dynamics}
b(t,a,x)=A(t,x)+B(t,x)a\quad \text{for all }(t,a,x)\in[0,T]\times \RR^r\times \RR^d, 
\ee
for some given functions $A\colon [0,T]\times \RR^d \to \RR^d$ and $B\colon [0,T]\times \RR^d \to \RR^{d\times r}$. In particular, the dynamics that we consider covers the case of MFGs with control on the acceleration studied, at the continuous level, in \cite{MR4102464,MR4132067}.   We take the viewpoint of \cite{MR4030259}, and the resulting discretization is an instance of the discrete MFG problem introduced in \cite{MR2601334}. The key point is to observe that, under our assumptions, if the support of the initial distribution $m_0$ of the agents is compact, then the set of optimal trajectories, and hence the support of any Lagrangian equilibrium, will be contained on a fixed compact set during the whole time interval $[0,T]$. Based on this a priori information on  optimal trajectories, we construct a suitable discretization of the optimal control problem solved by a typical player by considering a finite time-space grid contained on a compact subset of $[0,T]\times \RR^d$, the latter being independent of the discretization steps. Moreover, the proposed approximation takes advantage of the particular form of the dynamics in \eqref{control_affine_dynamics} and allows, at the fully-discrete level, to write the set of admissible feedback controls in terms of the state variables. The resulting scheme to approximate optimal control problems with control-affine dynamics seems to be new and, compared with standard SL schemes (see e.g. \cite{MR3341715}), it reduces the complexity of computing interpolants over grids contained in $\RR^d$ to the computation of interpolants over grids contained in $\RR^{d-r}$. In order to obtain a discretization of the MFG problem, the previous scheme is coupled with a discretization of the FP equation which has a probabilistic interpretation in terms of a  finite state Markov chain representing the evolution of a typical agent in the discrete framework. As usual in game theory, the computation of a discrete equilibrium, i.e. the solution to the fully-discrete scheme, can be written as a fixed point problem in terms of a {\it best response mapping}. Existence of at least one equilibrium is shown in Proposition~\ref{uniqueness_mfg_discreto}, while uniqueness is established in Proposition~\ref{existencia_caso_finito} under the standard Lasry-Lions monotonicity conditions on the couplings $f$ and $g$. As in \cite{MR4030259}, the (unique) equilibrium of the discrete MFG can be computed by using the fictitious play method provided that the aforementioned monotonicity conditions on $f$ and $g$ hold. Given a sequence of equilibria of the discrete MFG problems, corresponding to a sequence of discretization parameters tending to zero, we associate a sequence of probability measures on $C([0,T];\RR^d)$ which is pre-compact for the topology of weak convergence. Our main result in Theorem~\ref{main_result}, valid in arbitrary dimensions, states that every accumulation point of this sequence is a MFG equilibrium in Lagrangian form. In particular, if system \eqref{mfg_Edp_formulation} is well-posed,  then its unique solution is the limit of solutions to our scheme.  Let us emphasize that in the particular case where $b$ 
is given by \eqref{b_control_velocity}, our scheme allows to approximate MFGs with cost functional that are much more general than those in \cite{MR4030259} (see Remark~\ref{comentarios_despues_del_teorema}{\rm(ii)} below). We illustrate our theoretical findings by considering two examples. In the first one, the problem has only one state variable, $b$ is non-linear with respect to the state and affine with respect to the control. In the second example, we consider a problem with two states, position and velocity, and only one control, given by the acceleration. In both examples, the discrete MFGs problems are solved by using the fictitious play method. 

The rest of the paper is organized as follows. The next section introduces some basic definitions, states the MFG problem, fix our main assumptions, and provides existence and uniqueness results of a Lagrangian equilibrium. In Section~\ref{approximation_value_function_general}, we describe the approximation scheme that we propose for the optimal control problem underlying the MFG problem. We first derive some standard but important properties of a semi-discrete scheme and then we present the fully-discrete scheme, which is the cornerstone of the discretization of the MFG problem. Section~\ref{sec:main_result} introduces the discretization of the MFG problem, provides  existence and uniqueness results of an equilibrium for the discrete MFG, and discuss its numerical solution by using the fictitious play method. In Section~\ref{convergence_result} we state and prove our main result, which shows that accumulation points of equilibria of discrete MFGs, as the discretization parameters tend to zero, are equilibria in Lagrangian form of the continuous MFG problem. Two numerical tests that support our theoretical findings are presented in Section~\ref{sec:numerical_result}. Finally, we provide an appendix with the proofs of some technical results from Section~\ref{approximation_value_function_general} and  the proof of the uniqueness result stated in Section~\ref{sec:definition_lagrangian_equilibria}.
\section{Preliminaries}   \label{sec:preliminaries}
We begin by introducing some standard notations. Let $(X,d)$ be a separable metric space, denote by $\B(X)$ the family of Borel subsets of $X$ and by $\P(X)$ the family of Borel probability measures on $X$. Let us define
$$
\P_1(X)=\left\{\mu\in\P(X)\,\Big|\, \int_X d(x, x_0)\dd \mu < \infty\;\;\text{for some }x_0\in X \right\},
$$
which is endowed with the Monge-Kantorovic metric
\be
\label{W1_distance}
d_1(\mu_1,\mu_2)=\inf_{\mu\in\Pi(\mu_1,\mu_2)}\int_{X\times X}d(x,y) \dd \mu(x,y)\quad \text{for all $\mu_1, \,\mu_2\in\P_1(X)$,}
\ee
where $\Pi(\mu_1,\mu_2)$ is the set of Borel probabilities measures on $X\times X$ with first and second marginals equal to $\mu_1$ and $\mu_2$, respectively. A sequence $(\mu_n)_{n\in\NN}\subset \P(X)$ narrowly converges to $\mu\in\P(X)$ if
$$
\lim_{n\to\infty}\int_{X} \varphi(x)\dd\mu_n(x)=\int_{X} \varphi(x)\dd\mu(x)\quad \text{for all $\varphi:X\to \cR$ bounded and continuous}.
$$
By \cite[Proposition 7.1.5]{ambrosio2008gradient}, for any compact set $K\subset X$, we have $\P(K)=\P_1(K)$,  $\P(K)$ is compact, and for any sequence $(\mu_n)_{n\in\NN}\subset \P(K)$ and $\mu\in \P(K)$, $\lim_{n\to\infty}d_1(\mu_n, \mu)=0$ if and only if $(\mu_n)_{n\in\NN}$ narrowly converges to $\mu$.

Finally, given two separable metric spaces $X$ and $Y$, we denote by $C(X;Y)$ the space of continuous functions from $X$ to $Y$ and, for a measure $\mu\in\P(X)$ and a Borel map $\varphi\colon X\to Y$, we denote by $\varphi\sharp\mu\in\P(Y)$ the {\it push-forward }of $\mu$ through $\varphi$, which is defined by
$$
\varphi\sharp\mu(\mathcal{O}) = \mu\left( \varphi^{-1}(\mathcal{O})\right) \quad \text{for all } \mathcal{O}\in \B(Y).
$$
\subsection{Assumptions and Lagrangian MFG equilibria}  
\label{sec:definition_lagrangian_equilibria}
In this section, we fix the assumptions on the data and we introduce the notion of Lagrangian MFG equilibrium. We also state existence and uniqueness results of such equilibrium. 

In what follows, for any $x\in\cR^n$ we will denote by $|x|=\max\{|x_i|\,|\,i=1,\dots, n \}$ its maximum norm and, for any $R>0$, we will set $\ov{\mathrm{B}}(x,R)$ for the corresponding closed ball of center $x$ and radius $R$. We will also denote by $|\cdot|$ the induced norm in the space of matrices $\cR^{n_1\times n_2}$ ($n_1$, $n_2\in\NN$).

Let  $A\colon [0,T] \times \RR^d  \to \RR^d$, $ B\colon [0,T] \times \RR^d  \to \RR^{d\times r}$, with $r\leq d$,  $\ell\colon [0,T] \times \RR^r \times \RR^d\times \mathcal{P}_1(\cR^d) \to \RR$,  $g\colon \RR^d\times \mathcal{P}_1(\cR^d) \to \RR$, and  $p\in (1,\infty)$. We consider the following assumptions. 
\medskip
\begin{itemize}
\item[{\bf(H1)}] The function $\ell\colon [0,T]\times\cR^r\times\cR^d\times \mathcal{P}_1(\cR^d) \to \cR$ is continuous and
\smallskip
\begin{enumerate}[(i)] 
\item  there exist constants $\underline{\ell}>0$, $\overline{\ell}>0$, and $C_\ell>0$ such that
$$
\underline{\ell}|\a|^{p}-C_\ell\leq \ell(t,\a,x,\mu)\leq \overline{\ell}|\a|^{p}+C_\ell \quad \text{for all }  (t,\a,x,\mu)\in [0,T]\times\cR^r\times  \cR^d\times \mathcal{P}_1(\cR^d).
$$
\item There exists $L_\ell>0$ such that
$$
|\ell(t,\a,x,\mu)-\ell(t,\a,y, \mu)|\leq L_\ell \left(1+|\a|^p\right)|x-y|\quad \text{for all }(t,\a,\mu)\in[0,T]\times\cR^r\times \mathcal{P}_1(\cR^d), \, x,y\in\cR^d. 
$$
\item The function $\ell$ is convex and continuously differentiable with respect to its second argument. 
\end{enumerate}
\vspace{0.25cm}
\item[{\bf(H2)}] The function $g\colon \cR^d\times\mathcal{P}_1(\cR^d)\to \cR$ is continuous and there exist $c_g\in \RR$ and $L_g>0$ such that
$$
 g(x,\mu) \geq c_g  \quad \mbox{and} \quad |g(x, \mu)-g(y, \mu)|\leq L_g|x-y| \quad \text{for all } x,y\in\cR^d, \, \mu\in\mathcal{P}_1(\cR^d). 
$$ 
\item[{\bf(H3)}] The functions $A\colon [0,T]\times\cR^d\to\cR^d$ and $B\colon [0,T]\times\cR^d\to\cR^{d\times r}$ are continuous and 
\smallskip
\begin{enumerate}[(i)]
\item there exist $L_A$, $L_ B>0$ such that
$$
|A(t,x)-A(t,y)|\leq L_A |x-y|,\quad | B(t,x)- B(t,y)|\leq L_ B |x-y| \quad \text{for all } x, y\in\cR^d, \, t\in [0,T].
$$
\item There exists $C_ B>0$ such that
$$
\left|B(t,x)\right|\leq C_ B \quad \text{for all } (t,x)\in [0,T]\times \cR^d.
$$
\item There exists $\{i_1, \hdots, i_{r}\}\subset \{1,\hdots, d\}$ such that, for all $(t,x)\in [0,T]\times \RR^d$, the rows $i_1, \hdots, i_r$ of $B(t,x)$ are linearly independent. 
\end{enumerate}
\end{itemize}
\begin{remark}
\label{despues_de_hipotesis}
{\rm(i)} Notice that, setting $C_{A}= \max\{\max_{t\in [0,T]}|A(t, 0 )|,L_{A}\}$, {\bf(H3)}{\rm(i)} implies that 
\be
\label{erwrwqrqxa}
|A(t,x)|\leq C_{A}(1+|x|)  \quad \text{for all } (t,x) \in [0,T] \times \RR^d. 
\ee
{\rm(ii)} Assumption~{\bf(H3)}{\rm(iii)}, which is satisfied for instance if $B$ is constant and has full column rank, implies that, for all $(t,x) \in [0,T] \times \RR^d$, $B(t,x)$ can be decomposed into two submatrices, the first one  is formed by the rows $i_1, \hdots, i_r$ of $B(t,x)$ and is invertible, and the other one is formed by the remaining rows. As we will see in Sections \ref{fully_discrete_hjb} and \ref{sec:main_result}, this decomposition will play an important role in the construction of the approximation of the mean field game problem that we deal in this work.   
\end{remark}

For $\mu\in\P(\cR^d)$, we denote by $\supp(\mu)$ its support. We will assume the following condition on the initial distribution $m_0$.
\smallskip
\begin{itemize}
\item[{\bf(H4)}] The set $\supp(m_0)$ is compact. 
\end{itemize}\vspace{0.15cm}

Let  $x\in\cR^d$ and let $m\in C\left([0,T];\mathcal{P}_1(\cR^d)\right)$. In the MFG problem that we will consider, a typical player solves an optimal control problem of the form 
\be
\label{oc_problem}
\left\{\ba{l}
\ds \inf  \; \int_{0}^{T} \ell\left(s, \alpha(s),\gamma(s), m(s) \right)\dd s + g(\gamma(T),m(T)) \\[10pt]
\mbox{s.t. } \hspace{0.3cm} \dot{\gamma}(s)= A(s,\gamma(s))+ B(s,\gamma(s))\alpha(s) \quad \text{for a.e. } s \in (0,T),\\[6pt]
\hspace{0.9cm} \gamma(0)=x, \\[6pt]
\hspace{0.9cm} \gamma\in W^{1,p}([0,T];\cR^d), \; \alpha\in L^{p}([0,T];\cR^r).
\ea\right. \tag{$OC_{x,m}$}
\ee
 
Let us endow $\Gamma:=C\left([0,T];\cR^d\right)$ with the supremum norm $\|\cdot \|_{\infty}$ and, for all $t\in[0,T]$, define  $e_t\colon \Gamma\to\cR^d$ by $e_t(\gamma)=\gamma(t)$ for all $\gamma\in\Gamma$. Let us also set
$$
\P_{m_0}(\Gamma)=\{\xi\in\mathcal{P}_1\left(\Gamma\right)\,|\,e_0\sharp\xi = m_0\}.$$
Inspired by \cite{CH17, cannarsa_capuani_2018, cardaliaguet_meszaros_santambrogio_2018}, we consider the following problem. 
\begin{problem}
\label{mfg_problem} 
Find $\xi^*\in \mathcal{P}_{m_0}\left(\Gamma\right)$ such that $[0,T] \ni t\mapsto e_t\sharp \xi^* \in \P_1(\RR^d)$ belongs to $C\left([0,T];\mathcal{P}_1(\cR^d)\right)$ and  for $\xi^*$-a.e. $\gamma^*\in\Gamma$ there exists $\alpha^*\in L^p([0,T],\cR^r)$ such that $(\gamma^*,\alpha^*)$ solves \eqref{oc_problem} with $x=\gamma^*(0)$ and $m(t)=e_t\sharp \xi^*$ for all $t\in[0,T]$.
\end{problem}

Any $\xi^*\in  \mathcal{P}_{m_0}\left(\Gamma\right)$  solving Problem \ref{mfg_problem} is called a {\it Lagrangian {\rm MFG} equilibrium}. As mentioned in the introduction, our main focus is to approximate Lagrangian MFG equilibria (provided that they exist). Actually, the existence of solutions to Problem~\ref{mfg_problem} can be obtained as  a consequence of our approximation result in Section~\ref{convergence_result}.
\begin{theorem} Assume that {\bf(H1)-(H4)} hold. Then Problem~\ref{mfg_problem} has at least one solution. 
\end{theorem}
\begin{proof} 
This follows from Theorem~\ref{main_result}{\rm(i)} below. 
\end{proof}

Let us now consider a uniqueness result that will be useful in the sequel. A function $\Phi\colon\RR^d\times \P_1(\RR^d)\to\RR$ is said to be {\it monotone}, in the sense of \cite[Section 2.3]{MR2295621}, if 
\be
\label{monotonia_Phi}
\ds \int_{\RR^d}\big(\Phi(x,\mu_1)-\Phi(x,\mu_2)\big)\dd(\mu_1-\mu_2)(x)  \geq 0 \quad \text{for all $\mu_1$, $\mu_2\in \P_1(\RR^d)$}.
\ee
\smallskip 

We consider the following additional assumption.
\smallskip 
\begin{itemize}
\item[{\bf(H5)}]  The following hold:\smallskip
\begin{enumerate}[{\rm(i)}]
\item The function $\ell$ is given by 
$$
\hspace{0.7cm}\ell(t,a,x,\mu)= \ell_0(t,a,x) + f(t,x,\mu)  \quad \text{for all } t\in [0,T],\, a\in \RR^{r},\, x\in\RR^d,\, \mu\in \P_1(\RR^d),
$$
where $\ell_0\colon [0,T]\times \RR^{r}\times \RR^d \to \RR$ satisfies {\bf(H1)} and $f\colon[0,T]\times \RR^d\times\P_1(\RR^d)\to \RR$ is continuous, bounded, and there exists $L_{f}>0$ such that, for all $(t,\mu)\in [0,T]\times \P_1(\RR^d)$,
$$
\hspace{0.4cm}|f(t,x,\mu) -f(t,y,\mu)|\leq L_{f} |x-y| \quad \text{for all }x,y\in \RR^d. 
$$
\item  For all $t \in [0,T]$, the functions $f(t,\cdot, \cdot)$ and $g$  satisfy~\eqref{monotonia_Phi}.
\end{enumerate}
\end{itemize}
\bigskip

The proof of the following result follows essentially from the arguments in the proof of \cite[Theorem 3.2.2]{These_Saeed_18}. However, since we work under a different set of assumptions, we provide its proof in the Appendix II of this work. 
\begin{theorem}
\label{unicidad_xi_star} 
Assume that {\bf(H1)}-{\bf(H5)} hold and that
\be
\label{unicidad_casi_segura}
\text{for all } m\in C([0,T];\P_1(\cR^d)), \; \text{problem  \eqref{oc_problem} has a unique solution for  $m_0\text{-a.e.} \;  x\in \cR^d$}.
\ee
Then Problem~\ref{mfg_problem} admits a unique solution.
\end{theorem}
\section{Approximation of the value function associated with an optimal control problem}
\label{approximation_value_function_general}
In this section, we consider the approximation of a family of optimal control problems, depending on a fixed   $m\in C\left([0,T];\P_1(\RR^d)\right)$. The resulting scheme, whose main properties are shown in the Appendix I of this work,  will be the building block of the approximation of a MFG equilibrium that we propose in the next section. 

For $\alpha \in L^{p}([0,T];\RR^r)$, $t \in [0, T)$,  and $x \in \RR^d$, let us consider the equation 
\be
\label{ecuacion_controlada}
\dot{\gamma}(s)= A(s,\gamma(s))+ B(s,\gamma(s))\alpha(s), \quad \text{for a.e. } s\in (t,T), \quad \gamma(t)=x.
\ee
Under assumption {\bf(H3)}, \eqref{ecuacion_controlada} admits a unique solution.  Let us define the value function $v\colon [0,T]\times\cR^d\to \cR$ by
\be
\label{value_function_dependent_on_m}
v(t,x)=\inf\left\{\int_t^T\ell\left(s, \alpha(s),\gamma(s),m(s) \right)\dd s + g(\gamma(T), m(T)) \, \bigg| \, \alpha \in L^{p}\left([0,T];\RR^r\right), \; (\alpha, \gamma)  \mbox{ satisfies \eqref{ecuacion_controlada}}\right\},
\ee
for all $(t,x)\in [0,T]\times\RR^d$.  By ${\bf (H1)}$ and ${\bf (H2)}$, the function $v$ is well-defined and  solves the following Hamilton-Jacobi-Bellman equation in the viscosity sense {\rm(}see e.g. \cite{MR1613876,MR1484411}{\rm)}:
\be 
\label{eq:HJB}
\left\{\ba{rcl}
-\partial_t v(t,x)+ H(t,x,\nabla v(t,x),m(t))&=&0\quad\text{for }(t,x)\in (0,T)\times\cR^d,\\[6pt]
v(T,x)&=& g(x,m(T))\quad\text{for }x\in\cR^d,
\ea
\right.
\ee
where, for all $t\in [0,T]$, $x\in \RR^d$, $p\in \RR^d$, and $\mu \in \P_1(\RR^d)$, 
$$
H(t,x,p,\mu) =\sup_{\alpha\in \cR^r}\left\{ -\ell(t,\alpha,x,\mu)-\langle p, A(t,x)+B(t,x)\alpha\rangle \right\}.
$$
Moreover,    \cite[Theorem 2.1]{Da-Lio:2011aa} implies that $v$ is the unique viscosity solution to \eqref{eq:HJB}.
\subsection{A semi-discrete scheme}\label{semi_discrete_section}
Let $N_t\in \NN$ and set 
\be
\label{ienes}
\I=\{0,\hdots,N_t\}\quad\text{and}\quad\I^*=\I\setminus\{N_t\}.
\ee
For any $(k,x) \in \I^{*}\times \RR^d$, we also set   
\begin{equation}
\label{def_gamma_A_k}
\ba{rcl}
\Gamma_{k,x}&=&\{\gamma=(\gamma_k, \hdots, \gamma_{N_{t}}) \; | \; (\forall \, j=k, \hdots, N_{t}) \; \; \gamma_{j} \in \RR^d, \; \gamma_{k}=x\},\\[4pt]
\A_{k}&=& \{\alpha=(\alpha_k, \hdots, \alpha_{N_{t}-1}) \; | \; (\forall \, j=k, \hdots, N_{t}-1) \;  \; \alpha_{j} \in \RR^r\}.
\ea
\end{equation}
Set  $\Delta t=T/N_t$ and $t_k=k \Delta t$ for all $k\in \I$.  Let us define $J_{k,x}\colon \A_{k}\to \cR$ by
$$
J_{k, x}(\alpha)=\Delta t \sum_{j=k}^{N_t-1}\ell\left(t_j, \alpha_j, \gamma_j,m(t_j)\right)+g\left(\gamma_{N_t},m(T)\right),
$$
where $\gamma \in \Gamma_{k,x}$ is given by  
\be 
\label{def:dis-state}
\gamma_{j+1}=\gamma_j+\Delta t\left( A(t_j, \gamma_j)+B(t_j, \gamma_j)\alpha_j\right) \quad \text{for all }j=k, \dots, N_t-1.
\ee

We consider the following semi-discrete scheme: 
\be
\label{semidiscrete-scheme}
\left\{\ba{rcl}
v_k(x)&=&\inf\left\{J_{k, x}(\alpha)\;\big|\;\alpha\in\A_{k}\right\}\quad\text{for all } k\in \I^{*}, \; x\in\cR^d, \\[6pt]
v_{N_t}(x)&=&g(x,m(T))\quad\text{for all }x\in\cR^d. 
\ea
\right.
\ee

The discrete time value function~\eqref{semidiscrete-scheme} satisfies the following dynamic programming equation
\be
\label{semidiscrete-scheme-dpp}
\left\{\ba{rcl}
v_k( x)&=&\underset{\a\in\cR^r}\min \; \left\{\Delta t \ell(t_k,\a,x,m(t_k))+v_{k+1}\left(x+\Delta t[A(t_k, x)+ B(t_k, x)\a]\right)\right\}\quad \text{for all }k\in \I^{*},\; x\in\cR^d, \\[8pt]
v_{N_t}(x)&=&g(x,m(T))\quad\text{for all }x\in\cR^d.
\ea
\right.
\ee

In the Appendix I of this work we show the existence of $L_v>0$, independent of  $\Delta t$ and $m$, such that  
\be
\label{lem:Lips-v-semidis}
|v_k(x)-v_k(y)|\leq L_v|x-y| \quad \text{for all } k \in \I,\, x,y\in \RR^d.
\ee
Moreover, for all  $k \in \I^{*}$ and $x\in\cR^d$, the discrete time optimal control problem $\min\left\{ J_{k,x}(\alpha) \big| \;   \alpha \in \A_{k} \right\}$ admits at least one solution (see Lemma \ref{lem:existencia-control-dis}). In turn, \eqref{semidiscrete-scheme-dpp} and the Lipschitz property   \eqref{lem:Lips-v-semidis} yield the existence of $\widehat{C}>0$, independent of $\Delta t$, $m$, $k$, and $x$,  such that for every  $\widehat{\alpha} \in \argmin\left\{J_{k, x}(\alpha) \big| \;   \alpha \in \A_{k} \right\}$, we have 
\be\label{controles_acotados_unif} \max\left\{ \big|\widehat{\alpha}_{j} \big| \; | \; j=k,\hdots, N_{t}-1 \right\} \leq \widehat{C}
\ee
(see Lemma \ref{lem:controles-dis-acotados}). Thus, defining  
\be
\label{bounded_control_set}
\widehat{\A}_k= \left\{ \alpha \in \A_k \,| \, \mbox{$\alpha$ satisfies \eqref{controles_acotados_unif}} \right\} \quad \mbox{for all }k \in \I^{*},
\ee
we have that 
\be
\label{ecuacion_v_n_restriccion_alpha_acotado}
v_k( x)=  \inf\left\{J_{k, x}(\alpha)\,\big|\,\alpha\in\widehat{\A}_k\right\} \quad \mbox{for all }k \in \I^{*},\,x\in\cR^d.
\ee 

Let us state the following convergence result whose proof can be found in the Appendix I of this work. Consider two sequences $(N_t^n)_{n\in\NN}\subset\NN$ and $(m^n)_{n\in\NN}\subset C\left([0,T];\P_1(\cR^d)\right)$ such that $N_{t}^{n}\to \infty$ and $m^n\to m$ as $n\to\infty$. Set $\I^n=\{0,\hdots,N_t^n\}$,  $\Delta t_n=T/N_{t}^{n}$, $t_k^n=k \Delta t_n$ ($k\in\I^{n}$), and denote by $v^n$ the discrete value function \eqref{semidiscrete-scheme} associated with $N_{t}^{n}$ and $m^n$.
\begin{proposition}\label{prop:conv-value-function}
For every compact set $K\subset\cR^d$ we have  
$$
\sup_{(k,x)\in\I^n\times K}\left|v_{k}^n(x)-v(t_k^n,x)\right|\underset{n\rightarrow \infty}{\longrightarrow} 0. 
$$
\end{proposition}

The following remark will motivate the fully-discrete scheme, discussed in the next section, to approximate $ v(0,x)$, for $x$ in a given compact set.  
\begin{remark}[Restriction of the approximation scheme to a family of compact sets]\label{restriccion_compactos_semi_discreto}
In view of \eqref{def:dis-state} and \eqref{ecuacion_v_n_restriccion_alpha_acotado}, if $K_0\subseteq \RR^d$ is a nonempty compact set, then the values $\{v_{0}(x) \; | \; x\in K_0\}$ are determined by the following relations 
\be
\label{semidiscrete-scheme-dpp_bola}
\left\{\ba{rcl}
v_k( x)&=&\underset{\a\in\ov{\mathrm{B}}(0,\widehat{C})}\min \; \left\{\Delta t \ell(t_k, \a, x,m(t_k))+v_{k+1}\left(x+\Delta t[A(t_k, x)+ B(t_k, x)\a]\right)\right\} \\
& \; &  \hspace{7cm} \text{for all }k\in \I^{*},\,x\in K_{k}, \\[6pt]
v_{N_t}(x)&=&g(x,m(T)) \quad \text{for all } x\in K_{N_{t}},
\ea\right.
\ee
where, recalling Remark \ref{despues_de_hipotesis}{\rm{(i)}} and {\bf(H3)}{\rm(ii)}, the sets $K_{k}$ {\rm(}$k\in \I${\rm)} are defined by
\be\label{compactos_caso_semi_discreto}
K_{k+1} := K_{k} +\left(\Delta t\left[ C_{A}\big(1+\sup_{x\in K_{k}}|x|\big)+C_{B} \widehat{C}\right]\right) \ov{\mathrm{B}}(0,1)  \quad \text{for all } k \in \I^*.
\ee
By the discrete Gr\"onwall's lemma, it is easy to check that the family of compact sets $\{K_{k} \, | \, k\in \I\}$ is bounded, uniformly with respect to $\Delta t$ {\rm(}see the proof of Lemma~\ref{the_big_compact} for a justification of this fact in a slightly different framework{\rm)}. 
\end{remark}
\subsection{A fully-discrete approximation}
\label{fully_discrete_hjb}  
We introduce in this section a fully-discrete approximation of $\{ v(0,x) \; | \; x\in K_0\}$, where $K_0\subset \RR^d$ is a given nonempty compact set. This scheme will be shown to be particularly adapted to build  an approximation of Problem \ref{mfg_problem}.

From Remark~\ref{despues_de_hipotesis}{\rm(ii)}, we can assume, without loss of generality, that  $B$ has the following form 
\be
\label{B_structure}
B(t,x)= \left( \ba{c} B_{1}(t,x) \\
B_2(t,x)\ea \right)\quad\mbox{for all }(t,x)\in [0,T]\times \RR^d,
\ee 
where $B_{1}\colon  [0,T]\times\cR^d\to \RR^{r \times r}$ is such that $B_1(t,x)$ is invertible for all $(t,x) \in [0,T]\times \RR^d$ and $B_{2}\colon  [0,T]\times\cR^d\to \RR^{(d-r) \times r}$. 
Accordingly, we also decompose 
\be
\label{A_structure}
A(t,x)=\left(\ba{c}A_1(t,x)\\ 
A_2(t,x)\ea\right)\quad\mbox{for all }(t,x)\in[0,T]\times\RR^d,
\ee
where $A_1\colon [0,T]\times\cR^d\to\cR^r$ and $A_2\colon [0,T]\times\cR^d\to \cR^{d-r}$. For $x \in \RR^d$ we set $x=(\mathrm{x}_1, \mathrm{x}_2)$, where $\mathrm{x}_1\in \RR^{r}$ and $\mathrm{x}_2\in \RR^{d-r}$ denote the first $r$ and the last $d-r$ components of $x$, respectively.  

Let  $N_t\in\NN$, let $\Delta t>0$ be as in the previous section, and let $N_s\in \NN$ be such that $\Delta x:=1/N_s\leq \Delta t$. We define  
$$
\G_r =\left\{i\Delta x\;|\;i\in\ZZ^r\right\}, \quad \G_{d-r} =\left\{i\Delta x\;|\;i\in\ZZ^{d-r}\right\}, \quad \mbox{and} \quad \G =\G_r\times\G_{d-r}.
$$

Given a regular mesh $\mathscr{T}$ with vertices in $\G_{d-r}$, let $(\beta_{\mathrm{x}_2})_{\mathrm{x}_2\in\G_{d-r}}$ be a $\QQ_1$ basis, i.e. for every $\mathrm{x}_2\in\G_{d-r}$, $\beta_{\mathrm{x}_2}$ is a nonnegative polynomial of partial degree less than or equal to $1$ on each element of $\mathscr{T}$, $\beta_{\mathrm{x}_2}(\mathrm{x}_2)=1$, $\beta_{\mathrm{x}_2}(\mathrm{x}_2')=0$ for all $\mathrm{x}_2'\in \G_{d-r}$ with $\mathrm{x}_2'\neq \mathrm{x}_2$, $\sum_{\mathrm{x}_2' \in\G_{d-r}}\beta_{\mathrm{x}_2'}(z )=1$ for all $z\in \RR^{d-r}$, and the support $\supp(\beta_{\mathrm{x}_2})$ of $\beta_{\mathrm{x}_2}$ satisfies
\be\label{support_beta_x}
\supp(\beta_{\mathrm{x}_2}) \subseteq  \{ \mathrm{y}_2 \in \RR^{d-r} \; | \; |\mathrm{y}_2 -\mathrm{x}_2 | \leq C_I \Delta x \} \hspace{0.3cm} \mbox{for some $C_I>0$,  independent of $\Delta x$}.
\ee

Let $\widehat{C}>0$ be such that \eqref{controles_acotados_unif} holds.  Inspired by Remark~\ref{restriccion_compactos_semi_discreto}, and taking into account \eqref{support_beta_x}, we redefine the family of compact sets \eqref{compactos_caso_semi_discreto} as
\be\label{compacts_Ks}
K_{k+1} := K_{k} +\left(\Delta t\left[ C_{A}\big(1+\sup_{x\in K_{k}}|x|\big)+C_{B}\widehat{C}\right]+C_{I}\Delta x\right) \ov{\mathrm{B}}(0,1)  \quad \text{for all } k \in \I^*.
\ee
\begin{lemma}\label{the_big_compact} There exists a nonempty compact set $K\subset \RR^{d}$, independent of $\Delta t$ and $\Delta x$ as long as $\Delta x /\Delta t \leq 1$, such that 
$$
K_{k}\subset K_{k+1}\subset K\quad\text{for all }k\in \I^*.
$$
\end{lemma}
\begin{proof} The inclusion $K_k\subset K_{k+1}$ for $k\in \I^*$ follows directly from \eqref{compacts_Ks}.  Let $k\in \I$ and set $c_k=\sup_{x\in K_{k}} |x|$. Equation~ \eqref{compacts_Ks} yields
$$
\ba{rcl}
c_{k+1}&\leq & c_{k}+ \left(\Delta t\left[ C_{A}\big(1+c_{k}\big)+C_{B}\widehat{C} +C_{I}\frac{\Delta x}{\Delta t}\right]\right) \\[6pt]
\; &\leq & \big(1+ \Delta tC_{A}\big) c_{k} + \Delta t (C_{A}+ C_{B}\widehat{C}  +C_{I}).
\ea
$$
Thus, by the discrete Gr\"onwall's lemma\footnote{Consider three sequences $(p_n)_{n\in \NN}\subset [0,\infty)$, $(q_n)_{n\in \NN}\subset\RR$, and $(r_n)_{n\in \NN}\subset\RR$ such that 
$$
r_{n+1} \leq p_{n+1} r_{n} + q_{n+1}\quad\text{for all }n\in\NN.
$$
Then, setting $\mathtt{P}_{n}= \prod_{j=1}^{n} p_j$ and $\mathtt{P}_{k,n}= \prod_{j=k}^{n} p_j$, we have
$$
r_{n}\leq\mathtt{P}_{n}r_0+\sum_{k=1}^{n} \mathtt{P}_{k,n} q_k\quad\text{for all }n\in\NN.
$$
}, there exists  $\ov{c}>0$, independent of $\Delta t$ and $\Delta x$,  such that $\max\{c_k \, | \, k \in \I\} \leq \ov{c}$. The result follows.
\end{proof}
Notice that  {\bf (H3)}  yields the existence of  $c_K>0$ such that 
\be\label{c_k}
|B_1(t, x)^{-1}|\leq c_K \quad \text{for all }t\in [0,T],\,x\in K. 
\ee
 
Let us fix $k\in \I^*$ and $x\in  \G$. For any $\mathrm{y}_1\in \RR^r$, let us set
\be
\label{definition_alpha_n}
\ba{rcl}
\alpha (k,x, \mathrm{y}_1)&:=&B_1(t_k, x)^{-1}\left[ \frac{\mathrm{y}_1-\mathrm{x}_1}{\Delta t}-A_1(t_k, x)\right]\in \RR^r,\\[6pt]
\ydos(k,x,\mathrm{y}_1)&:= &  \mathrm{x}_2 +\Delta t \left[A_2(t_k, x)+B_2(t_k,x)\alpha(k, x, \mathrm{y}_1)\right]  \in\RR^{d-r},
\ea
\ee
and  define the finite sets
\be
\label{definicion_de_los_S}
\ba{rcl}
\SS_{k+1}^1(x)&=&\left\{\mathrm{y}_1\in\G_r\; \big|\; |\alpha (k,x, \mathrm{y}_1)|\leq  \widehat{C} \right\}, \\[6pt]
\SS_{k+1}^2(x,\mathrm{y}_1)&=&\left\{\mathrm{y}_2 \in\G_{d-r}\; \big|\; \ydos(k,x,\mathrm{y}_1)\in \supp{\beta_{\mathrm{y}_2}} \right\} \quad \text{for }  \mathrm{y}_1\in \SS_{k+1}^1(x),\\[6pt]
\SS_{k+1}(x)&=&\left\{(\mathrm{y}_1,\mathrm{y}_2)\in\G\; \big|\; \mathrm{y}_1\in\SS_{k+1}^1(x),\;      \mathrm{y}_2\in \SS_{k+1}^2(x,\mathrm{y}_1)\right\}.
\ea
\ee
Notice that \eqref{definition_alpha_n} and \eqref{compacts_Ks} imply 
\be
\label{cotas_inductivas}
\SS_{k+1}(x)\subseteq K_{k+1}\quad \text{for all } k\in \I^*,\,x\in K_{k}.
\ee
Let us define 
\be
\label{conjunto_de_parametros}
\widehat{\Delta}= \big\{(\Delta t, \Delta x) \in  (0,\infty)^2 \, | \, \Delta x/\Delta t\leq \min\{1,\widehat{C}/c_K\}\big\}.
\ee
\begin{lemma}
\label{S_1_nonempty} 
Let $c_K>0$ be as in \eqref{c_k} and suppose that $(\Delta t,\Delta x) \in \widehat{\Delta}$.  Then, for any $k\in\I^*$ and $x\in K_k\cap \G$, the set $\SS_{k+1}(x)$ is nonempty. 
\end{lemma}
\begin{proof} Let $\bar{j}\in\ZZ^r$ be such that 
$$
\left|\bar{j}-\frac{\Delta t}{\Delta x}A_1(t_k, x)\right|\leq 1.
$$
Setting $\mathrm{y}_1= \mathrm{x}_1+\bar{j}\Delta x\in\G_r$, it follows that
$$
|\alpha(k,x, \mathrm{y}_1)|\leq \big|B_1(t_k, x)^{-1}\big|\left| \frac{\mathrm{y}_1-\mathrm{x}_1}{\Delta t}-A_1(t_k, x)\right|\leq c_K\frac{\Delta x}{\Delta t}\left|\bar{j}-\frac{\Delta t}{\Delta x}A_1(t_k, x)\right|\leq \widehat{C},
$$
which implies that $\mathrm{y}_1\in \SS^1_{k+1}(x)$. In addition, since  $\sum_{\mathrm{y}_2 \in\G_{d-r}}\beta_{\mathrm{y}_2 }( \ydos(k,x,\mathrm{y}_1))=1$ we deduce that   $\SS_{k+1}(x)\neq \emptyset$.
\end{proof}

Let us set  
\be
\label{eq:def-Sk}
\SS_{0}=\G\cap K_{0}\quad\text{and}\quad\SS_{k+1}=\bigcup\limits_{x\in\SS_{k}}\SS_{k+1}(x)\quad\mbox{for all }k\in \I^*.
\ee
Notice that \eqref{cotas_inductivas} and Lemma~\ref{the_big_compact} imply that
\be
\label{s_k_in_big_compact}
\SS_{k}\subset K_{k}\subset K\quad\text{for all }k\in\I. 
\ee

Given $F\subseteq\G_{d-r}$ and $\varphi\colon F\to\RR$, define the interpolant $I_{F}[\varphi]\colon \RR^{d-r}\to\RR$ by
$$
I_{F}[\varphi](\mathrm{y}_2)=\sum_{\mathrm{x}_2 \in F}\beta_{\mathrm{x}_2}(\mathrm{y}_2)\varphi(\mathrm{x}_2 )  \quad \text{for all }\mathrm{y}_2\in\RR^{d-r}.
$$

We consider the following scheme to approximate $v(t_k, x)$ for all $k\in \I$ and $x\in \SS_{k}$:
\be
\label{fully_discreto_HJB}
\ba{rcl}
 \mathcal{V}_k(x) &=&\min\limits_{ \mathrm{y}_1\in  \SS^1_{k+1}(x)} \bigg\{\Delta t \ell(t_k, \alpha(k,x, \mathrm{y}_1), x,m(t_k)) +I_{\SS_{k+1}^2(x,\mathrm{y}_1)}[\mathcal{V}_{k+1}(\mathrm{y}_1,\cdot)](\ydos(k,x,\mathrm{y}_1))\bigg\}\\[12pt]
 \; & \; & \hspace{10cm}\text{for all } k\in \I^*,\,x\in \SS_{k}, \\[-6pt]
\mathcal{V}_{N_t}(x)&=& g(x,m(T)) \quad \text{for all }x\in \SS_{N_t}. 
\ea  
\ee
Notice that the previous scheme can be rewritten as 
\be
\label{fully_discrete_HJB}
\ba{l}
\mathcal{V}_k(x) =\min\limits_{p \in\P(\SS^1_{k+1}(x))}\bigg\{\sum\limits_{\mathrm{y}_1 \in  \SS^1_{k+1}(x)} p(\mathrm{y}_1)\Big[\Delta t \ell(t_k, \alpha(k,x, \mathrm{y}_1), x,m(t_k)) \\[5pt]\hspace{5cm}+I_{\SS_{k+1}^2(x,\mathrm{y}_1)}[\mathcal{V}_{k+1}(\mathrm{y}_1,\cdot)]\big(\ydos(k,x,\mathrm{y}_1)\big)\Big]\bigg\}\quad \text{for all }k\in \I^*, \; x\in \SS_{k},  \\[2pt]
\mathcal{V}_{N_t}(x)= g(x,m(T))\quad\text{for all }x\in \SS_{N_t}.
\ea  
\ee

An interesting property of formulation~\eqref{fully_discrete_HJB} is that the optimization problem defining $\mathcal{V}_k(x)$ is a linear programming problem with a nonempty and convex set of solutions.  In the next section, and in the context of Problem \ref{mfg_problem},  we will consider a strictly convex variation of~\eqref{fully_discrete_HJB} ensuring the uniqueness of minimizers. The latter will play an important role in the numerical solution of the resulting scheme for Problem \ref{mfg_problem}.  
\begin{remark}{\rm(i)} Consider three sequences $(N_t^n)_{n\in\NN}\subset\NN$, $(N_s^n)_{n\in\NN}\subset\NN$, and $(m^n)_{n\in\NN}\subset C\left([0,T];\P_1(\cR^d)\right)$  such that, as $n\to\infty$, $N_{t}^{n}\to \infty$, $N_{s}^{n}\to \infty$, $N_{t}^{n}/N_{s}^{n}\to 0$, and $m^n\to m$. Set $\I^n=\{0,\hdots, N_t^n\}$, $\G^n=\{i/N_{s}^{n} \, | \, i\in\ZZ^d\}$, and denote by $v^n$ and $\mathcal{V}^n$ the solutions to \eqref{semidiscrete-scheme-dpp_bola},  with $\Delta t= T/N_{t}^{n}$, and to \eqref{fully_discreto_HJB}, with $\Delta t= T/N_{t}^{n}$ and $\Delta x= 1/N_{s}^{n}$, respectively.  In Lemma~\ref{lem:conv-v-vt} we will show that 
$$
\sup_{x\in \G^n\cap K_0}\left|\mathcal{V}_0^n(x)-v^n_{0}(x)\right|\underset{n\rightarrow \infty}{\longrightarrow} 0
$$
and hence, from Proposition~\ref{prop:conv-value-function}, we have
\be   \label{eq:lim-sup-v-n_fully}
\sup_{x\in \G^n\cap K_0}\left|\mathcal{V}_0^n(x)-v(0,x)\right|\underset{n\rightarrow \infty}{\longrightarrow} 0,
\ee
where $v$ is the unique viscosity solution to \eqref{eq:HJB}.

{\rm(ii)} Notice that, for every $k\in \I^*$ and $x\in K_{k}$, we have
$$
\SS^{1}_{k+1}(x)=  \left[\mathrm{x}_1 + \Delta t A_1(t_k,x) +\Delta t B_1(t_k,x) \ov{\mathrm{B}}(0,\widehat{C})\right]\cap \G_{r}.
$$ 

In the numerical implementation of \eqref{fully_discrete_HJB}, the sets $\{\SS^{1}_{k+1}(x) \; | \; k\in \I^*, \; x \in \SS_{k}\}$ can be difficult to compute when $B(t_k,x)$ does not have a simple structure. These sets can be replaced by larger ones such as 
$$
\widehat{\SS}^{1}_{k+1}(x)= \left\{\mathrm{y}_1\in\G_r\; \big|\; |\mathrm{y}_1-\mathrm{x}_1-\Delta t A_1(t_k, x)|\leq \Delta t |B(t_k, x)|\widehat{C} \right\},
$$
or even 
$$
\widehat{\SS}^{1}_{k+1}(x)=\{ \mathrm{y}_1\in\G_r \, | \; (\mathrm{y}_1, \ydos(k,x,\mathrm{y}_1)) \in K_{k+1}\}.
$$
With these modifications, the results in the following sections hold true and the corresponding proofs are slightly easier. However, we have chosen to present our results with $\SS^{1}_{k+1}(x)$ because it is the smallest set for which our approach allows to approximate the values $\{v(0,x) \; | \; x\in K_0\}$. 
\end{remark}
\section{Discretization of the MFG problem}
\label{sec:main_result}
 We study in this section the existence and the uniqueness of solutions to a full-discretization of Problem \ref{mfg_problem}. At the end of this section, we also discuss the convergence of the  {\it fictitious play} method  \cite{MR4030259} towards the solution of the discrete problem.

The definition of MFG equilibrium involves an optimal control problem solved by a ``typical" player. Thus, in order to find a discrete version of Problem \ref{mfg_problem}, amenable to numerical computation, it is natural to first discretize the aforementioned optimal control problem. For this purpose, we will consider a variation of \eqref{fully_discrete_HJB} with an additional entropy term that ensures the existence of a unique minimizer. 

Let $N_t\in\NN$,  $N_s\in\NN$, and $(\Delta t, \Delta x)\in \widehat{\Delta}$ (see \eqref{conjunto_de_parametros}). Recall that  for any $(k, x)\in \I^*\times \G$, the sets $ \SS^1_{k+1}(x)$, $\SS^2_{k+1}(x,\mathrm{y}_1)$, for $\mathrm{y}_1\in \SS^1_{k+1}(x)$,   and $ \SS_{k+1}(x)$ are defined in \eqref{definicion_de_los_S}. Let us define the family of sets $\{\SS_{k} \, | \, k \in \I\}$ by \eqref{eq:def-Sk} with $K_0=\supp(m_0)$.  In order to discretize the initial distribution $m_0$, let us set 
$$
E(x) =\left\{y\in\cR^d\;|\;|x-y|\leq \Delta x/2\right\} \quad \mbox{for} \; x\in \G.
$$ 
By modifying $\Delta x$, if necessary, we can assume that $m_0(\partial E(x))=0$ for all $x\in\G$. For a finite set $F$, let us define the (nonpositive) entropy function $\mathcal{E}_F\colon \mathcal{P}(F)\to\cR$ by
$$
\mathcal{E}_F(p) =\sum_{x\in F}p(x)\log(p(x)) \quad \text{with the convention that  $p(x)=p(\{x\})$ and $0\log(0)=0$}.
$$
In what follows, we identify every $p\in \P(\SS_k)$ ($k\in \I$) with the probability measure  $\sum_{x\in\SS_k} p(x)\delta_x\in\P_1(\cR^d)$. Let $\eps>0$,  set  $\M=\prod_{k \in \I}\P(\SS_{k})$ and, for $M=(M_k)_{k\in \I}\in \M$,  define $\{V_{k}^{M} \, | \, k\in \I\}$ by 
\be
\label{eq:hjb_eq}
\ba{l}
V_k^{M}(x)=\min\limits_{p\in \mathcal{P}( \SS^1_{k+1}(x))} \bigg\{\sum\limits_{\mathrm{y}_1\in  \SS_{k+1}^1(x)}p(\mathrm{y}_1)\bigg[ \Delta t  \ell(t_k, \alpha(k,x, \mathrm{y}_1), x, M_k) \\[6pt]
\hspace{3.2cm} +I_{\SS_{k+1}^2(x,\mathrm{y}_1)}[V_{k+1}^{M}(\mathrm{y}_1,\cdot)](\ydos(k,x,\mathrm{y}_1))\bigg] +\eps\mathcal{E}_{\SS^1_{k+1}(x)}(p)\bigg\} \quad \text{for all } k\in \I^*,\,x\in \SS_{k}, \\ [8pt]
 V_{N_t}^{M}(x)=g(x, M_{N_t}) \quad \text{for all }x\in \SS_{N_t},
\ea
\ee
where we recall that $\alpha(k,x,\mathrm{y}_1)$ is defined in \eqref{definition_alpha_n} for every $(k,x,\mathrm{y}_1)\in \I^*\times \G\times \G_r$. For $k\in \I^*$, $x\in \SS_k$, and $y\in \SS_{k+1}$, we also set 
\be\label{eq:def-P_k_M}
P_k^{M}(x, y):=  \begin{cases}p_k^{M}(x, \mathrm{y}_1) \beta_{\mathrm{y}_2}(\ydos(k,x,\mathrm{y}_1)) & \mbox{if $y \in \SS_{k+1}(x)$}
,\\[1mm]
0 &  \mbox{if $y \in \SS_{k+1} \setminus \SS_{k+1}(x)$}, 
\end{cases}
\ee
where 
\be\label{eq:p_k_M}
\ba{l}
p_k^{M}(x, \cdot) \in  \argmin\limits_{p\in\mathcal{P}(\SS_{k+1}^1(x))} \bigg\{\sum\limits_{\mathrm{y}_1\in \SS_{k+1}^1(x)}p(\mathrm{y}_1)\bigg[ \Delta t \ell(t_k, \alpha(k,x, \mathrm{y}_1), x, M_k)  \\[10pt]
\hspace{6cm} +I_{\SS_{k+1}^2(x,\mathrm{y}_1)}[V_{k+1}^{M}(\mathrm{y}_1,\cdot)](\ydos(k,x,\mathrm{y}_1))\bigg]+\eps \mathcal{E}_{\SS^1_{k+1}(x)} (p)\bigg\}.
\ea
\ee

Notice that, for every $k\in\I^*$, $x\in\SS_k$, we have $P_{k}^{M}(x,\cdot)\in \P(\SS_{k+1})$ and $\supp(P_{k}^{M}(x,\cdot))=\SS_{k+1}(x)$.  Finally, define the {\it best response mapping} ${\bf br}\colon \M \ni M \mapsto {\bf br}(M)=\widehat{M} \in \M$, where $\widehat{M}$ is given by 
\be
\label{def-br}
\ba{rcl}
\widehat{M}_{k+1}(y)&=&\sum\limits_{x\in\SS_{k}}P_k^{M}(x,y)\widehat{M}_{k}(x)\quad\text{for all }k\in \I^*,\,y\in \SS_{k+1},\\ [13pt]
\hspace{0.34cm}\widehat{M}_0(x)&=&m_0(E(x))\quad\text{for all } x\in \SS_{0}.
\ea
\ee
\begin{remark}\label{rem:prop_br} 
 Notice that the presence of the entropy term in the right-hand side  of \eqref{eq:p_k_M} ensures that $p_{k}^M(x,\cdot)$ is uniquely defined for all $k\in\I^*$, $x\in \SS_k$, and $M\in \M$, and $p_{k}^M(x,\mathrm{y}_1)>0$ for all $\mathrm{y}_1\in \SS_{k+1}^1(x)$. Consequently, for every $\mathrm{y}_1\in \SS_{k+1}^{1}(x)$,  the application $\M\ni M \mapsto p_{k}^{M}(x,\mathrm{y}_1)\in [0,1]$ is continuous. Therefore, the best response mapping ${\bf br}$ is well-defined and continuous. 
 
Moreover, given $M\in \M$, setting $\widehat{M}={\bf br}(M)$, \eqref{def-br} and \eqref{eq:def-P_k_M} imply
$$
\widehat{M}_k(x)>0\quad \text{for all }k\in\I,\,x\in\SS_k.
$$
\end{remark} 

The discretization of Problem \ref{mfg_problem} that we consider in this work is the following. 
\begin{problem}
\label{MFG-SL_discrete_scheme_n-degenerate}  
Find $M\in \M$ such that $M={\bf br}(M)$.
\end{problem}

The following proposition deals with the existence of solutions to Problem \ref{MFG-SL_discrete_scheme_n-degenerate}.
\begin{proposition}
\label{existencia_caso_finito} 
Problem \ref{MFG-SL_discrete_scheme_n-degenerate} admits at least one solution $M^*\in \M$. 
\end{proposition}
\begin{proof} Since $\M$ is non-empty, convex and compact, and ${\bf br}$ is continuous (see Remark \ref{rem:prop_br}), the result follows from Brouwer's fixed point theorem. 
\end{proof}
\begin{remark}\label{compact_support_uniform}  Since $\supp(m_0)$ is compact,  \eqref{s_k_in_big_compact}  implies the existence of $C_{\infty}>0$, independent of $(\Delta t,\Delta x)\in \widehat{\Delta}$  and $\eps >  0$,  such that $\SS_k \subseteq \ov{\mathrm{B}}(0,C_{\infty})$ for all $k\in \I$. In particular,  for any solution $M^*\in \M$ to Problem~\ref{MFG-SL_discrete_scheme_n-degenerate},  the support of $M_{k}^*$ {\rm(}$k\in \I${\rm)} is contained in $\ov{\mathrm{B}}(0,C_{\infty})$.
\end{remark} 

We now discuss the uniqueness of solutions to Problem~\ref{MFG-SL_discrete_scheme_n-degenerate}.
\begin{proposition}
\label{uniqueness_mfg_discreto}   Under {\bf(H5)}, Problem~\ref{MFG-SL_discrete_scheme_n-degenerate} admits a unique solution. 
\end{proposition}
\begin{proof}
Suppose that $M$ and $\bar{M}$ solve Problem \ref{MFG-SL_discrete_scheme_n-degenerate}. By \eqref{eq:hjb_eq}  we have 
$$
\ba{lll}
V_0^M(x)&\leq &\Delta t\sum\limits_{\mathrm{y}_1\in \SS^1_{1}(x)}p^{\bar{M}}_0(x,\mathrm{y}_1)\big[\ell(t_0, \alpha(0,x,\mathrm{y}_1), x,M_0) +I_{\SS^2_{1}(x,\mathrm{y}_1)}[V_1^M(\mathrm{y}_1,\cdot)](\ydos(0,x,\mathrm{y}_1))\big]\\[12pt]
\; & \; &\quad  \quad +\eps\mathcal{E}_{\SS^1_{k+1}(x)} (p_0^{\bar{M}}(x,\cdot))\\[6pt]
\;&=&\Delta t\sum\limits_{y\in \SS_{1}(x)}P^{\bar{M}}_0(x,y)\left(\ell(t_0, \alpha(0,x,\mathrm{y}_1), x,M_0)+V_1^M(y)\right)+\eps\mathcal{E}_{\SS^1_{k+1}(x)}(p_0^{\bar{M}}(x,\cdot)).
\ea
$$
 By multiplying by $\bar{M}_0(x)$ and summing over $x\in\SS_0$, \eqref{eq:def-Sk}, \eqref{eq:def-P_k_M}, and \eqref{def-br} yield 
$$
\ba{lll}
\sum\limits_{x\in\SS_0}\bar{M}_0(x)V_0^M(x)&\leq &\Delta t\sum\limits_{x\in\SS_0}\sum\limits_{y\in \SS_{1}}\bar{M}_0(x)P^{\bar{M}}_0(x,y)\ell(t_0, \alpha(0,x,\mathrm{y}_1), x,M_0)+\sum\limits_{y\in \SS_{1}}\bar{M}_1(y)V_1^M(y)\\[12pt]
\;&\;&\quad+\eps\sum\limits_{x\in\SS_0} \bar{M}_0(x) \mathcal{E}_{\SS^1_{k+1}(x)}(p_0^{\bar{M}}(x,\cdot)).
\ea
$$
Applying the same argument to $V_1^M(y)$ and so forth,  we obtain
\be\label{in:m_bar_m}
\ba{lll}
\sum\limits_{x\in\SS_0}\bar{M}_0(x)V_0^M(x)&\leq&\Delta t\sum\limits_{k=0}^{N_t-1}\sum\limits_{x\in\SS_k}\sum\limits_{y\in \SS_{k+1}}\bar{M}_k(x)P_k^{\bar{M}}(x,y)\ell(t_k, \alpha(k,x,\mathrm{y}_1),x,M_k)\\[8pt]
\;&\;&+\sum\limits_{y\in\SS_{N_t}} \bar{M}_{N_t}(y)V^M_{N_t}(y) +\eps\sum\limits_{k=0}^{N_t-1}\sum\limits_{x\in\SS_k}\bar{M}_k(x)\mathcal{E}_{\SS^1_{k+1}(x)}(p_k^{\bar{M}}(x,\cdot)).
\ea
\ee
Notice that Remark~\ref{rem:prop_br}  implies that the inequality above is an equality if and only if $\bar{M}=M$. Thus, it holds that 
\be\label{eq:aux_uniqueness_3}
\ba{lll}
\sum\limits_{x\in\SS_0}M_0(x)V_0^M(x)&=&\Delta t\sum\limits_{k=0}^{N_t-1}\sum\limits_{x\in\SS_k}\sum\limits_{y\in \SS_{k+1}}M_k(x)P_k^{M}(x,y)\ell(t_k, \alpha(k,x,\mathrm{y}_1),x,M_k)
\\[8pt]
\;&\;&+\sum\limits_{y\in\SS_{N_t}} M_{N_t}(y)V^M_{N_t}(y)+\eps\sum\limits_{k=0}^{N_t-1}\sum\limits_{x\in\SS_k}M_k(x)\mathcal{E}_{\SS^1_{k+1}(x)}(p_k^{M}(x,\cdot)).
\ea
\ee
Now, let us suppose that  $M\neq \bar{M}$. By interchanging the roles of $M$ and $\bar{M}$ in \eqref{in:m_bar_m}, we obtain
\be\label{eq:aux_uniqueness_4}
\ba{lll}
\sum\limits_{x\in\SS_0}M_0(x)V_0^{\bar{M}}(x)&<&\Delta t\sum\limits_{k=0}^{N_t-1}\sum\limits_{x\in\SS_k}\sum\limits_{y\in \SS_{k+1}}M_k(x)P_k^M(x,y)\ell(t_k, \alpha(k,x,\mathrm{y}_1),x,{\bar{M}}_k)\\[8pt]
\;&\;&+\sum\limits_{y\in\SS_{N_t}} M_{N_t}(y)V^{\bar{M}}_{N_t}(y)+\eps\sum\limits_{k=0}^{N_t-1}\sum\limits_{x\in\SS_k}M_k(x)\mathcal{E}_{\SS^1_{k+1}(x)}(p_k^M(x,\cdot)).
\ea
\ee
By {\bf(H5)}{\rm(i)}, \eqref{eq:aux_uniqueness_3}, and \eqref{eq:aux_uniqueness_4}, we get 
\be\label{eq:aux_uniqueness_1}
\ba{lll}
\sum\limits_{x\in\SS_0}M_0(x)\left(V_0^{\bar{M}}(x)-V_0^M(x)\right)&<&\Delta t\sum\limits_{k=0}^{N_t-1}\sum\limits_{x\in\SS_k}M_k(x)\left(f(x,{\bar{M}}_k)-f(x,M_k(x))\right)\\[8pt]
\;&\;&\quad+\sum\limits_{y\in\SS_{N_t}} M_{N_t}(y)\left(V^{\bar{M}}_{N_t}(y)-V_{N_t}^M(y)\right).
\ea
\ee
Similarly, recalling that $M_0=\bar{M}_0$, we have that 
\be\label{eq:aux_uniqueness_2}
\ba{lll}
\sum\limits_{x\in\SS_0}M_0(x)\left(V_0^{M}(x)-V_0^{\bar{M}}(x)\right)&<&\Delta t\sum\limits_{k=0}^{N_t-1}\sum\limits_{x\in\SS_k}\bar{M}_k(x)\left(f(x, M_k)-f(x,\bar{M}_k(x))\right)\\[8pt]
\;&\;&\quad+\sum\limits_{y\in\SS_{N_t}} \bar{M}_{N_t}(y)\left(V^{M}_{N_t}(y)-V_{N_t}^{\bar{M}}(y)\right).
\ea
\ee
Noting that $V^{M}_{N_t}(y)=g(y, M_{N_{t}})$ and $V^{\bar{M}}_{N_t}(y)=g(y, \bar{M}_{N_{t}})$, by adding \eqref{eq:aux_uniqueness_1} and \eqref{eq:aux_uniqueness_2},  we obtain  
$$\ba{rcl}
0 &<& \Delta t\sum\limits_{k=0}^{N_t-1}\sum\limits_{x\in\SS_k}\left(M_k(x)-\bar{M}_k(x)\right)\left(f(x,{\bar{M}}_k)-f(x,M_k(x))\right)\\[8pt]
\; & \; & +\sum\limits_{y\in\SS_{N_t}} \left(M_{N_t}(y)-\bar{M}_{N_t}(y)\right)\left(g(y,{\bar{M}}_{N_t})-g(y, M_{N_t})\right),\ea
$$
which contradicts {\bf(H5)}(ii). 
\end{proof}

Besides the uniqueness of solutions to Problem~\ref{MFG-SL_discrete_scheme_n-degenerate}, the particular structure of $\ell$  and the monotonicity assumptions on $f$ and $g$ in {\bf(H5)} are important in order to approximate the unique solution $M^*$ to Problem~\ref{MFG-SL_discrete_scheme_n-degenerate}. Indeed, if {\bf(H5)} holds and $f$ and $g$ satisfy 
\begin{eqnarray} 
|f(t,x,\mu_1)-f(t,x,\mu_2)|&\leq &C_{f} d_{1}(\mu_1,\mu_2) \quad \text{for all } t\in [0,T],\,x\in \RR^d,\,\mu_1,\,\mu_2\in\P_1(\RR^d), \label{Lipschitz_property_f_w_r_t_mu}\\
|g(x,\mu_1)-g(x,\mu_2)|&\leq &C_{g} d_{1}(\mu_1,\mu_2)\quad\text{for all }x\in \RR^d,\,\mu_1,\,\mu_2\in \P_1(\RR^d),\label{Lipschitz_property_g_w_r_t_mu}
\end{eqnarray}
and we consider the {\it fictitious play} sequence 
$$
\ov{\mathsf{M}}^0\in \M \; \; \text{arbitrary}, \;  \quad (\forall\,n\geq 1) \quad \mathsf{M}^{n+1}= {\bf br}(\ov{\mathsf{M}}^{n}), \quad \ov{\mathsf{M}}^{n+1}= \frac{n}{n+1}\ov{\mathsf{M}}^{n} + \frac{1}{n+1}\mathsf{M}^{n+1}, 
$$
then \cite[Theorem~3.2]{MR4030259} implies that 
$$(\mathsf{M}^{n}, \ov{\mathsf{M}}^{n}) \underset{n\to\infty}{\longrightarrow} (M^*, M^*).
$$

\section{Convergence  of solutions to the finite MFG problem towards a solution to  Problem \ref{mfg_problem}}
\label{convergence_result} 
Let $(N_t^n)_{n\in\NN}\subset\NN$, $(N_s^n)_{n\in\NN}\subset\NN$, $(\eps_n)_{n\in\NN}\subset (0,\infty)$,  and, for every $n\in\NN$, set $\Delta t_n= T/N_{t}^{n}$, $\Delta x_n= 1/N_{s}^{n}$, $\I^n=\{0,\hdots, N_t^n\}$, $\I^{n,*}:=\I^n\setminus \{N_{t}^{n}\}$,  $t^n_k=k\Delta t_n\;(k\in \I^n)$, and $\G^n=\{i\Delta x_n \, | \, i\in\ZZ^d\}$.  We assume that $(\Delta t_n, \Delta x_n)\in \widehat{\Delta}$, where we recall that $\widehat{\Delta}$ is defined in \eqref{conjunto_de_parametros}. For $k\in \I^{n,*}$ and $x\in \G^n$, we denote by $\SS_{k+1}^{1,n}(x)$, $\SS^{2,n}_{k+1}(x,\mathrm{y}_1)$, for $\mathrm{y}_1\in \SS^{1,n}_{k+1}(x)$,  and $\SS_{k+1}^n(x)$ the sets defined in \eqref{definicion_de_los_S} associated with $\Delta t_n$ and $\Delta x_n$. For $k\in \I^n$, the set $\SS^n_{k}$ is defined as in \eqref{eq:def-Sk}.   Recall that Proposition~\ref{existencia_caso_finito} ensures the existence of at least one solution $M^n\in \M^n:=\prod_{k \in \I^n}\P(\SS^n_{k})$ to Problem~\ref{MFG-SL_discrete_scheme_n-degenerate} associated with the previous parameters. Denote by $\Gamma^n$ the set of continuous functions $\gamma\colon [0,T]\to \cR^d$ such that for each $k \in \I^n$,   $\gamma(t^n_k)\in \SS_{k}^n$ and, for every $k\in\I^{n,*}$,  the restriction of $\gamma$ to the interval $[t^n_{k}, t^n_{k+1}] $ is affine. Recalling \eqref{eq:def-P_k_M}, let us define $\xi^n\in \P(\Gamma)$ as
\be 
\label{eq:def-xi-n}
 \xi^n= \sum_{\gamma \in \Gamma^{n}} M_0^n(\gamma(0))P^n(\gamma)\delta_{\gamma} \in \P(\Gamma), \quad \text{where} \quad P^{n}(\gamma):=  \prod_{k=0}^{N^n_t-1}P^{M^n}_k(\gamma(t^n_k), \gamma(t^n_{k+1})). 
\ee

Since, for every  $k\in \I^n$, $M^n_k=e_{t^n_k}\sharp \xi^n$,  it is natural to to extend $M^n$ to $[0,T]$ as 
\be
\label{rem:def-Mn-cont}
[0,T] \ni t \mapsto M^n(t):=e_t\sharp \xi^n \in  \P_1(\RR^d). 
\ee
It is easy to check  {\rm(}see Lemma~\ref{prop:compactness} below{\rm)}  that 
 $M^n\in C([0,T];\P_1(\RR^d))$. Finally, we denote by $\{V^n_k\;|\; k\in\I^n\}$ the solution to  \eqref{eq:hjb_eq} with $M=M^n$.

\begin{remark}
\label{rem:cota-inf-supp-xi}
{\rm(i)} Notice that \eqref{eq:def-P_k_M} implies that $\gamma \in \supp(\xi^n)$ if and only if for every $k\in \I^{n,*}$, we have  $\gamma(t_{k}^{n}) \in \SS_{k}^{n}$ and $\gamma(t_{k+1}^{n})\in \SS_{k+1}^{n}(\gamma(t_{k}^{n}))$. Moreover, if $\gamma \in \supp(\xi^n)$, then $\gamma$ is absolutely continuous with 
\be\label{expresion_Derivatives}
\dot{\gamma}(t)=  \frac{\gamma(t_{k+1}^n)-\gamma(t_{k}^n) }{\Delta t_n} \quad \text{for all }k\in \I^{n,*},\,t\in(t_{k}^n, t_{k+1}^n).
\ee
\smallskip\\
{\rm(ii)} By definition of $\xi^n$ we have $ \supp(\xi^n)\subset \Gamma^n\subseteq W^{1,\infty}([0,T];\RR^d)$. Thus, the definition of $\Gamma^n$ and  Remark~\ref{compact_support_uniform} imply that
$$
\|\gamma \|_{\infty} \leq C_{\infty}\quad\text{for all }\gamma\in \supp(\xi^n).
$$
Moreover, Lemma~\ref{lem:deriv-acotada-sup-xi-n}  below shows the existence of a constant $D_{\infty}>0$ such that 
$$
\|\dot{\gamma} \|_{\infty}\leq D_{\infty}\quad\text{for all }\gamma\in \supp(\xi^n).
$$
\end{remark}

We can now state the main result of this paper. 
\begin{theorem}   
\label{main_result}
 Assume that {\bf(H1)-(H4)} hold and that, as $n\to \infty$, $N_t^{n} \to \infty$, $N^n_s \to \infty$, $   N^n_t/N^n_s \to 0$, and $\eps_n=o\left( 1/(N^n_t \log(N^n_s)) \right)$. Then the following hold:
\begin{enumerate}[{\rm(i)}]
\item  There exists at least one accumulation point $\xi^*$ of $(\xi^n)_{n\in\NN}$, with respect to the narrow topology in $\P(\Gamma)$, and every such accumulation point solves Problem \ref{mfg_problem}.

\item  Let $ (\xi^{n_{l}})_{l\in\NN}$ be a  subsequence of  $\left(\xi^{n}\right)_{n\in\NN}$ converging towards a MFG equilibrium  $\xi^*\in \P(\Gamma)$ and set $[0,T] \ni t \mapsto m^*(t)= e_t \sharp \xi^* \in \P(\RR^d)$.  Then   $m^*\in C([0,T]; \P_1(\RR^d))$, $\supp(m^*(t))\subset \ov{\mathrm{B}}(0,C_{\infty})$ for all $t\in [0,T]$,   $M^{n_l}\to m^*$ in $C([0,T]; \P_1(\RR^d))$ as $l\to \infty$, and 
\be
\label{convergencia_value_function}
\sup \left\{ \left|V_{k}^{n_l}(x) - v^*(t_k^{n_l},x) \right| \; \big| \; k\in \I^{n_l}, \;  x\in \SS_k^{n_l}\right\}\underset{l\to \infty}{\longrightarrow} 0,
\ee
where $v^*$ is the value function \eqref{value_function_dependent_on_m} associated with $m^*$.
\end{enumerate}
\end{theorem}
\begin{remark} 
\label{comentarios_despues_del_teorema}
{\rm(i)} In addition to the assumptions of the previous theorem, let us assume that {\bf(H5)} holds. If \eqref{unicidad_casi_segura} holds, then Theorem~\ref {unicidad_xi_star}  implies that the MFG equilibrium $\xi^*$ is unique and hence  the whole sequence $(\xi^n)_{n\in \NN}$ converges $\xi^{*}$. In particular,  we have 
\be
\label{convergencia_marginales}
M^{n} \to m^* \quad \text{in } C([0,T];\P_1(\RR^d)) \; \; \text{as $n\to \infty$}.
\ee 
{\rm(ii)}  Notice that even in the particular case where the typical player controls directly its velocity, Theorem~\ref{main_result} significantly improves \cite[Theorem 4.1]{MR4030259}, where the cost function is assumed to satisfy some strong separability and regularity conditions. Indeed, for every $k\in\I^*$, $x\in \mathcal{S}_k$, and $M\in \M$, the value of $V_k^{M}(x)$ in \eqref{eq:hjb_eq} can be computed by using only the states $y\in \mathcal{S}_{k+1}^{1}(x)$, which is a uniformly bounded set with respect to the discretization parameters. On the other hand, the space grid in \cite{MR4030259}  is unbounded with respect to the discretization parameters. This lack of compactness of the support of the scheme in \cite{MR4030259} was the main issue in establishing the convergence result in  \cite[Theorem 4.1]{MR4030259} for more general cost functions. For instance, the function $[0,T]\times \RR^r\times \RR^d \times \P_1(\RR^d) \ni (t,a,x,\mu)\mapsto \ell(t,a,x,\mu)= \kappa(x)|a|^{p}\in \RR$, with $\kappa: \RR^d \to \RR$ being bounded from below by a strictly positive constant, Lipschitz continuous, and bounded, satisfies {\bf(H1)} but it does not satisfy the assumptions in \cite{MR4030259}.
\end{remark}
\subsection{Proof of Theorem~\ref{main_result}} In order to prove Theorem~\ref{main_result}, we first show some preliminary results. 
\begin{lemma} 
\label{lem:deriv-acotada-sup-xi-n} 
There exists $D_{\infty}>0$, independent of $n$,  such that for any $\gamma\in \supp(\xi^n)$ we have
$$
|\dot{\gamma}(t)|\leq  D_{\infty} \;\;\; \text{for a.e. }  t\in [0,T].
$$
\end{lemma}
\begin{proof} Let $\gamma \in   \supp(\xi^n)$.  Writing $\gamma(t_k^n) =\left(\gamma_1(t_k^n) ,\gamma_2(t_k^n) \right)\in \cR^r\times \cR^{d-r} $, Remark \ref{rem:cota-inf-supp-xi}{\rm(i)} and the definition of $\SS_{k+1}^{1,n}(\gamma(t_{k}^{n}))$ yield
$$
\gamma_1(t_{k+1}^n) =\gamma_1(t_{k}^n)+ \Delta t_n \left[ A_1(t_k^n, \gamma(t_{k}^n)) + B_1(t_k^n, \gamma(t_{k}^n))\alpha(k,\gamma(t_{k}^n),\gamma_1(t_{k+1}^n))  \right],
$$
with $\left|\alpha(k,\gamma(t_{k}^n),\gamma_1(t_{k+1}^n))\right| \leq \widehat{C}$. Thus, from Remark \ref{despues_de_hipotesis}{\rm(i)} and Remark \ref{rem:cota-inf-supp-xi}{\rm(ii)}, we obtain 
$$
\left| \frac{\gamma_1(t^n_{k+1})-\gamma_1(t^n_k)}{\Delta t_n} \right|\leq C_A(1+C_\infty)+C_B \widehat{C},
$$
and, by \eqref{definition_alpha_n}, 
$$
\ba{lll}
\ds \left| \frac{\gamma_2(t^n_{k+1})-\gamma_2(t^n_k)}{\Delta t_n} \right|&\leq&
\ds \left| \frac{\gamma_2(t^n_{k+1})-\ydos(k,\gamma(t_{k}^n),\gamma_1(t_{k+1}^n))}{\Delta t_n} \right| \\[15pt]
\;&\;& +|A_2(t^n_k, \gamma_n(t^n_k)) +B_2(t^n_k, \gamma(t^n_k))\alpha(k,\gamma(t_{k}^n),\gamma_1(t_{k+1}^n))|\\[12pt]
\;&\leq& \ds C_I\frac{\Delta x_n}{\Delta t_n}+C_A(1+C_\infty)+C_B \widehat{C}.
\ea
$$
The result follows from $(\Delta t_n, \Delta x_n)\in \widehat{\Delta}$. 
\end{proof}

By  Remark \ref{rem:cota-inf-supp-xi} and Lemma \ref{lem:deriv-acotada-sup-xi-n}, we obtain the following compactness results.
\begin{lemma}   \label{prop:compactness} The following hold: 
\begin{enumerate}[{\rm(i)}]
\item  The sequence $(\xi^n)_{n\in\NN}$ is a relatively compact subset of $\P(\Gamma)$ endowed with the topology of narrow convergence.

\item For $t\in [0,T]$ and $n\in\NN$ we have
$$
\supp\left(M^n(t)\right)\subset \ov{\mathrm{B}}(0,C_{\infty}).
$$
\item We have that 
$$
d_1\left(M^n(s), M^n(t)\right)\leq D_{\infty} |s-t|\quad\text{for all }s,\,t\in [0,T],\,n\in\NN.
$$
\item The sequence $\left(M^n\right)_{n\in\NN}$ is a relatively compact subset of $C\left([0,T];\P_1(\cR^d) \right)$.
\end{enumerate}
\end{lemma}
\begin{proof}{\rm(i)} Let $\Gamma_{\infty}:= \{ \gamma \in W^{1,\infty}([0,T]; \RR^d) \, | \, \| \gamma\|_{\infty} \leq C_{\infty}, \; \|\dot{\gamma}\|_{\infty} \leq D_{\infty} \}$. Since Remark~\ref{rem:cota-inf-supp-xi}{\rm(ii)} and Lemma~\ref{lem:deriv-acotada-sup-xi-n} imply that $\supp(\xi^n) \subseteq \Gamma_{\infty}$, the result follows from the compactness of $\Gamma_{\infty}$ in $(\Gamma,\|\cdot\|_{\infty})$ and \cite[Proposition 7.1.5]{ambrosio2008gradient}.

{\rm(ii)} This follows from \eqref{rem:def-Mn-cont} and Remark~\ref{rem:cota-inf-supp-xi}{\rm(ii)}.

{\rm(iii)} Set $\Gamma \ni \gamma \mapsto e_{s,t}(\gamma)=(\gamma(s), \gamma(t))\in \RR^d \times \RR^d$. Then, by \eqref{W1_distance} and Lemma~\ref{lem:deriv-acotada-sup-xi-n}, we have that
$$d_1\left(M^n(s), M^n(t)\right) \leq \int_{\RR^d \times \RR^d} |x-y| \dd (e_{s,t}\sharp \xi_n)(x,y)= \int_{\Gamma}|\gamma(s)-\gamma(t)| \dd \xi_{n}(\gamma) \leq D_{\infty} |s-t|.$$
{\rm(iv)} Since $\{\mu \in \P_1(\RR^d) \; | \; \supp(\mu) \subset  \ov{\mathrm{B}}(0,C_{\infty}) \}$ is compact in $\P_1(\RR^d)$ (see e.g. \cite[Proposition 7.1.5]{ambrosio2008gradient}), the result follows from {\rm(ii)}-{\rm(iii)}  and the Arzel\`a-Ascoli theorem. 
\end{proof}

We denote by  $\{v^n_k\;|\;k\in\I^n \}$ and $\{\mathcal{V}^n_k\;|\;k\in\I^n \}$   the family of functions given by \eqref{semidiscrete-scheme} and \eqref{fully_discreto_HJB} when $m=M^n$, respectively.  We will need the following technical result. 
\begin{lemma} 
\label{lem:conv-v-vt} 
Under the assumptions of Theorem~\ref{main_result}, as $n\to \infty$,  the following hold:
\begin{enumerate}[{\rm(i)}]
\item $\max \left\{\left|V^n_k(x)-\mathcal{V}^n_k(x)\right|\; | \; k\in \I^{n}, \;  x\in \SS_k^{n}\right\} \to  0$. \label{convergence_value_function_i} \vspace{0.2cm}
\item  $\max\left\{\left|v^n_k( x)-\mathcal{V}^n_k(x) \right|\;|\;k\in\I^n,\;x\in\SS^n_k\right\} \to 0$.\label{convergence_value_function_ii}
\end{enumerate}
\end{lemma}
\begin{proof}
(i)  For any $k\in \I^{n,*}$ and $x\in\SS^n_k$, we have
\be\label{diferencia_v_k_generica}
\ba{lll}
\left|V^n_k(x)-\mathcal{V}^n_k(x) \right|&\leq& \ds\max_{\mathrm{y}_1\in S^{1,n}_{k+1}(x)}\left| \left(I_{\SS^{2,n}_{k+1}(x,\mathrm{y}_1)}\left[V^n_{k+1}( \mathrm{y}_1,\cdot )\right]-I_{\SS^{2,n}_{k+1}(x,\mathrm{y}_1)}\left[\mathcal{V}^n_{k+1}( \mathrm{y}_1,\cdot )\right]\right)(\ydos(k,x,\mathrm{y}_1))\right|\\[16pt]
\;&\;&\ds+\max_{p\in \P\left(S^{1,n}_{k+1}(x)\right)}\bigg|\eps_n \sum_{\mathrm{y}_1\in S^{1,n}_{k+1}(x)}p(\mathrm{y}_1)\log(p(\mathrm{y}_1))\bigg|. 
\ea
\ee
On the one hand, for every $\mathrm{y}_1\in S^{1,n}_{k+1}(x)$, we have
\be\label{difference_v_k_generic_time}
\ba{l}
\left| \left(I_{\SS^{2,n}_{k+1}(x,\mathrm{y}_1)}\left[V^n_{k+1}(\mathrm{y}_1,\cdot)\right]-I_{\SS^{2,n}_{k+1}(x,\mathrm{y}_1)}\left[\mathcal{V}^n_{k+1}(\mathrm{y}_1,\cdot)\right]\right)(\ydos(k,x,\mathrm{y}_1))\right|\\[12pt]
\hspace{5cm}\leq \ds \sum_{\mathrm{y}_2\in \SS^{2,n}_{k+1}(x,\mathrm{y}_1)}\beta_{\mathrm{y}_2}(\ydos(k,x,\mathrm{y}_1))\left| V^n_{k+1}(\mathrm{y}_1,\mathrm{y}_2)-\mathcal{V}^n_{k+1}(\mathrm{y}_1,\mathrm{y}_2)\right| \\[22pt]
\hspace{5cm}\leq  \max \left\{\left|V^n_{k+1}(y)-\mathcal{V}^n_{k+1} (y)\right|\; | \;    y\in \SS_{k+1}^{n}\right\}.
\ea
\ee
On the other hand,
\be\label{bounding_the_entropy_term}
\ba{lll}
\ds\max_{p\in \P(S^{1,n}_{k+1}(x))}\left|\eps_n \sum_{\mathrm{y}_1\in S^{1,n}_{k+1}(x))}p(\mathrm{y}_1)\log(p(\mathrm{y}_1))\right| &=&-\ds\eps_n \min_{p\in \P\left(S^{1,n}_{k+1}(x)\right)} \sum_{\mathrm{y}_1\in S^{1,n}_{k+1}(x)}p(\mathrm{y}_1)\log(p(\mathrm{y}_1))\\[22pt]
\;&=& \eps_n \log\left(\big|S^{1,n}_{k+1}(x)\big|\right),
\ea
\ee
where $\big|S^{1,n}_{k+1}(x)\big|$ is the cardinal number of   $S^{1,n}_{k+1}(x)$. Since  $S^{1,n}_{k+1}(x)\subset \ov{B}(0,C_{\infty})$ (see Remark \ref{rem:cota-inf-supp-xi}{\rm(ii)}), there exists $C>0$, independent of $n$, such that $|S^{1,n}_{k+1}(x)| \leq C  (N^n_s)^d$. Setting 
$$\delta_k V:=\max \left\{\left|V^n_k(x)-\mathcal{V}^n_k(x)\right|\; |  \;  x\in \SS_k^{n}\right\},$$
relations \eqref{diferencia_v_k_generica}-\eqref{bounding_the_entropy_term} yield  $
\delta_k V \leq \delta_{k+1} V  + (Cd)\eps_n \log(N^n_s)$. The result follows by iterating the previous inequality, the identity $\delta_{N_{t}^n} V =0$, and   the assumption about the sequence $(\eps_n)_{n\in\NN}$.

{\rm(ii)}  By definition $v^n_{N^n_t}(x)=\mathcal{V}^n_{N^n_t}(x)$, for all $x\in\SS^n_{N^n_t}$. Given $k\in \I^{n,*}$, assume that there exists $L^n_{k+1}>0$ such that
\be
\label{eq:lips-k+1}
\left|v^n_{k+1}(x)-\mathcal{V}^n_{k+1}(x)\right|\leq L^n_{k+1}\Delta x_n \quad \text{for all }x\in \SS^n_{k+1}. 
\ee
Given $x=(\mathrm{x}_1, \mathrm{x}_2)\in\SS^n_k$, let $ \mathrm{y}_1^n\in S^{1,n}_{k+1}(x)$ be a minimizer associated with $\mathcal{V}^n_k(x)$ (see \eqref{fully_discreto_HJB}).  Recalling \eqref{definition_alpha_n}, from \eqref{semidiscrete-scheme-dpp_bola} and \eqref{fully_discreto_HJB}, we have
$$
\ba{lll}
v^n_k(x)-\mathcal{V}^n_k(x)&\leq&  v^n_{k+1} ( \mathrm{y}^n_1, \ydos(k,x,  \mathrm{y}^n_1)) -I_{\SS^{2,n}_{k+1}(x,\mathrm{y}^n_1)}[\mathcal{V}^n_{k+1}( \mathrm{y}^n_1, \cdot)]( \ydos(k,x,  \mathrm{y}^n_1)))\\[12pt]
\;&=&\ds\sum_{\mathrm{y}_2\in \SS^{2,n}_{k+1}(x,\mathrm{y}^n_1)}\beta_{\mathrm{y}_2}(\ydos(k,x,  \mathrm{y}^n_1))\left[ v^n_{k+1} ( \mathrm{y}^n_1, \ydos(k,x,  \mathrm{y}^n_1)))- v^n_{k+1}( \mathrm{y}^n_1, \mathrm{y}_2)\right]\\[20pt]
\;&\;&+\ds\sum_{\mathrm{y}_2\in \SS^{2,n}_{k+1}(x,\mathrm{y}^n_1)}\beta_{\mathrm{y}_2}( \ydos(k,x,  \mathrm{y}^n_1)))\left[ v^n_{k+1} ( \mathrm{y}^n_1, \mathrm{y}_2) - \mathcal{V}^n_{k+1} ( \mathrm{y}^n_1, \mathrm{y}_2)\right].
\ea
$$
By \eqref{lem:Lips-v-semidis} and \eqref{support_beta_x} we obtain 
$$
\sum_{\mathrm{y}_2\in \SS^{2,n}_{k+1}(x, \mathrm{y}^n_1)}\beta_{\mathrm{y}_2}( \ydos(k,x,  \mathrm{y}^n_1)))\left[ v^n_{k+1}(\mathrm{y}^n_1,  \ydos(k,x,  \mathrm{y}^n_1)))- v^n_{k+1}(\mathrm{y}^n_1, \mathrm{y}_2)\right] \leq L_v C_I \Delta x_n. 
$$
From \eqref{eq:lips-k+1} we also have that
$$
\sum_{\mathrm{y}_2\in \SS^{2,n}_{k+1}(x, \mathrm{y}^n_1)}\beta_{\mathrm{y}_2}( \ydos(k,x,  \mathrm{y}^n_1)))\left[ v^n_{k+1}(\mathrm{y}^n_1, \mathrm{y}_2) - \mathcal{V}^n_{k+1}(\mathrm{y}^n_1, \mathrm{y}_2)\right]\leq L^n_{k+1}\Delta x_n
$$
and hence
\be\label{v_menos_V}
v^n_k( x)-\mathcal{V}^n_k(x)\leq \left(L^n_{k+1} + L_vC_I \right)\Delta x_n. 
\ee

On the other hand, by Lemma~\ref{lem:controles-dis-acotados} in the Appendix I of this work, there exists $ \bar\alpha^{n}  \in \ov{\mathrm{B}}(0,\widehat{C})$ such that 
$$
v^n_k(x)=   \Delta t_n \ell(t_k^n, \bar\alpha^{n} , x,M^n(t_k^n))+v_{k+1}^n\left(x+\Delta t[A(t_k^n, x)+ B(t_k^n, x) \bar\alpha^{n} ]\right).
$$
Let us set $ c_{B, \infty} =\max\{ |B_1(t,x)^{-1}|  \; | \; t\in [0,T], \, x\in \ov{\mathrm{B}}(0,C_{\infty})\}\in (0,\infty)$ and  $\delta^n=c_{B, \infty}  \Delta x_n/\Delta t_n$. Since $\Delta x_n/ \Delta t_n \to 0$, as $n\to \infty$,  we have that $\delta^n \in (0,\widehat{C})$ for all  $n\in \NN$ large enough. For all those $n\in \NN$, let $ \hat{\alpha}^{n} \in\cR^r$  be  such that 
\be\label{alpha_hat_dentro}\left|\bar\alpha^{n}-\hat\alpha^{n}  \right| \leq \delta^n  \quad \text{and} \quad \left|\hat\alpha^{n} \right| \leq  \widehat{C}-\delta^n.
\ee
Setting 
$$
 \bar{\mathrm{y}}_1^{n}:=\mathrm{x}_1+\Delta t_n\left[ A_1(t^n_k, x)+B_1(t^n_k, x) \bar\alpha^n  \right]\quad \text{and}\quad  \hat{\mathrm{y}}_1^n :=\mathrm{x}_1+\Delta t_n\left[ A_1(t^n_k, x)+B_1(t^n_k, x) \hat\alpha^{n}  \right],
$$
there exists $ \tilde{\mathrm{y}}_1^{n} \in \G^n_r$ such that  $\left|\hat{\mathrm{y}}_1^{n}- \tilde{\mathrm{y}}_1^{n}\right|\leq \Delta x_n$. Defining $ \tilde\alpha^{n}:= \alpha(k, x, \tilde{\mathrm{y}}_1^n)$, from {\bf (H3)} we get
\be\label{alpha_hat_y_alpha_tilde}
\left| \hat\alpha^n -\tilde\alpha^n \right| \leq \left| B_1(t^n_k, x)^{-1}\left[ \frac{\hat{\mathrm{y}}_1^n-\mathrm{x}_1}{\Delta t_n} - \frac{\tilde{\mathrm{y}}_1^n-\mathrm{x}_1}{\Delta t_n}\right]\right|= \left| B_1(t^n_k, x)^{-1} \left[ \frac{\hat{\mathrm{y}}_1^n-\tilde{\mathrm{y}}_1^n}{\Delta t_n}\right] \right| \leq \delta^n.
\ee
Thus, \eqref{alpha_hat_dentro} and \eqref{alpha_hat_y_alpha_tilde} imply that    
\be\label{diferencia_controles_tilde}
\tilde{\alpha}^n \in \ov{\mathrm{B}}(0,\widehat{C}) \quad \text{and} \quad \left| \bar\alpha^{n}-\tilde\alpha^{n}\right|\leq \left| \bar\alpha^{n}-\hat\alpha^{n}\right|+\left| \hat\alpha^{n}-\tilde\alpha^{n}\right| \leq 2\delta^n. 
\ee
In particular,
\be\label{diferencia_y_s}
\tilde{\mathrm{y}}_1^{n}\in S^{1,n}_{k+1}(x) \quad \text{and} \quad  
 \left|\bar{\mathrm{y}}_1^n-\tilde{\mathrm{y}}_1^n \right|  \leq 2 C_B\Delta t_n  \delta^n= 2C_{B}c_{B,\infty} \Delta x_n,
\ee
where, to obtain the inequality,  we have used that $ \tilde{\mathrm{y}}_1^{n}:=\mathrm{x}_1+\Delta t_n\left[ A_1(t^n_k, x)+B_1(t^n_k, x) \tilde\alpha^n  \right]$ and  {\bf (H3)}(ii).
 
Similarly, setting  $\bar{\mathrm{y}}_2^n:=\mathrm{x}_2+\Delta t_n [A_2(t^n_k, x)+B_2(t^n_k, x)\bar\alpha^n]$ and $\tilde{\mathrm{y}}_2^n:=\mathrm{x}_2+\Delta t_n [A_2(t^n_k, x)+B_2(t^n_k, x)\tilde\alpha^n]$, \eqref{diferencia_controles_tilde} yields 
\be\label{diferencia_de_los_y_2_tildas}
\left|\bar{\mathrm{y}}_2^n-\tilde{\mathrm{y}}_2^n \right|  \leq 2C_{B}c_{B,\infty} \Delta x_n.
\ee

The optimality of $\bar\alpha^n$ implies
\be\label{V_menos_v}
\ba{lll}
\mathcal{V}^n_k( x)-v^n_k( x)&\leq& \Delta t_n\left[ \ell(t^n_k, \tilde\alpha^n, x, M^n(t^n_k))-\ell(t^n_k, \bar\alpha^n, x, M^n(t^n_k))\right]\\[6pt]
\;&\;&+ I_{\SS^{2,n}_{k+1}(x,\tilde{\mathrm{y}}^n_1)}[\mathcal{V}^n_{k+1}(\tilde{\mathrm{y}}_1^n, \cdot)](\tilde{\mathrm{y}}_2^n) - v^n_{k+1} (\bar{\mathrm{y}}_1^n, \bar{\mathrm{y}}_2^n).
\ea
\ee
Since  $\tilde\alpha^n, \;\bar\alpha^n\in \ov{\mathrm{B}}(0, \widehat{C})$, Remark~\ref{compact_support_uniform}, {\bf(H1)}(iii), and \eqref{diferencia_controles_tilde} imply the existence of $\Lb>0$ such that 
\be\label{diferencia_eles}
\left|\ell(t^n_k, \tilde\alpha^n, x, M^n(t^n_k))-\ell(t^n_k, \bar\alpha^n, x, M^n(t^n_k))\right|\leq \Lb |\tilde\alpha^n-\bar\alpha^n|\leq 2 \Lb  \delta^n .
\ee
Using \eqref{eq:lips-k+1}, \eqref{diferencia_y_s}, \eqref{diferencia_de_los_y_2_tildas}, and \eqref{support_beta_x}, we also have
\be\label{diferencia_I_V_y_v}
\ba{l}
 I_{\SS^{2,n}_{k+1}(x,\tilde{\mathrm{y}}^n_1)}[\mathcal{V}^n_{k+1}(\tilde{\mathrm{y}}_1^n, \cdot)](\tilde{\mathrm{y}}_2^n) - v^n_{k+1}(\bar{\mathrm{y}}_1^n, \bar{\mathrm{y}}_2^n)\\[6pt]
 \hspace{3.3cm}=\ds\sum_{\mathrm{y}_2\in \SS^{2,n}_{k+1}(x,\tilde{\mathrm{y}}^n_1)}\beta_{\mathrm{y}_2}(\tilde{\mathrm{y}}_2^n)\left[ \mathcal{V}^n_{k+1}(\tilde{\mathrm{y}}_1^n, \mathrm{y}_2) - v^n_{k+1}(\tilde{\mathrm{y}}_1^n, \mathrm{y}_2)\right]\\[22pt]
 \hspace{3.6cm}+\ds\sum_{\mathrm{y}_2\in \SS^{2,n}_{k+1}(x,\tilde{\mathrm{y}}^n_1)}\beta_{\mathrm{y}_2}(\tilde{\mathrm{y}}_2^n)\left[ v^n_{k+1}(\tilde{\mathrm{y}}_1^n, \mathrm{y}_2)-v^n_{k+1}(\bar{\mathrm{y}}_1^n, \bar{\mathrm{y}}_2^n)\right]\\[22pt]
 \hspace{3.3cm}\leq   L^n_{k+1}\Delta x_n +L_{v}   \left( |\tilde{\mathrm{y}}_{1}^{n}-\bar{\mathrm{y}}_{1}^{n} | + |\tilde{\mathrm{y}}_{2}^{n}-\bar{\mathrm{y}}_{2}^{n} |  +\sum_{\mathrm{y}_2\in \SS^{2,n}_{k+1}}\beta_{\mathrm{y}_2}(\tilde{\mathrm{y}}_2^n)| \mathrm{y}_{2}-\tilde{\mathrm{y}}_{2}^{n}|\right)\\[10pt]
 \hspace{3.3cm} \leq \left( L^n_{k+1} + L_v\left( 4C_{B}c_{B,\infty} +C_{I} \right) \right)\Delta x_n.
\ea
\ee

Therefore, using \eqref{v_menos_V} and \eqref{V_menos_v}-\eqref{diferencia_I_V_y_v}, we deduce that 
\be\label{diferencia_v_n_k}
|v^n_k (x)-\mathcal{V}^n_k (x)|\leq L^n_k \Delta x_n, 
\ee
where  
$$
L^n_k=  L^n_{k+1} + L_v C_I+ 4 L_v  C_{B}c_{B,\infty}   + 2\ov{L}c_{B,\infty} .
$$
Since $L^n_{N^n_t}=0$, we deduce that $L_k^n \leq N^n_{t}(L_v C_I+ 2 L_v  C_{B}c_{B,\infty}   + 2\ov{L}c_{B,\infty})$, for all $k\in \I^n$,  and hence \eqref{diferencia_v_n_k} yields the existence of $C>0$, independent of $n$,  such that  
$$
|v^n_k (x)-\mathcal{V}^n_k (x)|\leq C \Delta x_n/ \Delta t_n \quad \text{for all }k\in \I^n,\, x\in \SS_k^n,
$$
and the result follows from the assumption $\Delta x_n / \Delta t_n \to 0$ as $n\to \infty$.
\end{proof}

We have now all the elements to prove Theorem \ref{main_result}.
\begin{proof}[Proof of Theorem~\ref{main_result}]  
{\rm(i)} By Lemma~\ref{prop:compactness}{\rm(i)},  there exists $\xi^*\in\P(\Gamma)$ such that,  as $n\to \infty$ and up to some subsequence, $\xi^n \to \xi^*$ in $\P(\Gamma)$. Let $\gamma \in \supp(\xi^*)$ and let us show that $\gamma$ is a.e. differentiable in $[0,T]$ with 
\be
\label{edo_gamma_soporte}
\dot{\gamma}(t)= A(t, \gamma(t)) + B(t,\gamma(t))\alpha(t) \quad \text{for a.e. } t\in [0,T]
\ee
for some $\alpha \in L^{\infty}([0,T];\RR^r)$ with $\|\alpha\|_{\infty}\leq \widehat{C}$. Indeed, by \cite[Proposition 5.1.8]{ambrosio2008gradient} there exists a sequence $\gamma^{n} \in \supp(\xi^n)\subseteq \Gamma^{n} $ such that $\gamma^n \to \gamma$ uniformly in $[0,T]$. Arguing as in the proof of Lemma~\ref{lem:conv-v-vt}{\rm(ii)}, we have that 
\be
\label{integral_discreta}
\gamma^{n}(t^{n}_{k+1}) =  \gamma^{n}(0) + \Delta t_n\sum_{j=0}^{k-1}[ A(t_{j}^n, \gamma^{n}(t_j^n)) + B(t_{j}^n, \gamma^{n}(t_j^n))\alpha^n(t_{j}^n)] + O(\Delta x_n/\Delta t_n), 
\ee
where $[0,T)\ni t\mapsto \alpha^n(t)\in \ov{\mathrm{B}}(0,\widehat{C})$ is defined by 
$$\alpha^n(t)= 
B_1(t^n_k,\gamma^n(t_k^n))^{-1}\left[\frac{\gamma^n(t_{k+1}^n)-\gamma^n(t_{k}^n)}{\Delta t_n} -A_1(t_k^n,\gamma^n(t_k^n)) \right] \quad \text{for all } t\in [t_{k}^{n}, t_{k+1}^n).
$$
Denote by $\omega_{A}$ and $\omega_{B}$ two modulus of continuity associated with $A$ and $B$ on $[0,T]\times  \ov{\mathrm{B}}(0,C_{\infty})$. By Lemma~\ref{lem:deriv-acotada-sup-xi-n}  and \eqref{integral_discreta}, we obtain
\be
\label{edo_antes_del_limite}
\ba{rcl}
\gamma^{n}(t^{n}_{k+1}) &=&  \ds \gamma^{n}(0) + \int_{0}^{t^{n}_{k+1}} \left[ A(s, \gamma^{n}(s)) + B(s, \gamma^{n}(s))\alpha^n(s)\right]\dd s \\[10pt]
\; & \; & \ds+ O(\omega_{A}\left(\Delta t_n)+\omega_{B}(\Delta t_n)+ \Delta t_n +\Delta x_n/\Delta t_n\right).
\ea
\ee
Since $\|\alpha^n\|_{\infty}\leq \widehat{C}$ for all $n\in \NN$, there exists $\alpha\in L^{\infty}([0,T];\RR^r)$, with $\|\alpha\|_{\infty}\leq \widehat{C}$, such that, up to some subsequence, $\alpha^n \to \alpha$ in the weak* topology. Let us fix $t\in [0,T)$ and let $k(n)\in \I^{n,*}$ be such that $t \in [t_{k(n)}^n, t_{k(n+1)}^n)$. Taking $k=k(n)$ in \eqref{edo_antes_del_limite} and passing to the limit yields
$$
\gamma(t)= \gamma(0)+ \int_{0}^{t}\left[A(s, \gamma(s)) + B(s,\gamma(s))\alpha(s)\right]\dd s,
$$
which implies that  \eqref{edo_gamma_soporte} holds.  

Since $\xi^n \to \xi^*$ in $\P(\Gamma)$, setting $m^*(t):=e_t\sharp \xi^*$ for  $t\in [0,T]$, we deduce that  $\left(M^n(t)\right)_{n\in\NN}$, defined in \eqref{rem:def-Mn-cont}, narrowly converges to $m^*(t)\in \P(\cR^d)$ for every $t\in [0,T]$ and hence, by Lemma~\ref{prop:compactness}{\rm(iv)}, we get that $(M^n)_{n\in \NN}$ converges to $m^*$ in $C\left([0,T];\P_1(\cR^d)\right)$. Moreover, the previous convergence,  Lemma~\ref{prop:compactness}{\rm(ii)},   and \cite[Proposition 5.1.8]{ambrosio2008gradient} imply that 
\be\label{soporte_m_*} \supp(m^*(t)) \subset \ov{\mathrm{B}}(0,C_{\infty}) \quad \text{for all $t\in [0,T]$}.
\ee
Let us now show that $\xi^*$ is a MFG equilibrium,  which is equivalent to prove that 
$$
  v(0,\gamma(0))=\int_0^T \ell(t,\alpha(t), \gamma(t), m^*(t))\dd t+g(\gamma(T), m^*(T)) \quad  \text{for $\xi^*$-a.e. $\gamma\in\Gamma$} ,
$$
where, setting $\gamma(t)=(\mathrm{\gamma}_1(t), \mathrm{\gamma}_2(t)) \in \RR^r \times \RR^{d-r}$, 
\be  
\label{eq:main-alpha}
\alpha(t):=B_1(t,\gamma(t))^{-1}\left[\dot{\gamma}_1(t)-A_1(t,\gamma(t)) \right] \quad \text{for a.e. $t \in [0,T]$}. 
\ee

Given $\xi\in \P(\Gamma)$ and a Borel function $\phi\colon  \Gamma \to \RR$, we set $\EE_{\xi}(\phi)= \int_{\Gamma} \phi(\gamma) \dd \xi (\gamma)$, provided that the integral exists.  From the definition of $\xi^n$, we have
$$   
\ba{rcl}
\sum_{x\in\SS^n_0}V_0^n(x)M_0^n(x)&=& \EE_{\xi^n} \bigg( \Delta t_n \ds \sum_{k=0}^{N^n_t-1}\bigg[\ell(t_k, \alpha(k,\gamma(t_k), \gamma_1(t_{k+1})),\gamma(t_k), M^n(t_k))  \\[6pt]
\; & \; & \hspace{2.5cm}+\eps_n \log \big(p_k^{M^n}(\gamma(t_k),\gamma_1(t_{k+1}))\big) \bigg] + g(\gamma(T), M^n(T))\bigg),
\ea
$$
where we recall that for all $k\in \I^{n,*}$, $x\in \SS^{n}_{k}$, and $\mathrm{y}_1 \in \RR^{r}$, $\alpha(k,x,\mathrm{y}_1)$ is given by \eqref{definition_alpha_n} and for the sake of simplicity we denote $t_k$ instead of $t^n_k$ for $k\in\I^n$. Then the assumption on the sequence $(\eps_n)_{n\in\NN}$  yields 
\be\label{eq:valor-esperanza-n}
\sum_{x\in\SS^n_0}V_0^n(x)M_0^n(x)= \EE_{\xi^n} \bigg( \Delta t_n \ds \sum_{k=0}^{N^n_t-1} \ell(t_k, \alpha(k,\gamma(t_k), \gamma_1(t_{k+1})),\gamma(t_k), M^n(t_k)) + g(\gamma(T), M^n(T))\bigg) + o(1),
\ee
where $o(1)\to 0$, as $n\to \infty$. Combining  Proposition~\ref{prop:conv-value-function} and  Lemma~\ref{lem:conv-v-vt}, we have that 
\be  \label{eq:convV0-M0}
\sum_{x\in\SS^n_0}V_0^n(x)M_0^n(x) \underset{n \to \infty}{\longrightarrow} \int_{\cR^d} v^*(0,x)\dd m_0(x)=\EE_{\xi^*}\left[v^*(0,\gamma(0))\right],
\ee
where $v^*$ is the value function \eqref{value_function_dependent_on_m} associated with $m^*$. Let us define  $L\colon [0,T]\times \cR^d \times \cR^d \to \RR$ as
$$
L(t,x,y):=\ell\left(t, B_1(t,x)^{-1}[\mathrm{y}_1-A_1(t,x)],x,m^{*}(t)\right) \quad \text{for all }(t,x,y)\in [0,T]\times \cR^d \times \cR^d,
$$
where we recall the decomposition $y=(\mathrm{y}_1, \mathrm{y}_2) \in \RR^{r}\times \RR^{d-r}$. Notice that $L$ is convex with respect to its third variable and, by {\bf(H1)}{\rm(i)},
$$
L(t,x,y)\geq -C_{\ell} \quad  \text{for all }(t,x,y)\in [0,T]\times \cR^d \times \cR^d.
$$

Using that $\int_0^T \ell(t,\alpha(t), \gamma(t), m^*(t))\dd t=\int_0^T L(t, \gamma(t), \dot{\gamma}(t))\dd t$ for every $\gamma\in W^{1,p}([0,T];\cR^d)$, where $\alpha$ is given by \eqref{eq:main-alpha}, we deduce from  \cite[Corollary 3.24]{dacorogna89} that the map
$$
W^{1,p}([0,T];\cR^d)  \ni \gamma   \mapsto \int_0^T \ell(t,\alpha(t), \gamma(t), m^*(t))\dd t \in \RR,
$$
is  lower semicontinuous. Thus, by \cite[Lemma 5.1.7]{ambrosio2008gradient}, we have
\be \label{ineq:El-semicont}
\EE_{\xi^*}\left[\int_0^T \ell(t,\alpha(t), \gamma(t), m^*(t))\dd t\right]\leq \liminf_{n\to\infty}\;\EE_{\xi^n}\left[ \int_0^T \ell(t,\alpha(t), \gamma(t), m^*(t))\dd t\right].
\ee 

Let us show that
\be \label{eq:dif-esperanza-n-ell}
\EE_{\xi^n}\Bigg[\Bigg| \int_0^T \ell(t,\alpha(t), \gamma(t), m^*(t))\dd t- \Delta t_n\sum_{k=0}^{N^n_t-1}\ell(t_k, \alpha(k, \gamma(t_k), \gamma_1(t_{k+1})),\gamma(t_k), M^n(t_k))\Bigg|\Bigg]
\ee
tends to zero as $n\to \infty$. Since the set $[0,T] \times \ov{\mathrm{B}}(0,\widehat{C})\times \ov{\mathrm{B}}(0,C_{\infty})\times \P(  \ov{\mathrm{B}}(0,C_{\infty}))$ is compact, by {\bf(H1)} there exists a  nondecreasing modulus of continuity $\omega_{\ell}\colon [0,\infty) \to [0,\infty)$ such that $\lim_{\tau \downarrow 0} \omega_{\ell}(\tau)=0$ and
$$
\left|\ell(t,a,x, \mu)-\ell(t', a', x',  \mu')\right| \leq \omega_{\ell}\left( |t-t'| + |a-a'| +|x-x'| + d_{1}(\mu,\mu') \right) 
$$
for all $t,\, t'\in [0,T]$, $a, \, a'\in \ov{\mathrm{B}}(0, \widehat{C})$, $x, \, x' \in \ov{\mathrm{B}}(0, C_{\infty})$, and $\mu, \, \mu' \in \P(\ov{\mathrm{B}}(0, C_{\infty}))$. In particular,  since $\omega_\ell$ is nondecreasing, Remark~\ref{rem:cota-inf-supp-xi}{\rm (ii)}, Lemma~\ref{prop:compactness}{\rm(ii)}, \eqref{soporte_m_*},  and Lemma~\ref{lem:deriv-acotada-sup-xi-n} imply that  for all $\gamma\in \supp(\xi^n)$, $k\in \I^{n,*}$, and $t\in [t_k,t_{k+1}]$, we have that
\be\label{diferencia_de_las_ells}\ba{l}
\left|\ell(t,\alpha(t), \gamma(t), m^*(t))-\ell(t_k, \alpha(k,\gamma(t_k), \gamma_1(t_{k+1})),\gamma(t_k), M^n(t_k))\right| \\[6pt]
\hspace{3cm}\leq \omega_{\ell}\bigg(\Delta t_n +|\alpha(t)- \alpha(k,\gamma(t_k), \gamma_1(t_{k+1}))|+ D_{\infty} \Delta t_n +d_{1}(M^{n}(t_k),m^*(t))\bigg).
\ea
\ee 
 Let $t\in (t_{k},t_{k+1})$,  $\gamma=(\gamma_1, \gamma_2) \in \supp(\xi^n)$,  and recall that 
$\dot{\gamma}_1(t)= (\gamma_1(t_{k+1})-\gamma_1(t_{k}))/\Delta t_n$. By \eqref{eq:main-alpha}, we have that  
\be\label{estimacion_diff_alphas}
\ba{l}
\left|\alpha(t)- \alpha(k,\gamma(t_k), \gamma_1(t_{k+1})) \right|\\[6pt]
\hspace{1cm}\leq\left| B_1(t, \gamma(t))^{-1}\left[ \dot{\gamma}_1(t)-A_1(t,\gamma(t))\right]-B_1(t_k, \gamma(t_k))^{-1}\left[ \dot{\gamma}_1(t)-A_1(t_k,\gamma(t_k))\right]\right|\\[6pt]
\hspace{1cm}\leq \left| B_1(t,\gamma(t))^{-1}\right|\left| A_1(t,\gamma(t))-A_1(t_k, \gamma(t_k))\right|\\[6pt]
\hspace{1.3cm}+\left|B_1(t,\gamma(t))^{-1}-B_1(t_k,\gamma(t_k))^{-1} \right|\left| \dot{\gamma}_1(t)-A_1(t_k,\gamma(t_k))\right|.
\ea
\ee
Since $A_1$ and $B_1^{-1}$ are continuous on $[0,T]\times \ov{\mathrm{B}}(0,C_{\infty})$, there exists a nondecreasing function $\ov{\omega} \colon [0,\infty) \to [0,\infty)$  such that $\lim_{\tau \downarrow 0} \omega (\tau)=0$ and 
$$
|A_{1}(t,x)- A_{1}(t',x')|   + |B_1(t,x)^{-1}- B_1(t',x')^{-1}| \leq \ov{\omega} (|t-t'| + |x-x'|) 
$$
for all $t, \, t' \in [0,T]$, and $x, \, x'\in \ov{\mathrm{B}}(0,C_{\infty})$.
Let $c_{B, \infty} =\max\{ |B(t,x)^{-1}|  \; | \; t\in [0,T], \, x\in \ov{\mathrm{B}}(0,C_{\infty})\}\in (0,\infty)$. By \eqref{estimacion_diff_alphas}, Lemma~\ref{lem:deriv-acotada-sup-xi-n}, and Remark \ref{despues_de_hipotesis}{\rm(i)},  we have 
$$
\left|\alpha(t)- \alpha(k,\gamma(t_k), \gamma_1(t_{k+1})) \right|  \leq  \left(c_{B, \infty}+D_{\infty} + C_{A}(1+C_{\infty}) \right)\ov{\omega}\big((1+ D_{\infty}) \Delta t_n\big).
$$
By \eqref{diferencia_de_las_ells}, \eqref{estimacion_diff_alphas}, and the convergence of $M^n$ to $m^*$ in $C([0,T]; \P_1(\RR^d))$, we conclude that the term in \eqref{eq:dif-esperanza-n-ell} tends to $0$ as $n\to \infty$.  

Since, as $n\to\infty$, $\xi^n\to\xi^*$ in $\P(\Gamma)$ and $M_{n}(T)\to m^*(T)$ in $\P(\ov{\mathrm{B}}(0,C_{\infty}))$, the continuity of $g$ yields 
\be \label{eq:g-semicont}
\EE_{\xi^*} \left[g(\gamma(T), m^*(T))\right]=\lim_{n\to\infty}\EE_{\xi^n}\left[ g(\gamma(T), M^n(T))\right].
\ee
Combining  \eqref{eq:valor-esperanza-n}, \eqref{eq:convV0-M0}, \eqref{ineq:El-semicont}, and \eqref{eq:g-semicont}, we conclude that
\be  \label{ineq-esperanzas}
\EE_{\xi^*}\left[\int_0^T \ell(t,\alpha(t), \gamma(t), m^*(t))\dd t+ g(\gamma(T), m^*(T)) \right]\leq \EE_{\xi^*}\left[v^*(0,\gamma(0))\right].
\ee
Furthermore, by the definition of $v^*$, for any $\gamma\in W^{1,p}([0,T];\RR^d)$ we have
$$
v^*(0,\gamma(0))\leq \int_0^T \ell(t,\alpha(t), \gamma(t), m^*(t))\dd t+ g(\gamma(T), m^*(T)).
$$
This inequality and \eqref{ineq-esperanzas} imply that for $\xi^*$-a.e. $\gamma\in\Gamma$, 
$$
v^*(0,\gamma(0))= \int_0^T \ell(t,\alpha(t), \gamma(t), m^*(t))\dd t+ g(\gamma(T), m^*(T)),
$$
which implies that $\xi^*$ is a MFG equilibrium.\smallskip\\
{\rm(ii)} The convergence in $C([0,T]; \P_1(\RR^d))$ of $(M^{n_l})_{l \in \NN}$ towards $m^*$  has already been shown in {\rm(i)}, while relation \eqref{convergencia_value_function} follows from {\rm(i)}, Proposition~\ref{prop:conv-value-function}, and  Lemma~\ref{lem:conv-v-vt}. 
\end{proof}
\section{Numerical results}
\label{sec:numerical_result}
In this section, we consider two instances of Problem~\ref{mfg_problem} with monotone couplings. The first one deals with a one-dimensional problem and a nonlinear dynamics with respect to the space variable. In second instance, we approximate the MFG equilibrium of  a two-dimensional problem in which a typical agent controls its acceleration (see e.g.~\cite{MR4102464, MR4132067}).

Given some discretization parameters $(\Delta t,\Delta x)\in (0,\infty)^2$, and $\varepsilon>0$, we approximate the solutions to the corresponding instances of Problem~\ref{MFG-SL_discrete_scheme_n-degenerate} by using the fictitious play algorithm recalled at the end of Section~\ref{sec:main_result}. More precisely, in both examples we consider the following routine:

\smallskip
\begin{algorithm}[H]
\SetAlgoLined
\KwData{ $\mathsf{M}^0\in\M$ and a tolerance parameter $\delta>0$}
$e\leftarrow \delta +1$\\[3pt]
$n\leftarrow 1$\\[3pt]
$\ov{\mathsf{M}}^1\leftarrow \mathsf{M}^0$

\While{$e>\delta$}{\vspace{0.2cm}
$\mathsf{M}^{n+1}= {\bf br}(\ov{\mathsf{M}}^n)$\\[3pt]
$e\leftarrow  |\mathsf{M}^{n+1}-\ov{\mathsf{M}}^n|_{L^1}$\\[3pt]
$\ov{\mathsf{M}}^{n+1}= \frac{n}{n+1}\ov{\mathsf{M}}^n+\frac{1}{n+1}\mathsf{M}^{n+1}$\\[3pt]
$n\leftarrow n+1$ 
}
\Return{$\ov{\mathsf{M}}^{n-1}$}
 \caption{The algorithm stops at iterate $n$ if $|\ov{\mathsf{M}}^{n}-{\bf br}(\ov{\mathsf{M}}^{n})|_{L^1}\leq \delta$.}
 \label{alg:FP}
\end{algorithm}
\vspace{0.5cm}
In~\eqref{alg:FP}, for every $n\in\NN$, we have set
$$
|\mathsf{M}^{n+1}-\ov{\mathsf{M}}^n|_{L^1}= \frac{1}{N_t+1} \sum_{k=0}^{N_t}\sum_{x\in \mathcal{S}_{k}} |\mathsf{M}_k^{n+1}(x)-\ov{\mathsf{M}}_k^n(x)|.
$$
Inspired by the {\it modified fictitious play} iterates introduced in~\cite[Section~4.2]{MR4030259}, the approximated equilibrium $\ov{\mathsf{M}}^{n-1}$ returned by Algorithm~\ref{alg:FP}, for a given tolerance parameter $\delta>0$, will be used as an initial guess $\mathsf{M}^0$ for a subsequent application of the same algorithm with a smaller tolerance parameter. In this manner, we obtain better numerical performances, in terms of speed of convergence, than fixing a small tolerance  parameter and applying Algorithm~\ref{alg:FP} only once.  
\subsection{Example 1}
\label{subsection_example_1}
We take $T=1$ as time horizon,  
\be 
\label{def_A_B_example_1}
A(t,x)=-2x-\sin(x),\quad \text{and}\quad B(t,x)=1 \quad \text{for all } (t,x)\in (0,T)\times \RR
\ee
in~\eqref{ecuacion_controlada}, which, for $\alpha\in L^{2}([0,T];\RR)$, yields the controlled dynamics
\be 
\label{modified_descent_gradient_dynamics}
\dot{\gamma}(t)=-2\gamma(t)-\sin(\gamma(t))+\alpha(t) \in\RR \quad  \text{for a.e. } t\in (0,1).
\ee
Notice that in the uncontrolled case $\alpha\equiv 0$, \eqref{modified_descent_gradient_dynamics} is a gradient descent dynamics associated with the strongly convex functional $\RR\ni x\mapsto h(x) = x^2-\cos(x)\in \RR$. Therefore, in this uncontrolled case, for any initial condition $x\in\RR$,  $\gamma(t)$ approaches to $0$, the unique solution to $x= -\sin(x)/2$,  exponentially fast as $t\to\infty$. The cost function is decomposed as in {\bf(H5)}{\rm(i)}. More precisely, let $\sigma>0$ and set 
\be 
\label{rho_sigma}
\rho_\sigma(x):= \frac{1}{\sqrt{2\pi}\sigma}e^{-x^2/2\sigma^2} \quad \text{for all }x\in\RR.
\ee
Given $\theta_1$, $\theta_2\geq 0$,  define 
\be 
\ba{rcl}
 \ds f(t,x,\mu)&=&\theta_1 (\rho_\sigma\star \mu)(x) \quad \text{for all } (t,x,\mu)\in [0,T]\times\RR\times \P_1(\RR), \\[6pt]
\ds \ell(t,a,x,\mu) &=& \ds \frac{|a|^2}{2}+ f(t,x,\mu) \quad \text{for all } (t,a,x,\mu)\in [0,T]\times\RR\times\RR\times \P_1(\RR), \\[6pt] 
\ds g(x,\mu)&=&\theta_2 (\rho_\sigma\star \mu)(x) \quad \text{for all } (x,\mu)\in\RR\times \P_1(\RR). 
\ea
\ee
We consider an absolutely continuous initial distribution $m_0\in\P_1(\RR)$ given by
$$
\dd m_0(x)= \mathbb{I}_{[-1,1]}(x)\frac{e^{-x^2/0.04}}{\int_{-1}^1 e^{-y^2/0.04}\dd y}\dd x,
$$
where $\mathbb{I}_{[-1,1]}(x)=1$ if $x\in[-1,1]$ and $\mathbb{I}_{[-1,1]}(x)=0$, otherwise.  

In this framework, a typical agent with an initial position in the interval $[-1,1]$ aims to approach the target position $\bar{x}=0$, while being averse to crowded areas. The above dataset satisfies {\bf(H1)}-{\bf(H4)}, {\bf(H5)}{\rm(i)} and, arguing as in the proof of Lemma~\ref{monotonia_ej_2} below,  {\bf(H5)}{\rm(ii)} also holds. Moreover, for any $m\in C([0,T];\P_1(\RR))$, the associated value function $v$, given by~\eqref{value_function_dependent_on_m}, is Lipschitz with respect to the state variable and, hence, $v(0,\cdot)$ is differentiable almost everywhere. Thus, since $m_0$ is absolutely continuous with respect to the Lebesgue measure, \cite[Theorem~7.4.20]{MR2041617} implies that \eqref{unicidad_casi_segura} holds. Therefore, under the assumptions on the discretization parameters in~Theorem~\ref{main_result},  the solutions to~Problem~\ref{MFG-SL_discrete_scheme_n-degenerate}   converge to the unique solution to Problem~\ref{mfg_problem}.

In our numerical tests we take  $\sigma=0.03$ and we set the following discretization parameters:
$$
\Delta t = 1/30, \quad \Delta x = 1/150, \quad \text{and}\quad \eps = 0.002.
$$
We provide in Table~\ref{tabla1} the results for two values of $(\theta_1,\theta_2)$, which determine the degree of aversion to crowded areas. In order to initialize our numerical tests, we take the constant time marginals  $\mathsf{M}^0_k=M_0$, for all $k=0, \dots, 30$, where $M_0$ is the discretization of the initial distribution $m_0$ described at the beginning of  Section~\ref{sec:main_result} (see \eqref{def-br}).  We run our tests for different choices of the tolerance  parameter  $\delta$. We take first $\delta = 0.1$, and then, in order to improve the speed of convergence, we use the distribution returned by Algorithm~\ref{alg:FP} as the initial distribution for the next tolerance parameter  $\delta =0.01$. Finally, we proceed in the same way for the final tolerance parameter  $\delta = 0.001$.  In the following table, we provide the number of iterations needed for attaining the prescribed tolerances for the two considered values of $(\theta_1,\theta_2)$.
\begin{table}[h]
\begin{center}
\vspace{0.1cm}
\begin{tabular}{|c|c|c|}
	\hline
\; &\quad {\small Test 1: $\theta_1= 1$, $\theta_2 = 0$} \quad &\quad {\small Test 2: $\theta_1= 1$, $\theta_2 = 1$ }\\[2pt]
	\hline
Tolerance & No. of iterations & No. of iterations \\[2pt]
\hline
$\delta =0.1$ &$n= 9$ & $n=9$\\
$\delta =0.01$ & $n=9$& $n=11$\\
$\delta = 0.001$ & $n=9$ & $n=9$\\
\hline
\end{tabular}
\end{center}
\caption{Number of iterations to obtain the desired accuracies.}
\label{tabla1}
\end{table}

In Figure~\ref{two_sets_of_theta_example_1},  we display the returned distributions in both tests for the smallest tolerance parameter $\delta = 0.001$. As expected, the presence of the term penalizing crowded areas at the final time plays an important role in the final distribution of the agents. Finally, we show in Figure~\ref{fictitious_play_working} the results for the first four iterations of Algorithm~\ref{alg:FP}, with $(\theta_1,\theta_2)=(1,1)$ and tolerance parameter $\delta =0.1$, which provides an insight on how the the fictitious play method learns the equilibrium.
\begin{figure}
\centering
\begin{tabular}{cccc}
\subfloat[Test 1: $(\theta_1,\theta_2)=(1,0)$.]{\includegraphics[width=.5\textwidth]{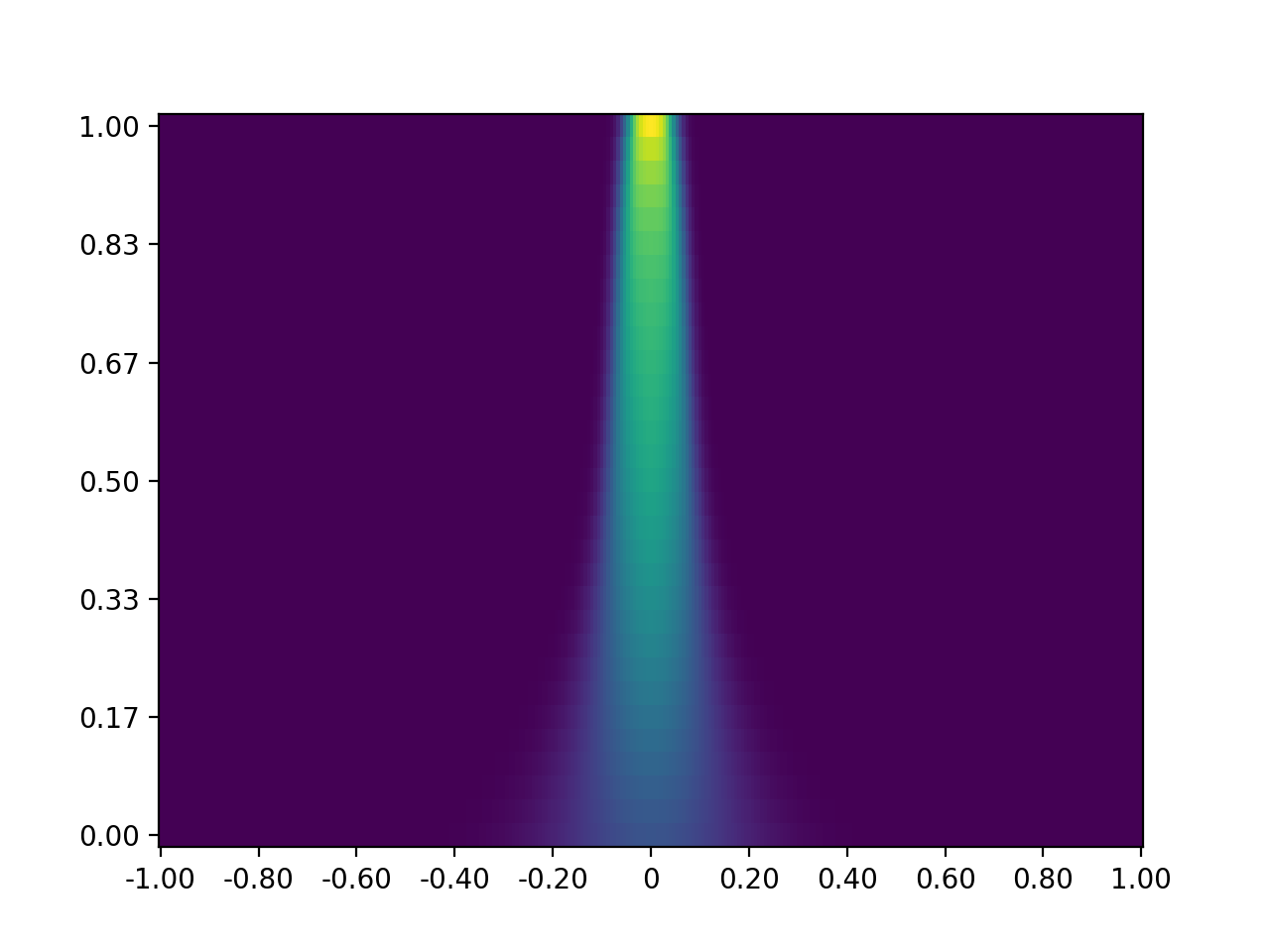}}& 
\subfloat[Test 2: $(\theta_1,\theta_2)=(1,1)$.]{\includegraphics[width=.5\textwidth]{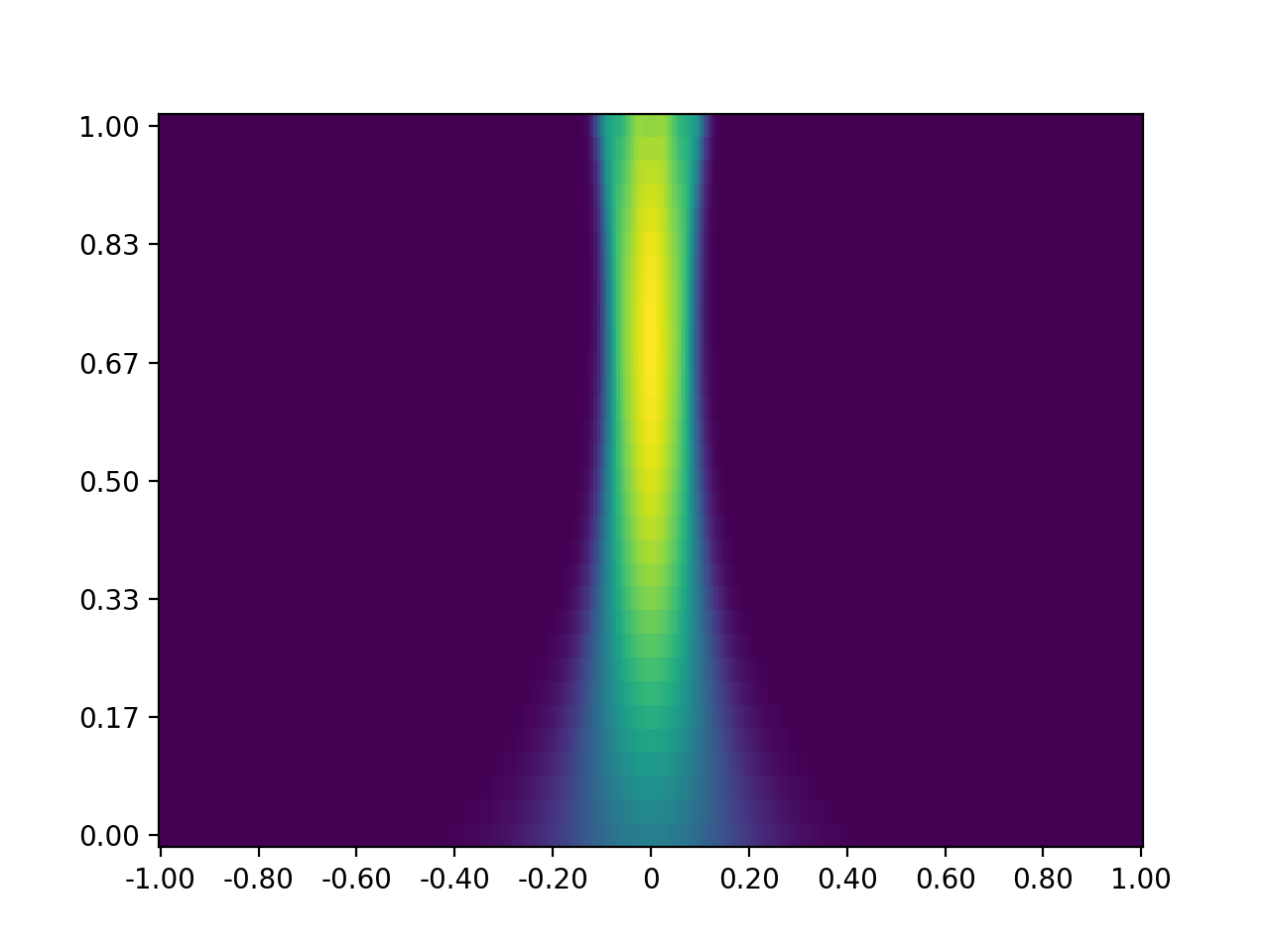}}
\end{tabular}
\caption{Distributions obtained for two values of $(\theta_1,\theta_2)$ and a tolerance parameter $\delta=0.001$.  The abscissa and ordinate axes represent position and time, respectively.}
\label{two_sets_of_theta_example_1}
\end{figure}

\begin{figure}
\centering
\begin{tabular}{cccc}
\subfloat[$\ov{\mathsf{M}}^1$]{\includegraphics[width=.3\textwidth]{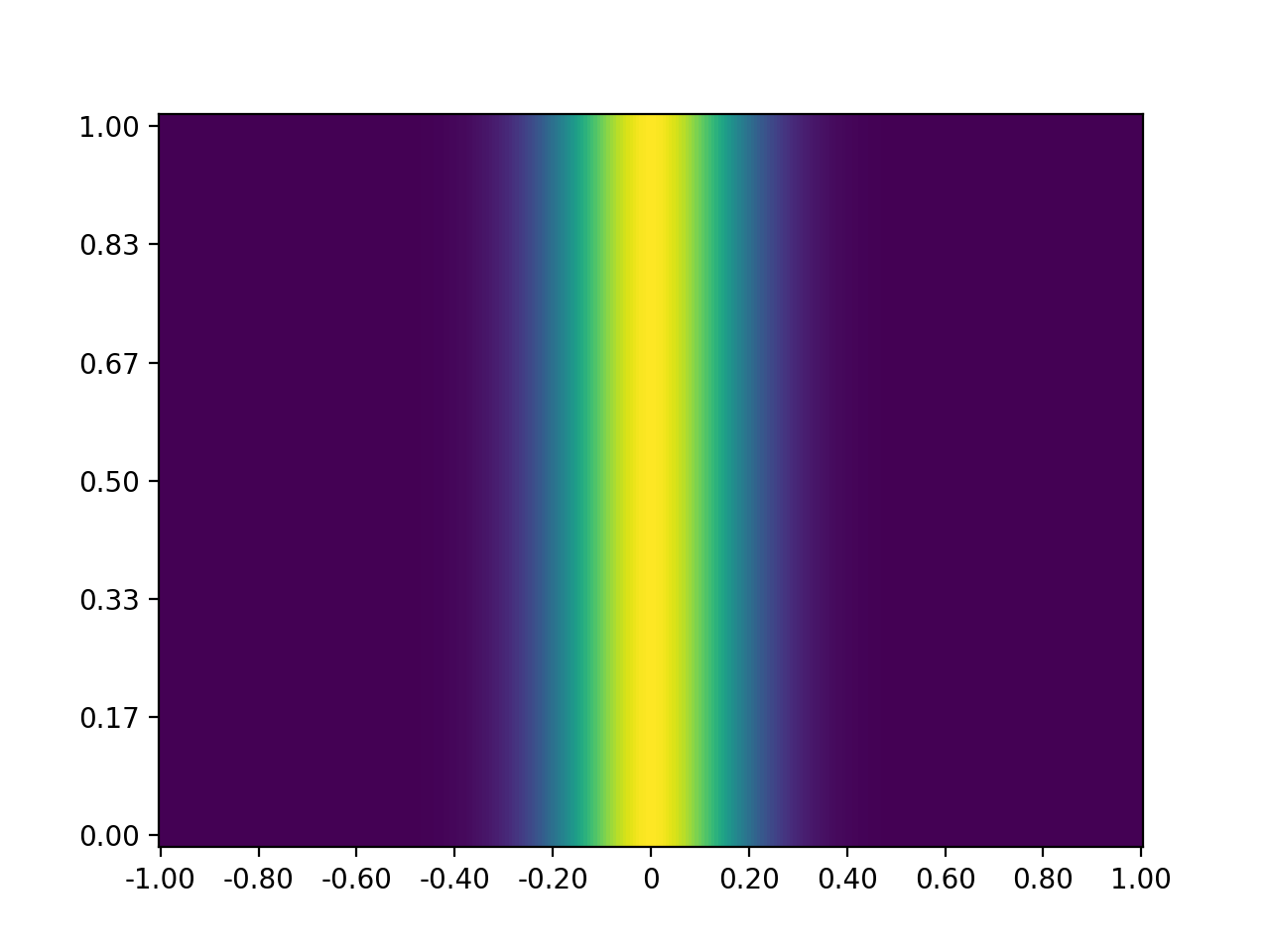}}& 
\subfloat[${\bf br}(\ov{\mathsf{M}}^1)$]{\includegraphics[width=.3\textwidth]{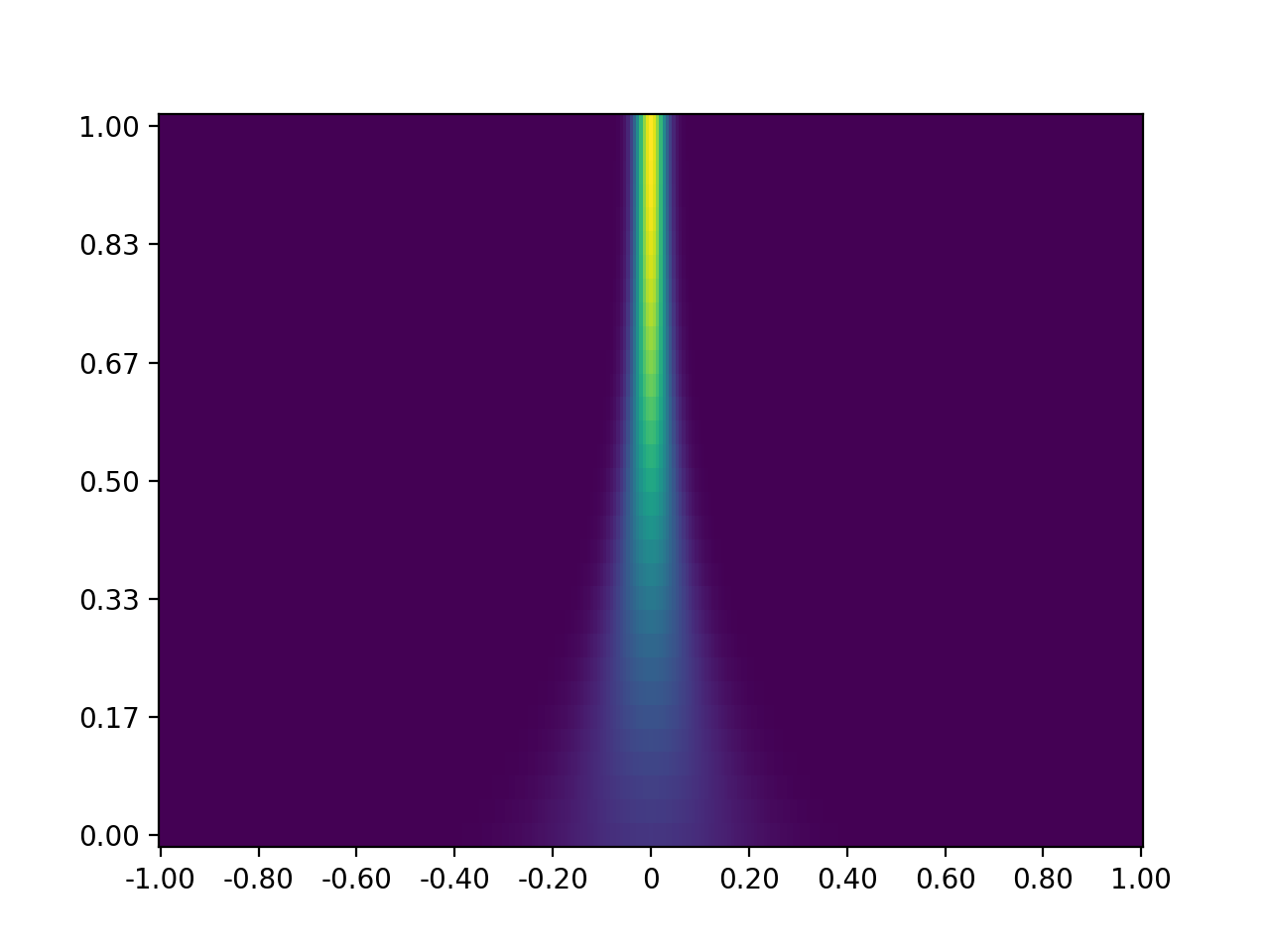}} \\
\subfloat[$\ov{\mathsf{M}}^2$]{\includegraphics[width=.3\textwidth]{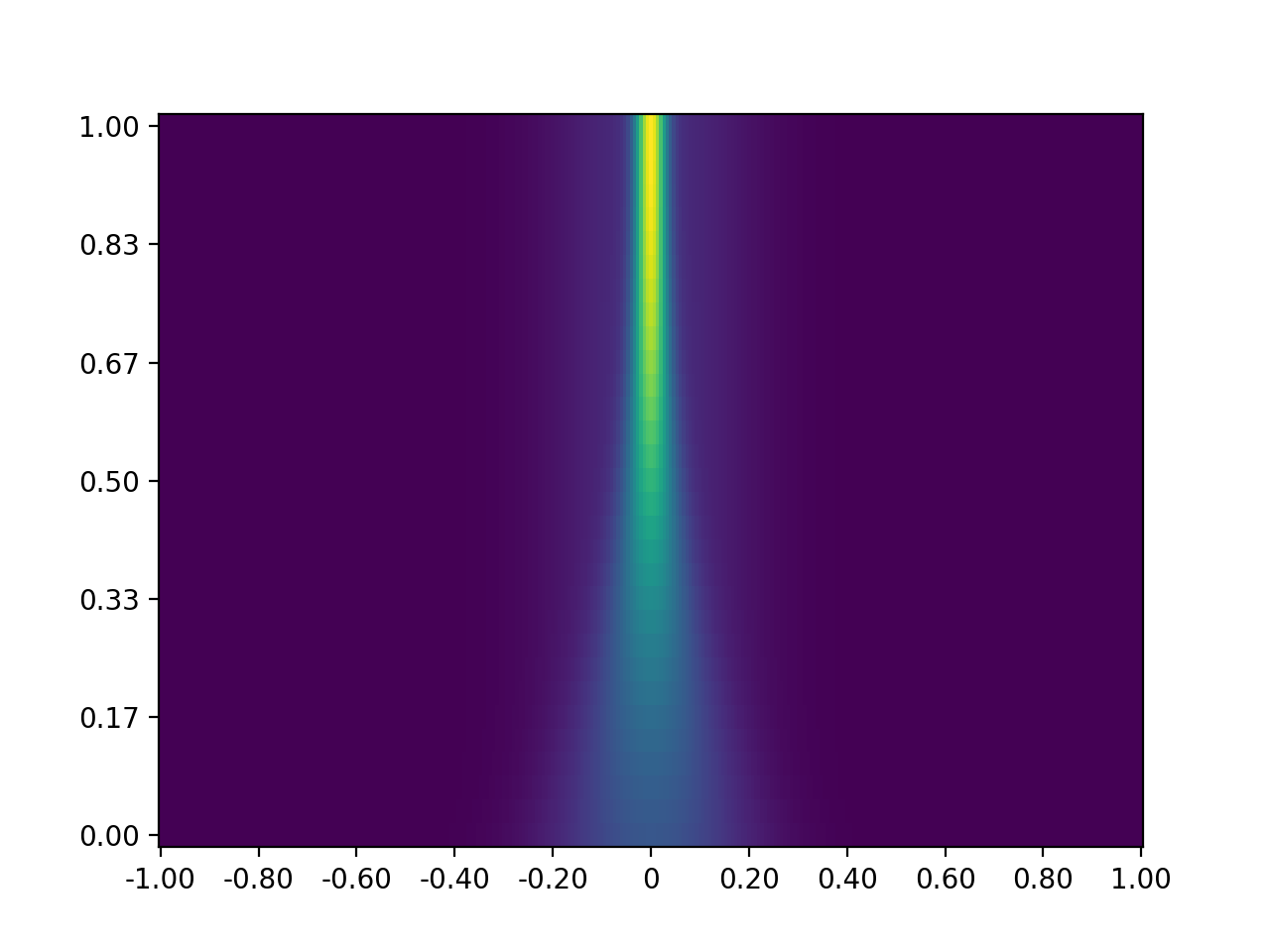}}& 
\subfloat[${\bf br}(\ov{\mathsf{M}}^2)$]{\includegraphics[width=.3\textwidth]{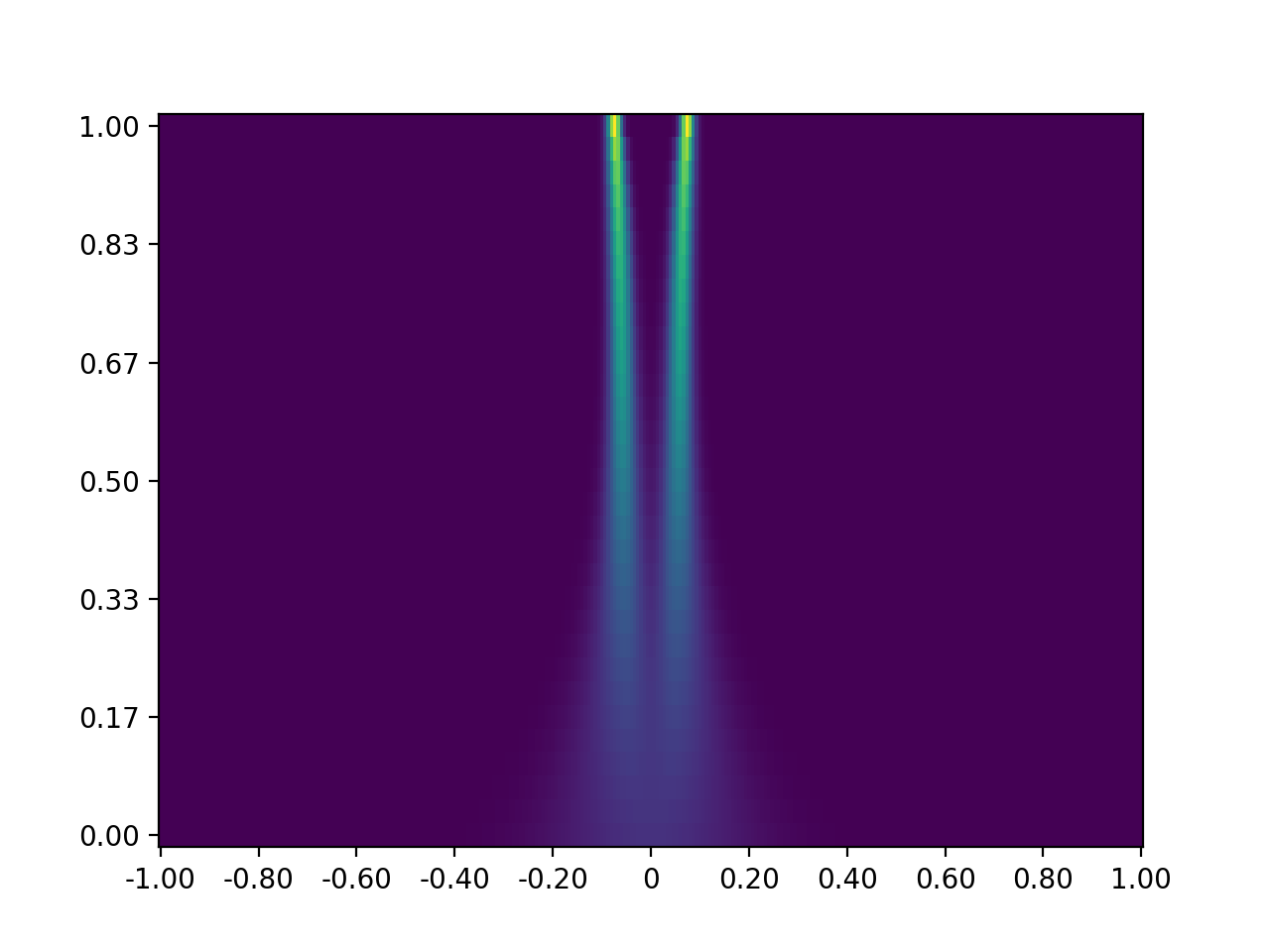}} \\
\subfloat[$\ov{\mathsf{M}}^3$]{\includegraphics[width=.3\textwidth]{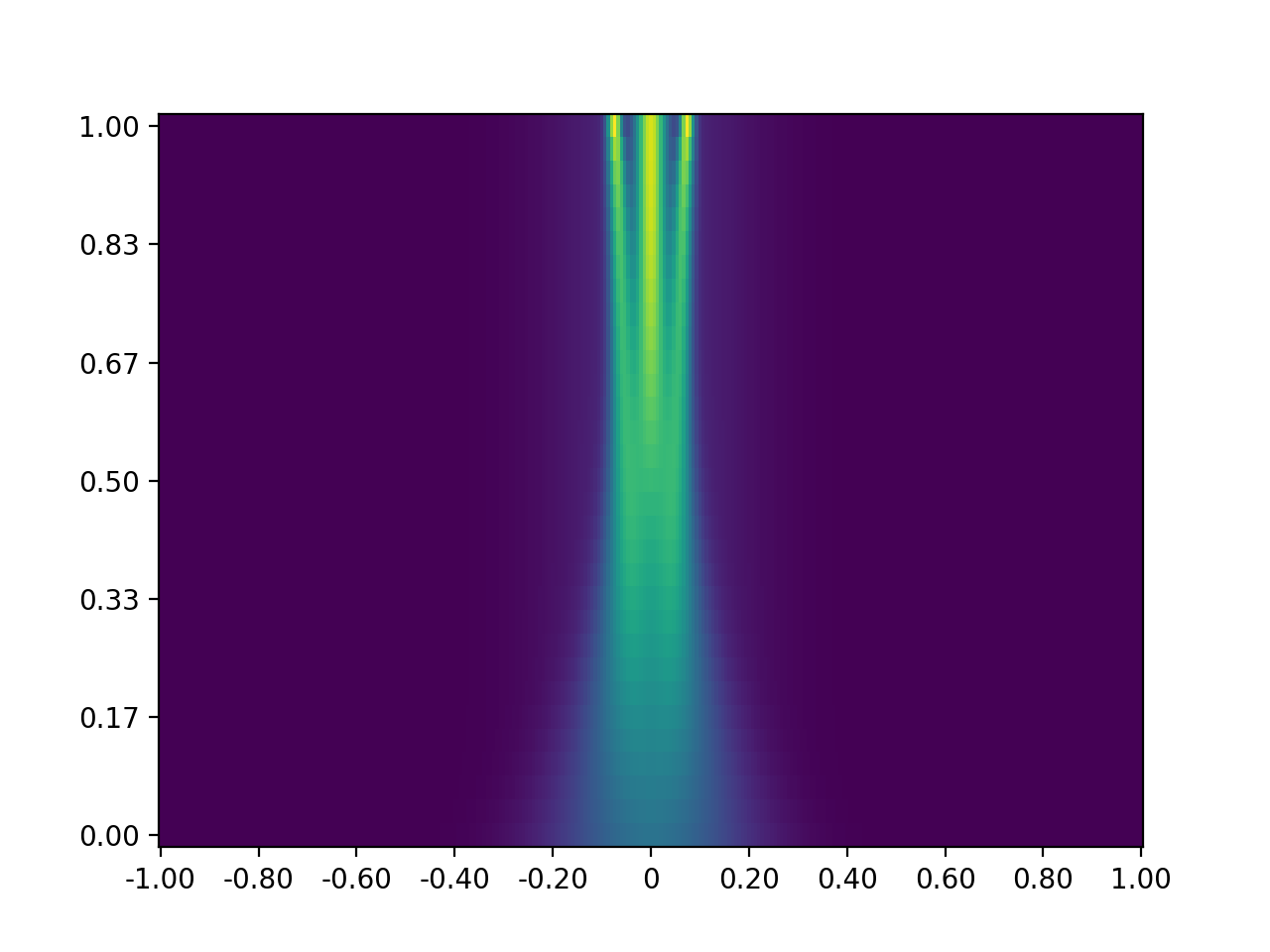}}& 
\subfloat[${\bf br}(\ov{\mathsf{M}}^3)$]{\includegraphics[width=.3\textwidth]{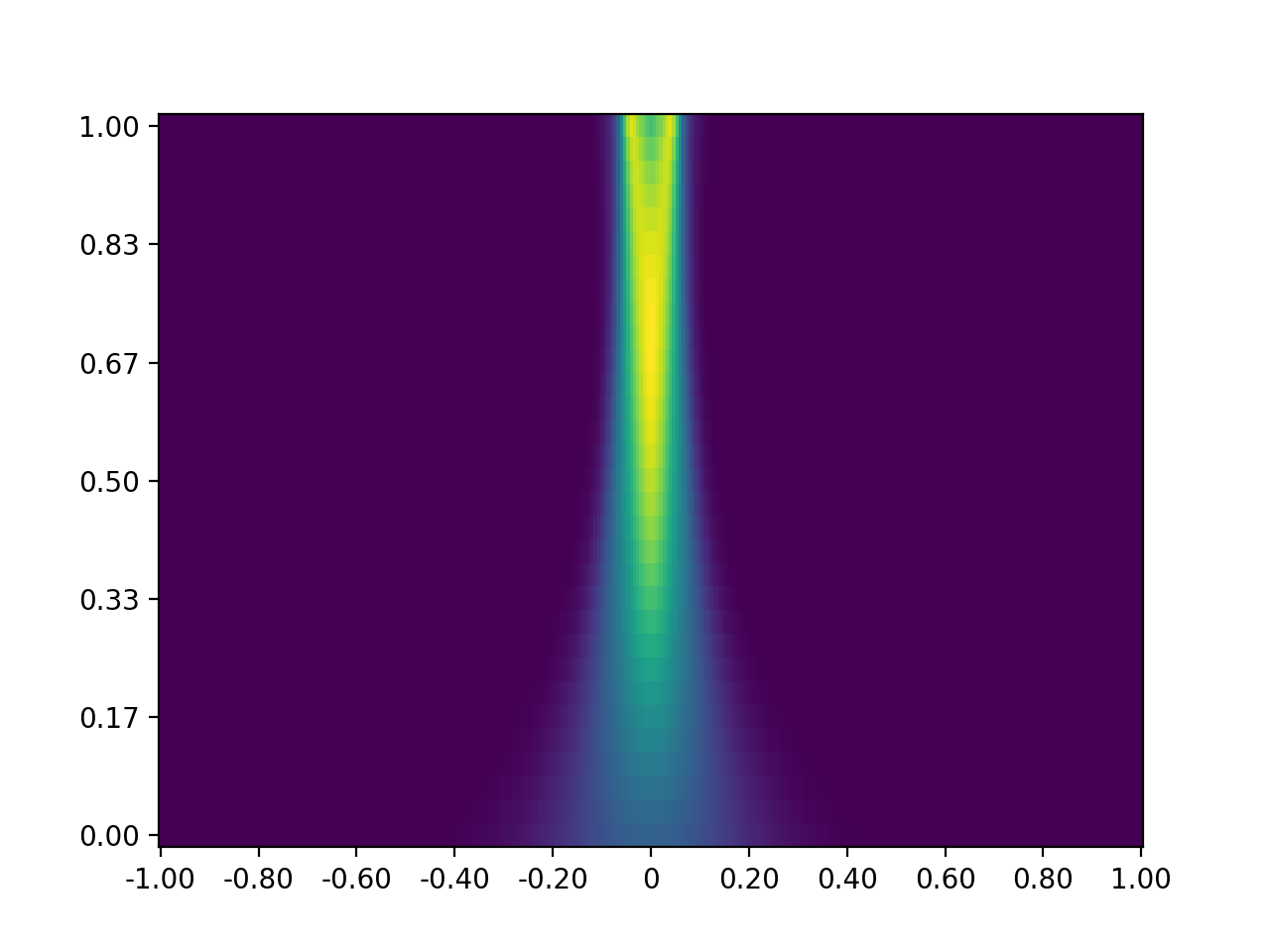}} \\
\subfloat[$\ov{\mathsf{M}}^4$]{\includegraphics[width=.3\textwidth]{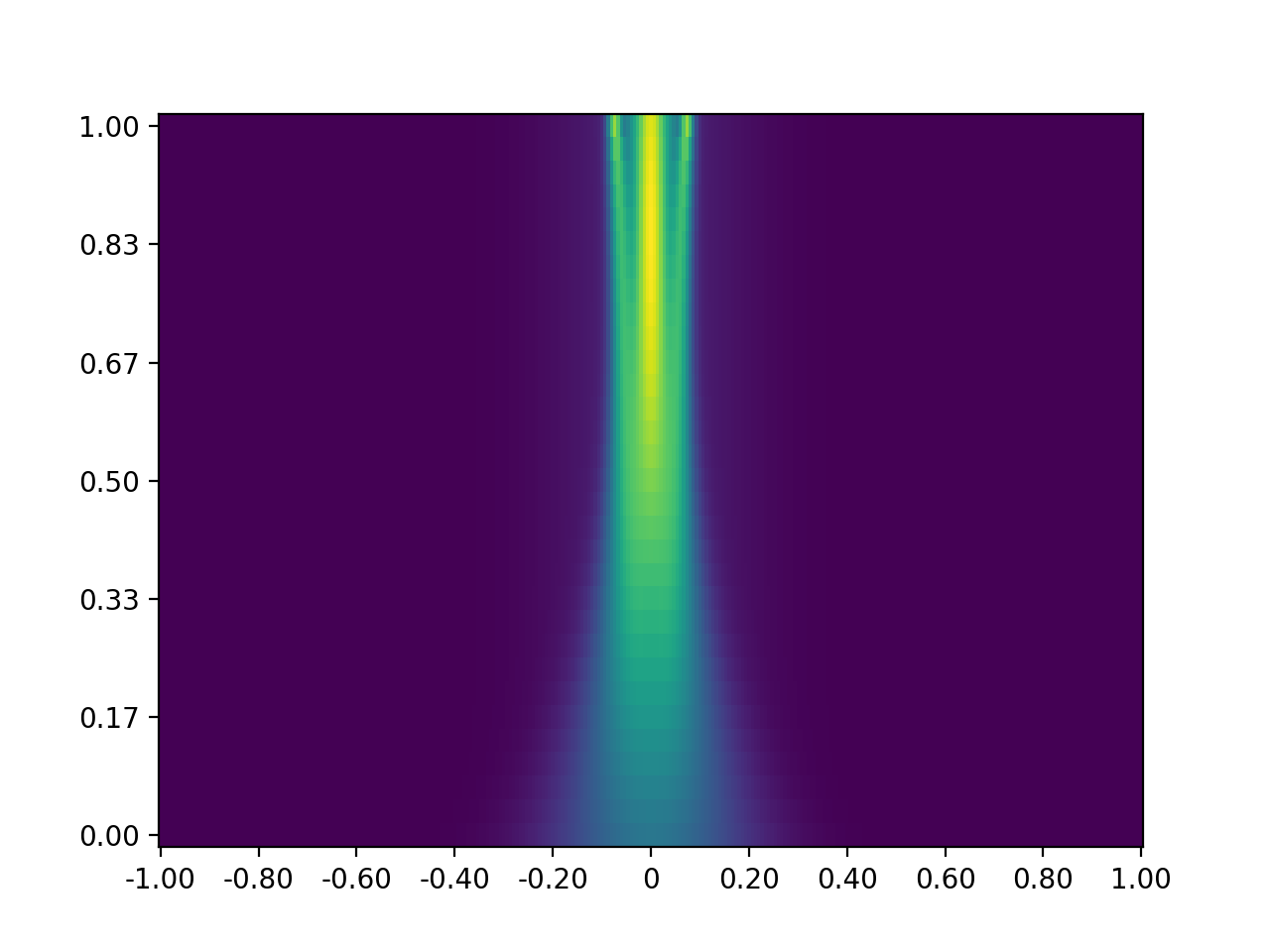}}& 
\subfloat[${\bf br}(\ov{\mathsf{M}}^4)$]{\includegraphics[width=.3\textwidth]{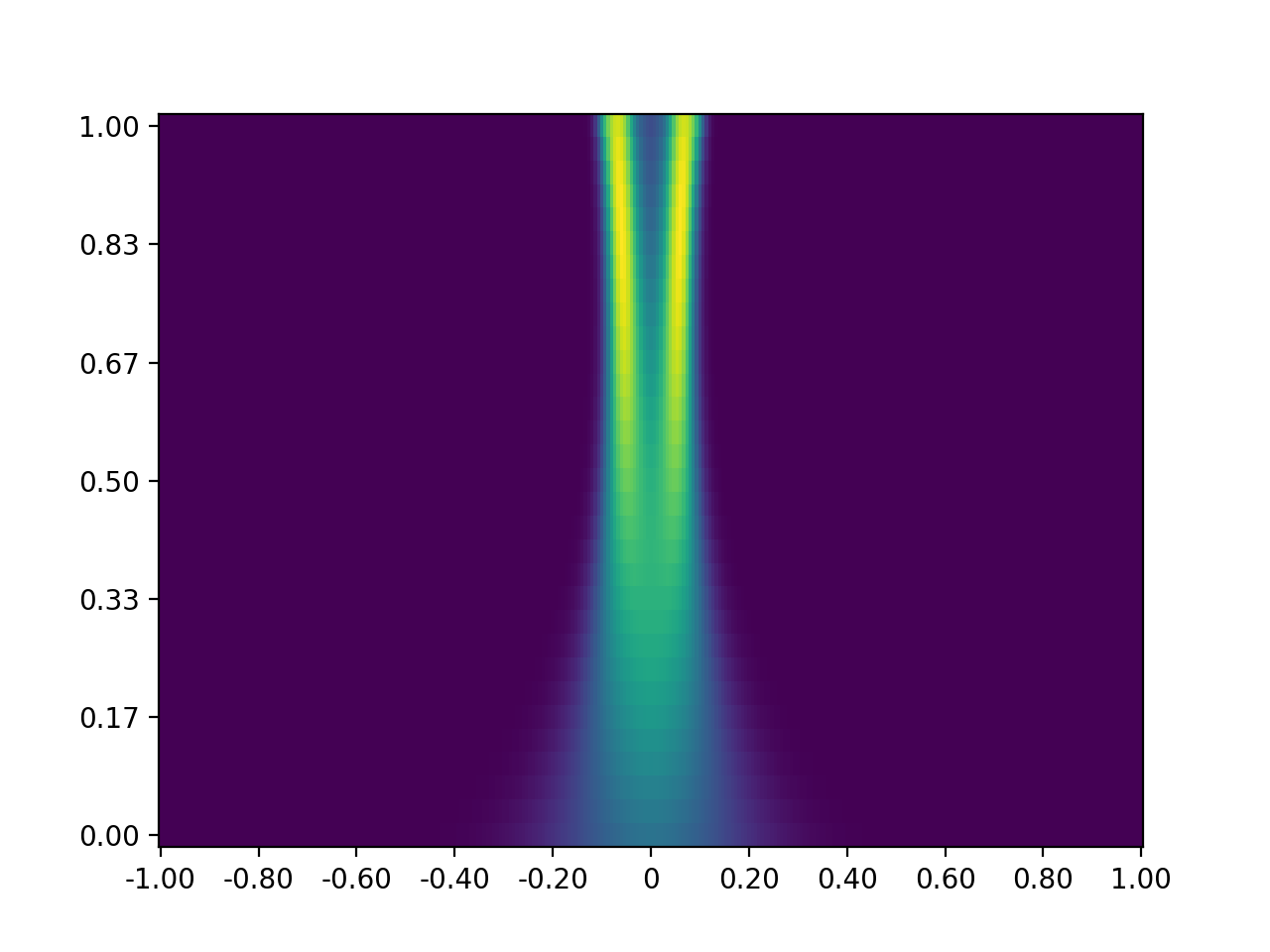}} \\
\end{tabular}
\caption{The first four iterations of Algorithm~\ref{alg:FP}, with $(\theta_1,\theta_2)=(1,1)$ and $\delta =0.1$. The figures on the left show the distributions obtained by averaging the previous guess and its best response, while, on the right, we display the best response to the current guess. The abscissa and ordinate axes represent position and time, respectively.}
\label{fictitious_play_working}
\end{figure}

\subsection{Example 2} 
\label{subsection_example_2}
We consider here a two-dimensional deterministic MFG, where a typical agent controls its acceleration (see~\cite{MR4102464, MR4132067}). We take $T=1$ as time horizon, 
\be 
\label{A_and_B_example_2}
A(t,x)=\left(\ba{cc}
0&0\\
1&0\ea\right),\quad \text{and}\quad B(t,x)=\left(\ba{c} 1\\ 0\ea\right) \quad \text{for all } (t,x)\in [0,T]\times \RR^2,
\ee
in~\eqref{ecuacion_controlada}, which, for $\alpha\in L^{2}([0,T];\RR)$, yields the controlled dynamics
$$
\left\{\ba{rcl} 
\dot{\gamma}_1(t)&=&\alpha(t),\\[6pt]
\dot{\gamma}_2(t)&=&\gamma_1(t),
\ea \right.\quad \text{for a.e. } t\in (0,1).
$$
In the above dynamics, $\gamma_1(t)$ and $\gamma_2(t)$ represent, respectively, the velocity and the position of an agent at time $t\in [0,T]$. As in the previous example, the cost function is decomposed as in {\bf(H5)}{\rm(i)}. More precisely, let $\sigma>0$, define $\rho_\sigma$ by~\eqref{rho_sigma}, and, given $\theta_1$, $\theta_2\geq 0$, let us take
\be 
\label{cost_functionals_second_example}
\ba{rcl}
 \ds f(t,x,\mu)&=&\theta_1 (\rho_\sigma\star \mu_2)(\mathrm{x}_2) \quad \text{for all } (t,x,\mu)\in [0,T]\times\RR^2\times \P_1(\RR^2), \\[6pt]
\ds \ell(t,a,x,\mu) &=& \ds \frac{|a|^2}{2}+(\mathrm{x}_2-0.3)^2+ f(t,x,\mu) \quad \text{for all } (t,a,x,\mu)\in [0,T]\times\RR\times\RR^2\times \P_1(\RR^2), \\[8pt] 
\ds g(x,\mu)&=&\theta_2 (\rho_\sigma\star \mu_2)(\mathrm{x}_2) \quad \text{for all } (x,\mu)\in\RR^2\times \P_1(\RR^2), 
\ea
\ee
where $x=(\mathrm{x}_1,\mathrm{x}_2)$ and  $\mu_2:=\pi_2\sharp \mu \in \P_1(\RR)$, with  $\pi_2:\RR^2\to \RR$ being defined as $\pi_2(x)=\mathrm{x}_2$ for all $x\in\RR^2$.  We consider  an absolutely continuous initial distribution $m_0\in\P_1(\RR^2)$ in product form,  given by
$$
\dd m_0(x)=\left(\frac{1}{\text{\footnotesize $0.04$}} \mathbb{I}_{[-0.02,0.02]}(\mathrm{x_1})\dd \mathrm{x}_1 \right)\otimes \left(\mathbb{I}_{[-1,1]}(\mathrm{x}_2)\frac{e^{-\mathrm{x}_2^2/0.001}}{\int_{-1}^1 e^{-y^2/0.001}\dd y}\dd \mathrm{x_2}\right).
$$
In this framework, $f$ and $g$ model some costs associated to the aversion of a typical player to crowded areas, while the additional terms in $\ell$ penalize high acceleration and deviation from the target position $\ov{\mathrm{x}}_2=0.3$. 
\begin{remark}
\label{quadratic_term_example_2} {\rm(i)} The presence of the quadratic term $(\mathrm{x}_2-0.3)^2$ in the definition of $\ell$ implies that {\bf(H1)} does not hold. Moreover, in this case, the corresponding value function~\eqref{semidiscrete-scheme} of the discrete time optimal control problem is merely locally Lipschitz. However, using that $B$ in~\eqref{A_and_B_example_2} does not depend on $x$, one can still show that the conclusions of Theorem~\ref{main_result} remain true. Extensions of this type will be the subject of a forthcoming note.
\smallskip\\
{\rm(ii)} As an alternative to working with this quadratic cost term,  one could replace it by a Lipschitz and bounded function of the state variable, which coincides with the quadratic function on an arbitrary interval of $\RR$ centered at $0.3$. Under this modification, our dataset satisfies {\bf(H1)}-{\bf(H4)}, and {\bf(H5)}{\rm(i)}. However, for the sake of simplicity and in view of {\rm(i)}, we have chosen to maintain the original quadratic term in our numerical simulations. 
\end{remark}

The purpose of the following result is to show that the coupling terms $f$ and $g$ are monotone. Thus, by~\cite[Theorem~3.2]{MR4030259}, we obtain the convergence of the fictitious play method when applied to Problem~\ref{MFG-SL_discrete_scheme_n-degenerate}.

\begin{lemma}
\label{monotonia_ej_2}
The functions $f$ and $g$ defined in~\eqref{cost_functionals_second_example} satisfy {\bf(H5)}{\rm(ii)}. 
\end{lemma}
\begin{proof}
Let $\mu, \bar\mu\in \P_1(\cR^2)$ and set $\mu_2=\pi_2\sharp\mu$ and $\bar\mu_2=\pi_2\sharp\bar\mu$. Since $\rho_\sigma= \rho_{\bar{\sigma}} \star \rho_{\bar{\sigma}}$, with $\bar{\sigma}=\sqrt{\sigma^2/2}$, Fubini's theorem implies that
$$
\ba{l}
\ds \int_{\cR^2} \left[ \left(\rho_\sigma \star \mu_2\right) (\mathrm{x}_2) - \left(\rho_\sigma \star \bar\mu_2\right) (\mathrm{x}_2)  \right] \dd (\mu-\bar\mu)(x)\\[6pt]
\hspace{2cm} = \ds \int_{\cR^2} \left[ (\rho_{\bar\sigma}\star \rho_{\bar\sigma} \star \mu_2) (\pi_2(x)) - (\rho_{\bar\sigma}\star \rho_{\bar\sigma} \star \bar\mu_2) (\pi_2(x))   \right] \dd (\mu-\bar\mu)(x)\\[12pt]
\hspace{2cm} = \ds  \int_{\cR} \left[ (\rho_{\bar\sigma}\star \rho_{\bar\sigma} \star \mu_2) (\mathrm{x}_2) - (\rho_{\bar\sigma}\star \rho_{\bar\sigma} \star \bar\mu_2) (\mathrm{x}_2)   \right]\dd (\mu_2-\bar\mu_2)(\mathrm{x}_2)\\[12pt]
\hspace{2cm} =  \ds \int_{\cR} \int_{\cR} \rho_{\bar\sigma}(\mathrm{x}_2-y) (\rho_{\bar\sigma} \star (\mu_2-\bar\mu_2))(y) \, \dd y \,  \dd (\mu_2-\bar\mu_2)(\mathrm{x}_2)\\[12pt]
\hspace{2cm} =\ds  \int_{\cR} \left[\left(\rho_{\bar\sigma} \star (\mu_2-\bar\mu_2)\right) (y) \right]^2 \dd y,
\ea
$$
which is nonnegative. 
\end{proof}

In the numerical simulations,  we consider the same values for $\sigma$, $\Delta t$, $\Delta x$, $\eps$, and $(\theta_1,\theta_2)$ as those in Example~\ref{subsection_example_1}.  We also keep the same tolerance parameters and fix the discretization of the initial distribution $m_0$, described at the beginning of  Section~\ref{sec:main_result}, as initial guess in Algorithm~\ref{alg:FP} for all the time marginals. Table~\ref{tabla2} provides the number of iterations needed for attaining the desired accuracies and Figures~\ref{first_figure_example_2} and~\ref{second_figure_example_2} display the evolution of the distribution of the velocities and positions of the agents at equilibrium for the two values of $(\theta_1,\theta_2)$ and tolerance parameter $\delta = 0.001$. Observe that, as expected, the presence in the cost functional of a term modeling aversion of a typical agent to crowded regions at the final time $T$ , i.e. $\theta_2\neq 0$, has an important impact on the final distribution.
\begin{table}[h]
\begin{center}
\vspace{0.1cm}
\begin{tabular}{|c|c|c|}
	\hline
\; &\quad {\small Test 1: $\theta_1= 1$, $\theta_2 = 0$} \quad &\quad {\small Test 2: $\theta_1= 1$, $\theta_2 = 1$ }\\[2pt]
	\hline
Tolerance & No. of iterations & No. of iterations \\[2pt]
\hline
$\delta =0.1$ &$n= 25$ & $n=30$\\
$\delta =0.01$ & $n=11$& $n=10$\\
$\delta = 0.001$ & $n=10$ & $n=10$\\
\hline
\end{tabular}
\end{center}
\caption{Number of iterations to obtain the desired accuracies.}
\label{tabla2}
\end{table}
\begin{figure}[h]
\begin{center}
\includegraphics[width=.5\textwidth]{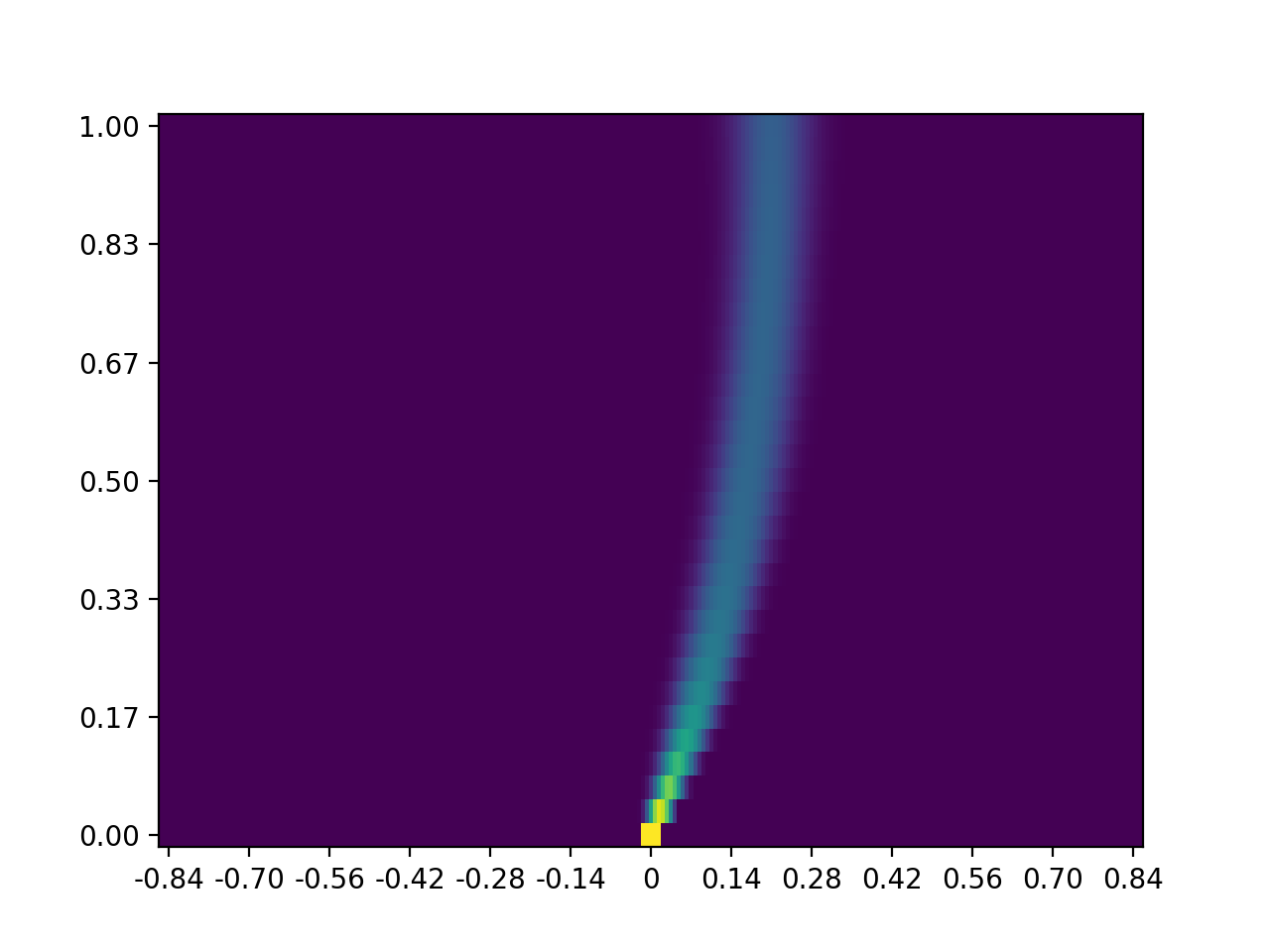}\includegraphics[width=.5\textwidth]{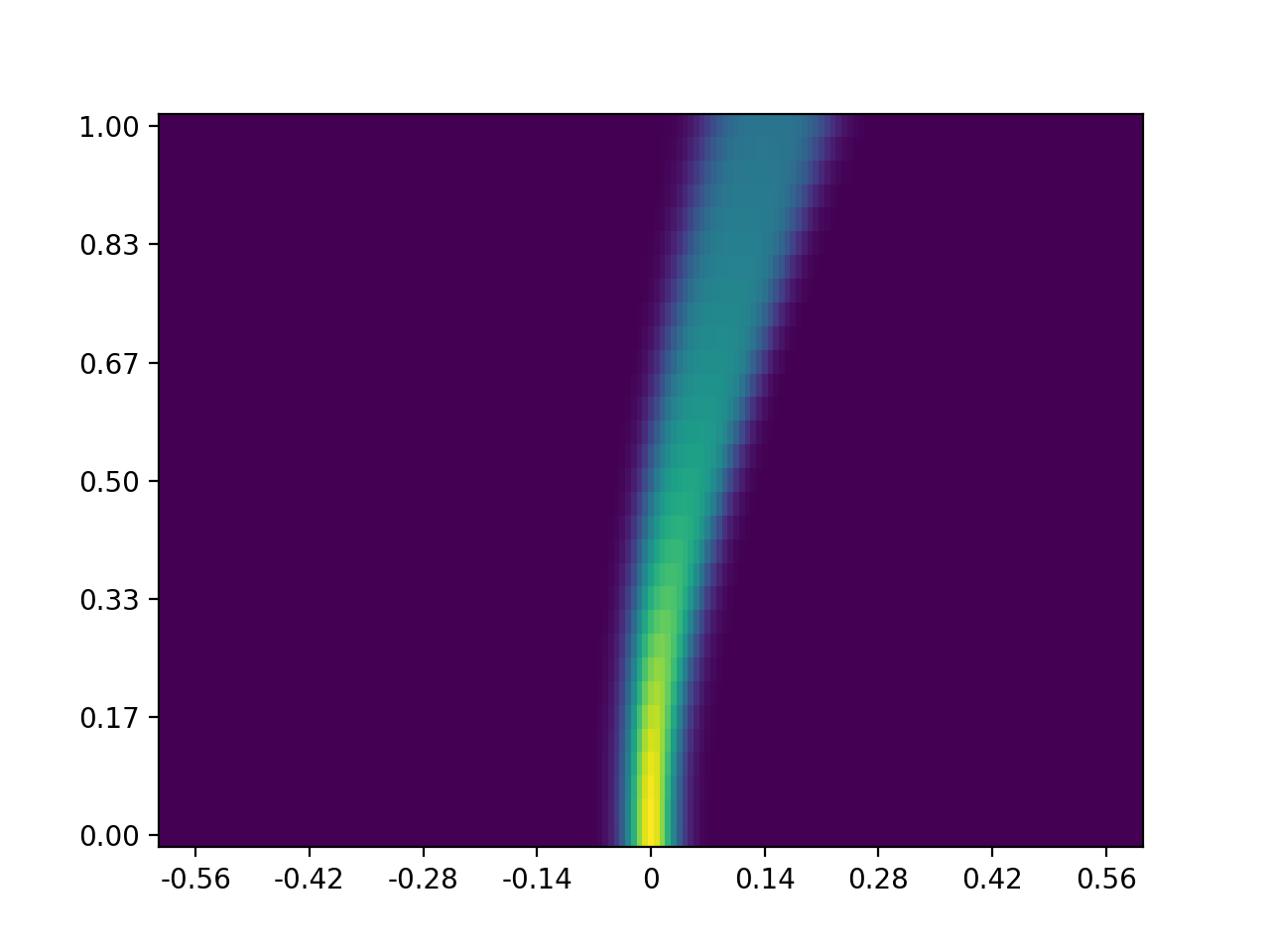}
\end{center}
\caption{Test 1:  Distribution of the velocities (left) and positions (right) for $\theta_1= 1$, $\theta_2=0$ and tolerance parameter $\delta=0.001$. In both figures  the ordinate axes represent time, while the abscissa axis represents velocity on the left figure and position on the right one.}
\label{first_figure_example_2}
\end{figure}
\begin{figure}[h]
\begin{center}
\includegraphics[width=.5\textwidth]{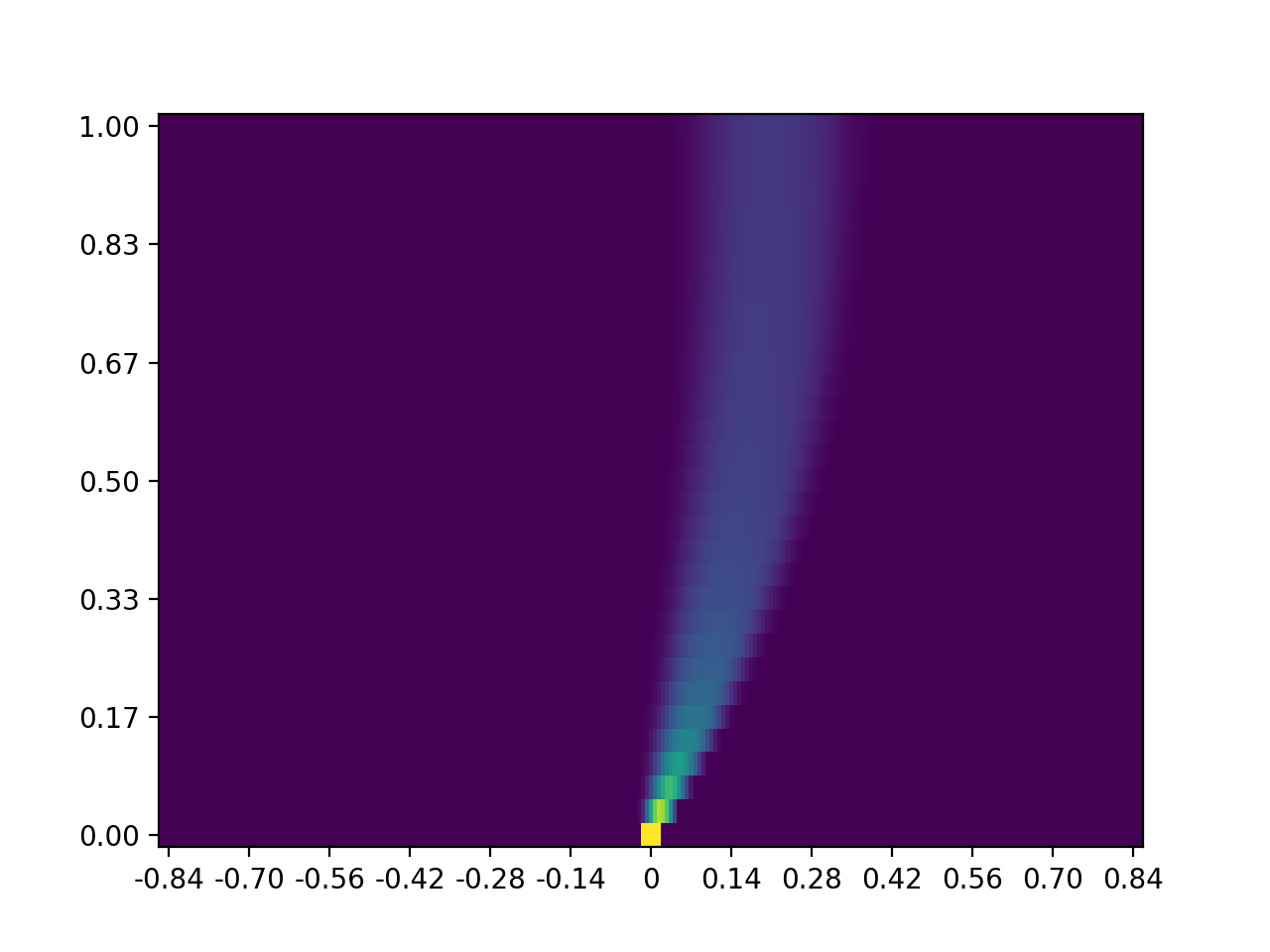}\includegraphics[width=.5\textwidth]{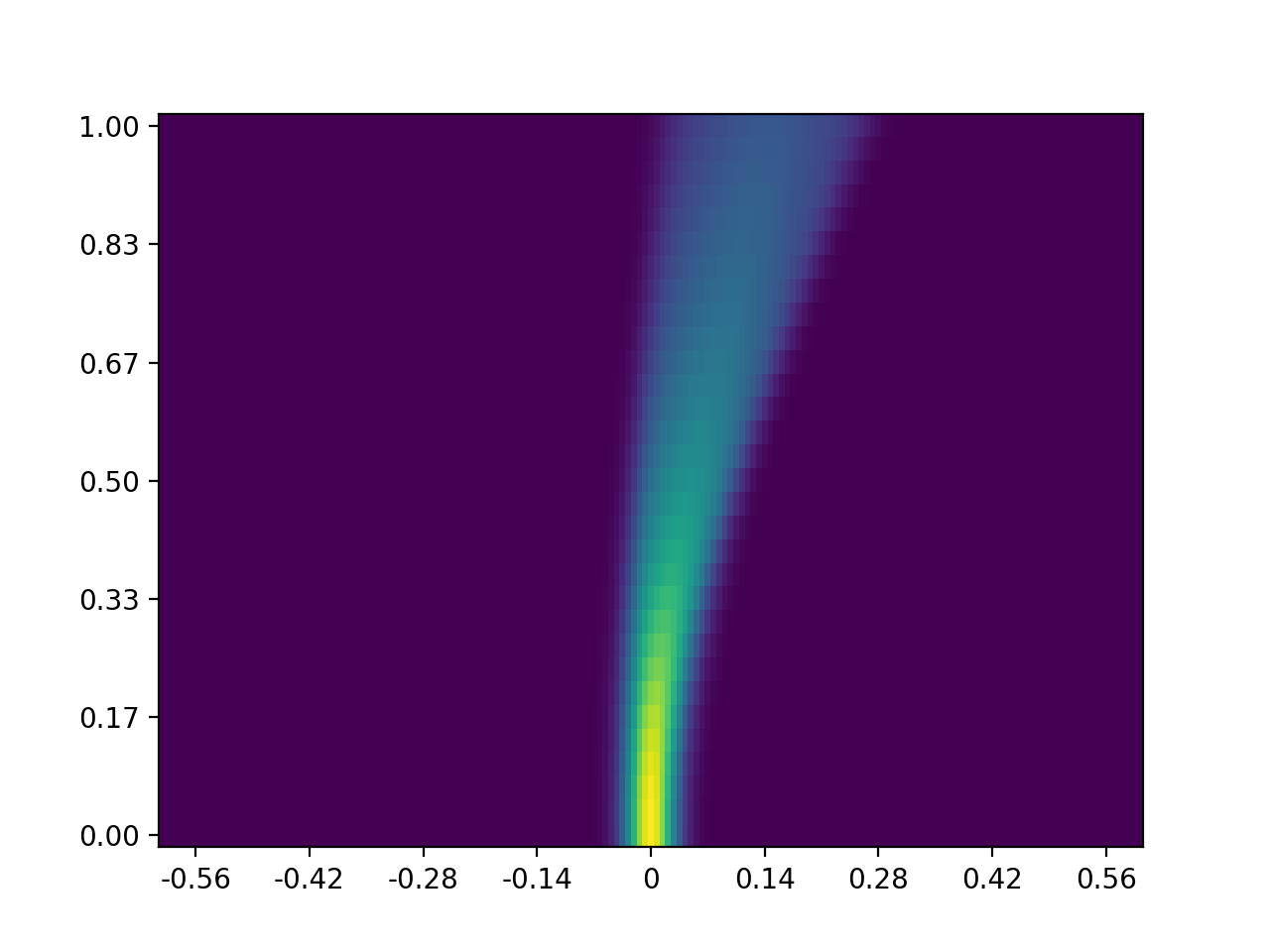}
\end{center}
\caption{Test 2:  Distribution of the velocities (left) and positions (right) for $\theta_1= 1$, $\theta_2=1$ and tolerance parameter $\delta=0.001$. In both figures  the ordinate axes represent time, while the abscissa axis represents velocity on the left figure and position on the right one.}
\label{second_figure_example_2}
\end{figure}

\section*{Appendix I}

In this section we prove some technical properties of the semi-discrete value function $v_{k}\colon \RR^d\to \RR$ ($k\in \I$) introduced in Section \ref{semi_discrete_section}.  In what follows we assume that {\bf(H1)}, {\bf(H2)} and {\bf(H3)} are in force.  

We start showing the existence of optimal controls for the problem defined by $v_k$ in \eqref{semidiscrete-scheme}. 
\begin{lemma}   
\label{lem:existencia-control-dis} 
Given  $k \in \I^{*}$ and $x\in\cR^d$, there exists  $\bar\alpha\in \A_{k}$ such that $v_k( x)= J_{k, x}(\bar\alpha)$. In addition, there exists $\tilde{C}>0$, independent of $\Delta t$, $m$, $k$, and $x$, such that any optimal solution $\tilde\alpha\in\A_k$ satisfies
\be  
\label{eq:cota-alpha-dis}
\Delta t \sum_{j=k}^{N_t-1}|\tilde\alpha_j|^p\leq \tilde{C}. 
\ee
\end{lemma}
\begin{proof}
Given   $k \in \I^{*}$ and $x\in\cR^d$, let $(\alpha^n)_{n\in\NN}\subset \A_k$ be a minimizing sequence for $J_{k,x}$. Let $(\gamma^n)_{n\in\NN}\subset \Gamma_{k,x}$ be the  sequence of states associated with $(\alpha^n)_{n\in\NN}$ via \eqref{def:dis-state} and let $\bar\gamma^0\in\Gamma_{k,x}$ be the state associated with the null control. By definition of minimizing sequence, {\bf(H1)}{\rm(i)}, and {\bf(H2)}, for any $\delta>0$, there exists $n_0\in \NN$ such that   
\be  \label{eq:est-control-i}
\Delta t \underline{\ell}\sum_{j=k}^{N_t-1}\left|\alpha^n_j\right|^p-(T-t_k)C_\ell\leq J_{k,x}(0)+\delta -g(\gamma^n_{N_t})\leq (T-t_k)C_\ell+\delta+L_g\left|\bar\gamma^0_{N_t} - \gamma^n_{N_t}\right|,
\ee
for all $n\geq n_0$.  Let us estimate $\left|\bar\gamma^0_{N_t} - \gamma^n_{N_t}\right|$. For every $j=k, \dots, N_t-1$ we have 
$$
\ba{lll}
\left|\bar\gamma^0_{j+1} - \gamma^n_{j+1}\right|&\leq& \left|\bar\gamma^0_{j} - \gamma^n_{j}\right|+\Delta t \left|A(t_j,\bar\gamma^0_{j}) -A(t_j, \gamma^n_{j})\right|+ \Delta t\left|B(t_j, \gamma^n_{j})\alpha^n_j\right|\\[10pt]
\;&\leq&\left( 1+\Delta t L_A\right)\left|\bar\gamma^0_{j} - \gamma^n_{j}\right|+\Delta t C_B\left|\alpha^n_j\right|.
\ea
$$
Since $\bar\gamma^0_{k} =\gamma^n_{k}$, by the discrete Gr\"onwall's lemma, there exists $C>0$ such that 
\be
\label{conseq_gronwall_discreto}
\max_{j=k,\dots, N_t}\left|\bar\gamma^0_{j} - \gamma^n_{j}\right| \leq C \Delta t\sum_{j=k}^{N_t-1}\left|\alpha^n_j\right|.
\ee
Thus, by Young's inequality, for every $\eta>0$ there exists $C_\eta>0$ such that 
$$
 L_gC \Delta t\sum_{j=k}^{N_t-1}\left|\alpha^n_j\right|\leq C_\eta+\eta\Delta t\sum_{j=k}^{N_t-1}\left|\alpha^n_j\right|^p. 
$$
Taking $\eta<\underline{\ell}$ and combining the above equation with \eqref{eq:est-control-i} and \eqref{conseq_gronwall_discreto},  we deduce the existence of $\tilde{C}>0$, independent of $\Delta t$, $m$, $k$, and $x$, such that 
$$
\Delta t \sum_{j=k}^{N_t-1}|\alpha_j^n|^p\leq \tilde{C}. 
$$
Therefore,  there exists at least one accumulation point $\bar\alpha\in\A_k$ of $(\alpha^n)_{n\in\NN}$ and, by the continuity assumptions in {\bf(H1)-(H3)}, we conclude that $v_k(x)= J_{k, x}(\bar\alpha)$. Finally, if $\tilde{\alpha}$ is any other optimal control for the problem defined by $v_k(x)$, the previous argument shows that \eqref{eq:cota-alpha-dis} holds.
\end{proof}
\begin{proof}[Proof of the Lipschitz property \eqref{lem:Lips-v-semidis}.]
Given $k \in \I^{*}$ and $x\in\cR^d$, let $\bar\alpha\in \A_k$ be an optimal control for $v_k(x)$ and let $\bar\gamma\in\Gamma_{k,x}$ be the associated state given by \eqref{def:dis-state}. Let $y\in\cR^d$ and let $\zeta\in \Gamma_{k,y}$ be the state associated with $\bar\alpha$ via  \eqref{def:dis-state}, 
then {\bf(H1)}{\rm(ii)} implies 
\be
\label{lipschitz_desigualdad_1}
\ba{lll}
v_k(y)-v_k(x)&\leq&\ds\Delta t\sum_{j=k}^{N_t-1}[\ell(t_j, \bar\alpha_j, \zeta_j, m(t_j))-\ell(t_j, \bar\alpha_j,  \bar\gamma_j, m(t_j))]\\[16pt]
\;&\;&+g(\zeta_{N_t},m(T))-g(\bar\gamma_{N_t},m(T))\\[6pt]
\;&\leq&\ds\left(L_\ell\left(T+\Delta t\sum_{j=k}^{N_t}|\bar\alpha_j|^p\right)+L_g\right)\max\limits_{j=k,\dots, N_t}|\zeta_j-\bar\gamma_j|.
\ea
\ee
On the other hand, for every $j=k,\dots, N_t-1$, 
$$
|\zeta_{j+1}-\bar\gamma_{j+1}|\leq \left(1+\Delta t \left[L_{A}  + L_B |\ov{\alpha}_j|\right]\right)|\zeta_{j}-\bar\gamma_{j}|.
$$
By the discrete Gr\"onwall's lemma, we have
$$
\max_{j=k,\dots, N_t}|\zeta_j-\bar\gamma_j|\leq \ds e^{\left(L_{A}T+ L_{B}\Delta t \sum\limits_{j=k}^{N_{t}-1}|\bar{\alpha}_j|\right)}|x-y|.
$$
By \eqref{eq:cota-alpha-dis}, H\"older's inequality, and \eqref{lipschitz_desigualdad_1}, there exists $L_{v}>0$, independent of $x$, $y$,  such that $v_k(x)-v_k(y)\leq L_v |x-y|$ for all $x$, $y\in \RR^d$,  which implies \eqref{lem:Lips-v-semidis}. 
\end{proof}

\begin{lemma} 
\label{lem:controles-dis-acotados}
Given  $k \in \I^*$ and $x\in\cR^d$, there exists $\alpha_{k,x}\in\cR^r$  such that
\be
\label{optimo_en_dpp}
v_k(x)=  \Delta t \ell(t_k, \alpha_{k,x}, x,m(t_k))+v_{k+1}\left(x+\Delta t[A(t_k, x)+ B(t_k, x)\alpha_{k,x}]\right).
\ee
Moreover, there exists $\widehat{C}>0$, independent of $\Delta t$, $m$, $k$, and $x$, such that
\be
\label{widehatC}
|\alpha_{k,x}|\leq \widehat{C}.
\ee
\end{lemma}
\begin{proof} Let $\bar{\alpha}$ be as in Lemma~\ref{lem:existencia-control-dis}. Then, by the dynamic programming principle, $\alpha_{k,x}=\bar{\alpha}_k$ satisfies \eqref{optimo_en_dpp} and hence
$$
\ba{l}
\Delta t\ell(t_k, \alpha_{k,x}, x, m(t_k))+v_{k+1}\left(x+\Delta t\left[A(t_k, x)+B(t_k, x)\alpha_{k,x} \right]\right)\\[6pt]
\hspace{5cm}\leq \Delta t\ell(t_k, 0, x, m(t_k))+v_{k+1}\left( x+\Delta tA(t_k, x) \right).
\ea
$$
Thus, by \eqref{lem:Lips-v-semidis},  {\bf (H1)}{\rm (i)}, and {\bf(H3)}{\rm (ii)}, we have
$$
 \underline{\ell}|\alpha_{k,x}|^p  \leq   2C_\ell+ L_v C_ B|\alpha_{k,x}|
$$
and the existence of $\widehat{C}>0$ such that \eqref{widehatC} holds follows from Young's inequality. 
\end{proof}

\begin{proof}[Proof of the Proposition \ref{prop:conv-value-function}.]   Notice that  {\bf(H1)}{\rm(i)}, with $a=0$, and {\bf(H2)} imply that  
\be
\label{cota_v_n_below_above}
-C_\ell T+c_{g}\leq v^n_{k}(x)\leq C_{\ell}T+g(\bar\gamma^{0,n}_{N_t^n},m^n(T))\quad \text{for all }n\in\NN,\,k\in\I^n,\, x\in\RR^d,
\ee
where $\bar\gamma^{0,n}$ is defined by \eqref{def:dis-state} with $\alpha_j=0$ for all $j=k, \hdots, N_t^n-1$   and $\bar{\gamma}^{0,n}_k=x$. By the definition of $\bar\gamma^{0,n}$, Remark \ref{despues_de_hipotesis}{\rm{(i)}}, and the discrete Gr\"onwall's lemma,  there exists $C>0$, independent of $n$, $k$, and $x$, such that  
\be
\label{cota_uniforme_discreta}
|\bar\gamma^{0,n}_j| \leq C(1+ |x|)\quad \text{for all }n\in\NN,\,k\in\I^n,\,j=k,\hdots,N^n_t,\;x\in \RR^d,
\ee
which, together with \eqref{cota_v_n_below_above}, implies that $v^n$ is  locally bounded, uniformly with respect to $n$.  Therefore, we can define $v^*:[0,T]\times \RR^d \to \RR$ and $v_{*}:[0,T]\times \RR^d \to \RR$ by 
$$
v^*(t,x)=\limsup_{{\tiny{\ba{c}n\to\infty\\t^n_{k(n)}\to t, \,  k(n)\in\I^n\\ x^n\in\cR^d\to x\ea}}}v^n_{k(n)}(x^n)\quad\mbox{and}\quad v_*(t,x)=\liminf_{{\tiny{\ba{c}n\to\infty\\t^n_{k(n)}\to t,  k(n)\in\I^n\\ x^n\in\cR^d\to x\ea}}}v^n_{k(n)}(x^n),
$$
for all $(t,x)\in [0,T]\times \RR^d$. By \cite[Chapter V, Lemma 1.5]{MR1484411} we have that $v^*$ and $v_*$ are upper and lower semicontinuous, respectively.  We claim that  \smallskip\\
{\bf (a)} $v^*(T,x)=v_*(T,x)=g(x,m(T))$ for all $x\in \RR^d$. \smallskip\\
{\bf (b)} $v^*$ satisfies, in the viscosity sense (see e.g. \cite[Chapter II]{MR1484411}), 
\be
\label{subsolution_hjb}
-\partial_t v^*(t,x)+ H(t,x,\nabla_xv^*(t,x),m(t))\leq 0\quad \text{for all }(t,x)\in(0,T)\times\cR^d.
\ee
{\bf (c)} $v_*$ satisfies, in the viscosity sense, 
$$
-\partial_t v_*(t,x)+ H(t,x,\nabla_{x} v_*(t,x),m(t))\geq 0\quad\text{for all }(t,x)\in (0,T)\times\cR^d.
$$
If {\bf (a)}, {\bf(b)}, and {\bf(c)} hold, then by the comparison principle \cite[Theorem 2.1]{Da-Lio:2011aa} we obtain that $v^*=v_*=v$ and  the result follows from   \cite[Chapter V, Lemma 1.9]{MR1484411}. It remains to prove the claim. The proofs of {\bf(b)} and {\bf(c)}  being analogous, we only show {\bf(a)} and {\bf(b)}. \medskip\\
{\it Proof of {\bf(a)}.} Let $x\in \RR^d$ and $(k(n), x^n)\in \I^n\times\cR^d$ be such that $(t^n_{k(n)}, x^n)\to (T,x)$ as $n\to \infty$.  Denote by  $\widehat{\A}_{k}^n$ the set defined in \eqref{def_gamma_A_k} for $k=0,\hdots,N^n_t-1$. By Lemma \ref{lem:existencia-control-dis} and Lemma \ref{lem:controles-dis-acotados}, there exists $\alpha^n\in \widehat{\A}_{k(n)}^n$ such that
$$
v^n_{k(n)}(x^n)=\Delta t_n \sum_{j=k(n)}^{N^n_t-1}\ell(t_j^n, \alpha^n_j, \gamma^n_j, m^n(t_j^n))+g\left(\gamma^n_{N^n_t}, m^n(T)\right),
$$
where $\gamma^n$ is the state associated with $\alpha^n$ via  \eqref{def:dis-state}   and $\gamma^n_{k(n)}=x^n$. By {\bf (H1){\rm(i)}} we obtain
\be    \label{eq:ineq-lim-v-n}
-C_{\ell} (T-t_{k(n)}^n)+g\left(\gamma^n_{N^n_t},m^n(T)\right)\leq v^n_{k(n)}(x^n)\leq \left(\overline{\ell}\widehat{C}^p+C_\ell\right) (T-t_{k(n)}^n)+g\left(\gamma^n_{N^n_t},m^n(T)\right).
\ee
By {\bf (H3)} and Lemma \ref{lem:controles-dis-acotados} we have
\be
\label{largo_trayectoria_completa}
\ba{lll}
\left|\gamma^n_{N^n_t}-x^n \right|&=&\ds\Delta t_n\sum_{j=k(n)}^{N^n_t-1}\left[A(t_j^n, \gamma^n_j)+B(t_j^n, \gamma^n_j)\alpha^n_j \right]\\[16pt]
\;&\leq&(T-t_{k(n)}^n)\left[C_A \left(1+ \max_{j}|\gamma^n_j|\right)+C_B\widehat{C} \right].
\ea
\ee
On the other hand, arguing as in the proof of \eqref{cota_uniforme_discreta}, there exists $C>0$, independent of $n$, such that 
$$
\max_{j}|\gamma^n_j|\leq C(1+|x^n|),
$$
which, together with \eqref{largo_trayectoria_completa} and the boundeness of $(x^{n})_{n\in \NN}$, yields 
$$
\lim_{n\to\infty}\left| \gamma^n_{N^n_t}-x^n\right|=0 
$$
and hence
$$
\lim_{n\to\infty}g\left(\gamma^n_{N^n_t}, m^n(T)\right)= g(x,m(T)).
$$
The result follows from the previous equation and \eqref{eq:ineq-lim-v-n}.\medskip\\
{\it Proof of {\bf(b)}.}   
To check that \eqref{subsolution_hjb} holds in the viscosity sense, following the lines of \cite[Proposition 4.3]{MR4030259}, let $\phi\in C^1\left([0,T]\times\cR^d\right)$ and $(t^*, x^*)\in (0,T)\times\cR^d$ be such that $v^*-\phi$ has a local maximum in $(t^*, x^*)$. Modifying $\phi$, if necessary, we can assume that  $(t^*, x^*)$  is a  strict maximum for $v^*-\phi$ on $\ov{\mathrm{B}}((t^*, x^*), \delta)$, for some $\delta>0$. Therefore,  by \cite[Chapter V, Lemma 1.6]{MR1484411} there exists a sequence $((k(n), x^n))_{n\in \NN} \subset \I^{n,*}\times\cR^d$ such that $(t^n_{k(n)}, x^n)\to (t^*, x^*)$, $v^n_{k(n)}(x^n)\to v^*(t^*, x^*)$, and 
\be
\label{maximo_local_discreto}
v^{n}_{k}(y)-\phi(t_k^n,y) \leq v^{n}_{k(n)}(x^n)-\phi(t_{k(n)}^n,x^n) \quad \text{for } (k,y)\in \I^n\times \RR^d, \; |t_{k}^n-t_{k(n)}^n|\leq \delta, \; |y-x^n|\leq \delta.
\ee
By \eqref{maximo_local_discreto}, \eqref{cota_v_n_below_above}, {\bf(H2)},  \eqref{cota_uniforme_discreta}, and modifying $\phi$ outside $\ov{\mathrm{B}}((t^*, x^*), \delta)$, if necessary, we can assume that   $\phi\in C^1\left([0,T]\times\cR^d\right)$, $\partial_t \phi$ and $\nabla \phi$ are bounded, and   $(k(n), x^n)$ is a global maximum of $\I^n \times \RR^d \ni (k,x) \mapsto v^n_{k}(x)-\phi(x)\in \RR$. In particular, for all $y\in\cR^d$, we have 
$$
v^n_{k(n)+1}\left( y\right)-v^n_{k(n)}\left(x^n\right)\leq \phi\left(t^n_{k(n)+1}, y\right)-\phi\left(t^n_{k(n)}, x^n\right).
$$
Using \eqref{semidiscrete-scheme-dpp}, we obtain
$$
\ba{lll}
0&=&\min\limits_{\alpha\in\cR^r}\Delta t_n \ell\left(t^n_{k(n)}, \alpha, x^n, m^n(t^n_{k(n)})\right)+v^n_{k(n)+1}\left( x^n+\Delta t_n[A(t^n_{k(n)}, x^n)+B(t^n_{k(n)}, x^n)\alpha]\right)\\[8pt]
\; & \; & -v^n_{k(n)}\left(x^n\right) \\[6pt]
\;&\leq & \inf\limits_{\alpha\in\cR^r} \Delta t_n \ell\left(t^n_{k(n)}, \alpha, x^n, m^n(t^n_{k(n)})\right)+\phi\left(t^n_{k(n)+1}, x^n+\Delta t_n[A(t^n_{k(n)}, x^n)+B(t^n_{k(n)}, x^n)\alpha]\right)\\[8pt]
\; & \; & -\phi\left(t^n_{k(n)}, x^n\right).
\ea
$$
Since $\nabla \phi$ is bounded, {\bf(H1)}(i) implies that the last infimum above is reached at some  $\bar\alpha^n\in\cR^r$  and there exists $C_{\phi}>0$, independent of $n$, such that $|\bar\alpha^n|\leq C_\phi$. Therefore, for any $\alpha\in\cR^r$ we have
$$
\ba{lll}
0&\leq& \ell\left(t^n_{k(n)}, \bar\alpha^n, x^n, m^n(t^n_{k(n)})\right)+\frac{\phi\left(t^n_{k(n)+1}, x^n+\Delta t_n[A(t^n_{k(n)}, x^n)+B(t^n_{k(n)}, x^n)\bar\alpha^n]\right)-\phi\left(t^n_{k(n)}, x^n\right)}{\Delta t_n}\\[6pt]
\;&\leq& \ell\left(t^n_{k(n)}, \alpha, x^n, m^n(t^n_{k(n)})\right)+\frac{\phi\left(t^n_{k(n)+1}, x^n+\Delta t_n[A(t^n_{k(n)}, x^n)+B(t^n_{k(n)}, x^n)\alpha]\right)-\phi\left(t^n_{k(n)}, x^n\right)}{\Delta t_n}.
\ea
$$
Since the sequence $(\bar\alpha^n)_{n\in \NN}$ is bounded    by $C_\phi$, there exists a subsequence, still denoted by $(\bar\alpha^n)_{n\in\NN}$, and $\alpha^*\in \ov{\mathrm{B}}(0, C_\phi)$ such that $\bar\alpha^n\to \alpha^*$ as $n\to\infty$. Passing to the limit in the previous inequality we obtain
\be  \label{eq:subsol}
\ba{lll}
 0&\leq &\ell\left(t^*, \alpha^*, x^*, m(t^*)\right)+\partial_t\phi(t^*, x^*)+ \langle \nabla_{x}\phi(t^*, x^*),  A(t^*, x^*)+B(t^*, x^*)\alpha^*\rangle\\[6pt]
\;&\leq&  \ell\left(t^*, \alpha, x^*, m(t^*)\right)+\partial_t\phi(t^*, x^*)+\langle\nabla_{x}\phi(t^*, x^*), A(t^*, x^*)+B(t^*, x^*)\alpha\rangle  \quad \text{for all } \alpha\in\cR^r,
 \ea
\ee
from which we deduce that 
$$
H\left(t^*, x^*, \nabla_{x}\phi(t^*, x^*), m(t^*)\right)= - \ell\left(t^*, \alpha^*, x^*, m(t^*)\right)-\nabla_{x}\phi(t^*, x^*)\cdot \left[A(t^*, x^*)+B(t^*, x^*)\alpha^*\right].
$$
Finally, by \eqref{eq:subsol}, we obtain
$$
-\partial_t\phi(t^*, x^*)+H\left(t^*, x^*, \nabla_{x}\phi(t^*, x^*), m(t^*)\right)\leq 0,
$$
which proves assertion {\bf (b)}. 
\end{proof}
\section{Appendix II}
\begin{proof}[Proof of Theorem~\ref{unicidad_xi_star}] As in Section~\ref{fully_discrete_hjb}  we assume that $B$ and $A$ are decomposed as in \eqref{B_structure} and \eqref{A_structure}, respectively, with $B_1(t,x)\in \RR^{r\times r}$ being invertible for all $(t,x)\in[0,T]\times\RR^d$. Let $\xi^{1}$ and $\xi^{2}$ be two solutions to Problem~\ref{mfg_problem} and, for $i=1,2$, define $m^i\in C([0,T];\P_1(\RR^d))$ as $m^{i}(t)=e_{t}\sharp\xi^{i}$ for all $t\in[0,T]$. Given $x\in\RR^d$ and $i=1,2$, let us set 
$$
\text{Opt}^{i}(x)=\left\{\gamma \in W^{1,p}([0,T];\RR^d) \, \big| \, \exists \, \alpha\in L^{p}([0,T];\RR^r), \; (\gamma,\alpha) \; \text{solves  $({\color{blue}OC_{x,m^i}})$} \right\}.
$$
Observe that, by Remark~\ref{despues_de_hipotesis}{\rm(ii)}, if $\gamma\in \text{Opt}^{i}(x)$, then there exists a unique control  $\alpha[\gamma]\in L^{p}([0,T];\RR^r)$, given by~\eqref{eq:main-alpha},  such that $(\gamma,\alpha[\gamma])$ solves  $({\color{blue}OC_{x,m^i}})$. Let $\RR^d\ni x\mapsto \gamma^{i,x} \in  \Gamma$ be a Borel measurable selection of the set-valued map $\RR^d\ni x \rightrightarrows \text{Opt}^{i}(x)\in 2^{\Gamma}$. The existence of such a selection follows from exactly the same arguments than those in the proof of ~\cite[Lemma~2.1]{MR4304905}. 
Since $\xi^i$ solves Problem~\ref{mfg_problem}, we have   $\supp(\xi^i)\subseteq \cup_{x\in \supp(m_0)}\text{Opt}^{i}(x)$ and, hence,  condition~\eqref{unicidad_casi_segura} implies that $\xi^i=\gamma^i\sharp m_0$. 

Let us define $J^i:W^{1,p}([0,T];\RR^d)\to \RR$ as
$$
J^i(\gamma)=\int_{0}^{T}\ell(t,\alpha[\gamma](t), \gamma(t), m^i(t))\dd t + g(\gamma(T),m^i(T)) \quad \text{for all } \gamma\in  W^{1,p}([0,T];\RR^d)
$$
an notice that, by {\bf(H5)}, we have 
\be 
\label{monotonicity_difference_J}
\int_{\Gamma} \left(J^1(\gamma)-J^2(\gamma) \right) \dd(\xi^1-\xi^2)(\gamma) \geq 0. 
\ee
If  $\xi^1\neq\xi^2$, then~\eqref{unicidad_casi_segura} implies that 
$$
\int_{\cR^d}J^1(\gamma^{1,x})\dd m_0(x)<\int_{\cR^d}J^1(\gamma^{2,x})\dd m_0(x).
$$
Using that $\xi^i=\gamma^i\sharp m_0$, the previous condition can be rewritten as
\be 
\label{ineq_J_1}
\int_{\Gamma}J^1(\gamma)\dd \xi^1(\gamma)<\int_{\Gamma}J^1(\gamma)\dd \xi^2(\gamma).
\ee
Analogously, we obtain
\be 
\label{ineq_J_2}
\int_{\Gamma}J^2(\gamma)\dd \xi^2(\gamma)<\int_{\Gamma}J^2(\gamma)\dd \xi^1(\gamma).
\ee
Combining \eqref{ineq_J_1} and \eqref{ineq_J_2}, we obtain a contradiction with~\eqref{monotonicity_difference_J}. 
\end{proof}
\bibliographystyle{plain}
\bibliography{hjb,mfg}
\end{document}